\documentclass[sn-mathphys]{sn-jnl}
\jyear{2021}%
\usepackage{MnSymbol}

\theoremstyle{plain}
\newtheorem{theo}{Theorem}[section]
\newtheorem{lem}{Lemma}[section]%
\newtheorem{prop}{Proposition}[section]%
\newtheorem{cor}{Corollary}[section]

\theoremstyle{definition}
\newtheorem{defi}{Definition}%
\newtheorem{example}{Example}%
\newtheorem{ass}{Assumption}

\theoremstyle{remark}
\newcommand{\qedf}{\hfill\filledsquare}
\newtheorem{rems}{Remark}
\newenvironment{rem}
  {\pushQED{$\qedf$}\rems}
  {\popQED\endrems}

\usepackage{hyperref}
\usepackage[nameinlink,capitalize]{cleveref}

\crefname{defi}{Defn.}{Definitions}
\crefname{ass}{Asm.}{Assumptions}
\crefname{theo}{Thm.}{Theorems}
\crefname{lem}{Lem.}{Lemmata}
\crefname{rem}{Rmk.}{Remarks}
\crefname{rems}{Rmk.}{Remarks}
\crefname{cor}{Cor.}{Corollaries}
\crefname{prop}{Prop.}{Propositions}

\usepackage{mathtools}
\usepackage{mathrsfs}

\renewcommand{\R}{\mathbb{R}}

\renewcommand{\epsilon}{\varepsilon}
\newcommand{\N}{\mathbb N}

\newcommand{\eps}{\epsilon}

\renewcommand{\tt}{\tilde t}
\newcommand{\ve}{v_\epsilon}
\newcommand{\xe}{x_\epsilon}

\newcommand{\sL}{{\mathsf{L}\!}}
\newcommand{\sBV}{{\mathsf{BV}}}
\newcommand{\sW}{\mathsf W}
\newcommand{\sC}{\mathsf C}
\newcommand{\sTV}{\mathsf{TV}}

\newcommand{\mL}{\mathcal{L}}

\newcommand{\cC}{\mathcal C}
  \newcommand{\sgn}{\ensuremath{\textnormal{sgn}}}
\renewcommand{\e}{\mathrm e}

\renewcommand{\:}{\mathrel{\coloneqq}}
\newcommand{\equivd}{\ensuremath{\vcentcolon\equiv}}
\newcommand{\OT}{{\Omega_{T}}}
\newcommand{\loc}{\textnormal{loc}}
\newcommand{\sk}{\,;\!\!}

\newcommand{\dd}{\ensuremath{\,\mathrm{d}}}
\DeclareMathOperator{\supp}{supp}
\DeclareMathOperator*{\esssup}{ess-\sup}
\DeclareMathOperator*{\essinf}{ess-\inf}
\newcommand{\cW}{\ensuremath{\mathcal{W}}}

\newcommand{\cZ}{\mathcal Z}
\newcommand{\cX}{\mathcal X}

\newcommand{\uv}{\underline{v}}

\usepackage{amsmath}
\usepackage[utf8]{inputenc}

\usepackage{multicol}
\begin{document}


\title[Discontinuous nonlocal conservation laws and related discontinuous ODEs]{Discontinuous nonlocal conservation laws and related discontinuous ODEs}

\subtitle{Existence, Uniqueness, Stability and Regularity}

\author*[1,2]{\fnm{Alexander} \sur{Keimer}}\email{keimer@berkeley.edu}

\author*[2,3]{\fnm{Lukas} \sur{Pflug}}\email{lukas.pflug@fau.de}


\affil[1]{\orgdiv{Institute of Transportation Studies}, \orgname{UC Berkeley}, \orgaddress{\street{Sutardja Dai Hall}, \city{Berkeley}, \postcode{94720}, \state{California}, \country{USA}}}

\affil[2]{\orgdiv{Competence Unit for Scientific Computing}, \orgname{Friedrich-Alexander University Erlangen-Nuremberg (FAU)}, \orgaddress{\street{Martensstr.~5a}, \city{Erlangen}, \postcode{91058}, \state{Bavaria}, \country{Germany}}}

\affil[3]{\orgdiv{Department of Mathematics, Chair of Applied Mathematics (Continuous Optimization)}, \orgname{Friedrich-Alexander University Erlangen-Nuremberg (FAU)}, \orgaddress{\street{Cauerstr.~11}, \city{Erlangen}, \postcode{91058}, \state{Bavaria}, \country{Germany}}}


\abstract{We study nonlocal conservation laws with a discontinuous flux function of regularity \(\sL^{\infty}(\R)\) in the spatial variable and show existence and uniqueness of weak solutions in \(\sC\big([0,T]\sk \sL^{1}_{\loc}\big)\), as well as related maximum principles. 
We achieve this well-posedness by a proper reformulation in terms of a fixed-point problem. This fixed-point problem itself necessitates the study of existence, uniqueness and stability of a class of discontinuous ordinary differential equations.
On the ODE level, we compare the solution type defined here with the well-known Carath\'eodory and Filippov solutions.}




\maketitle

\section{Introduction}\label{sec:intro}
In this contribution we study nonlocal conservation laws with a discontinuous part of the velocity in space. The discontinuity enters the equation as a multiplicative term and is assumed to be bounded away from zero. The only additional requirement is that it possesses the regularity \(\sL^{\infty}\). Nonlocal refers to the fact that the flux of the conservation law at a given space-time point not only depends on the solution at this point, but also on a spatial averaging around this position by means of a convolution, in equations
\begin{align}
    q_{t}+ \big(v(x)V(\gamma\ast q)\,q\big)_x &=0, \label{eq:NBL}
\end{align}
with \textbf{discontinuous} part of the velocity \(v\in\sL^{\infty}(\R;\R_{>\uv})\) for \(\uv \in \R_{>0}\), \textbf{Lipschitz-continuous} part of the velocity $V \in \sW^{1,\infty}_{\loc}(\R)$ and \textbf{nonlocal weight}  \(\gamma \in \sBV(\R;\R_{\geq 0})\). For details see \cref{defi:disc_conservation_law}.

A variety of results for nonlocal conservation laws have been provided over the last few years \cite{keimer1,coron,ackleh,groeschel,keimer2,amorim,colombo_nonlocal,betancourt,zumbrun,pflug,aggarwal,pflug2,spinola,spinolo,ngoduy,teixeira,colombo2012,pflug3,pflug4,coron_pflug,spinolo_2,sarkar,keimer3,li,colombo_guerra_2,coron2013output,chen2017global,baker1996analytic,piccoli2013transport,difrancesco2019deterministic,bressan2019traffic,colombo2018local,colombo2015nonlocal,colombo2018nonlocal,chalons2018high,friedrich2019maximum,lee2019thresholds,ridder2019traveling,bayen2020modeling,coclite2020general,coclite2021existence,bressan2021entropy,COLOMBO20211653,karafyllis2020analysis,Filippis,crippa2013existence,scialanga,chiarello,blandin2016well,chiarello2019non-local,goatin2019wellposedness,gong2021weak,kloeden}, but only in recent publications \cite{chiarello2021existence,chiarello2021nonlocal} has a discontinuity been considered exactly as denoted in \cref{eq:NBL}.
The authors use Entropy methods together with a type of Godunov discretization scheme and a viscosity approximation to demonstrate well-posedness. They also present a maximum principle for a discontinuity with one jump, where the discontinuity is monotonically chosen so that the solution cannot increase.  This can be envisioned by considering a nonlocal version of the classical LWR model in traffic (\cite{lighthill1955kinematic}) and assuming that traffic flows to the right. Then, if the discontinuous part of the velocity is monotonically increasing, the velocity is faster after each jump, meaning that no increase of density can appear around the discontinuities.

Describing our approach in more general terms, we consider the discontinuous velocity as an \(\sL^{\infty}(\R)\) function that is positive and bounded away from zero, and we then deal with the dynamics introduced in \cref{eq:NBL} and detailed in \cref{defi:disc_conservation_law}. We show that weak solutions exist and are unique \textbf{without} an Entropy condition, and present several maximum principles under which the solution exists semi-globally. In contrast to the Diperna Lions ansatz \cite{diperna}, where the existence of ODEs is shown by studying the corresponding (linear) conservation law, we tackle the problem by formulating the characteristics of the conservation law as a fixed-point problem (as we first described in \cite{pflug} based on the idea proposed in \cite{wang}) and dealing with the corresponding discontinuous ODE. This emerging nonlinear discontinuous ODE, then reads as
\begin{align}
    x'(t) = v(x(t)),\lambda(t,x(t)) \label{eq:IVP}
\end{align}
with \textbf{discontinuous} \(v\in\sL^{\infty}(\R;\R_{>\uv})\) for \(\uv \in \R_{>0}\) and   \(\lambda\in \sL^{\infty}\big((0,T);\sW^{1,\infty}(\R)\big)\) \textbf{Lipschitz-continuous} w.r.t.\ the spatial variable. For details see \cref{defi:disc_ODE}.

For the broad theory on discontinuous ODEs and initial value problems we refer the reader to \cite{filippov1960differential,filippov1988,bressan1998uniqueness,bownds1970uniqueness, coddington1955theory,fjordholm2018sharp,osgood1898beweis,hu1991differential,binding1979differential,agarwal1993uniqueness}. 
As the solution to the discontinuous ODE is later subject to the aforementioned fixed-point problem, we not only show existence and uniqueness of 
solutions, but also stability and continuity results. This is not covered by the established theory of discontinuous ODEs and requires the specific structure of the discontinuous ODE considered here.
Having established these stability estimates and results, demonstrating the existence and uniqueness of discontinuous nonlocal conservation laws on small time horizons is straightforward following the approaches adopted in \cite{pflug}. This is supplemented by an approximation result in a ``weak'' topology, which ultimately enables us to present different types of maximum principles resulting in semi-global well-posedness.
\subsection{Outline}
In \cref{sec:intro}, we introduce the problem and compare our results with those in the literature. We conclude the section with some basic definitions in \cref{subsec:basic_definitions}, which specify what we mean by ``solutions to the introduced problem class.''

\Cref{sec:discontinuous_ODE} is dedicated to the well-posedness and stability properties of solutions to the class of discontinuous ODEs introduced. Having defined what we mean by solutions and stated required assumptions, we then concentrate in \cref{subsec:existence_uniqueness} on the existence and uniqueness of solutions and how they compare to Carath\'eodory and Filippov solutions.
For the existence theory for nonlocal conservation laws in \cref{sec:nonlocal_well_posedness} we require stability of the characteristics with regard to input datum, In \cref{subsec:stability_initial_datum_velocities} we thus consider the stability of the solutions to the discontinuous ODE with respect to initial datum, Lipschitz velocity and discontinuous velocity in a suitable topology.
\Cref{subsec:time_continuity_partial_2} considers the regularity of the derivative of solutions of the discontinuous ODE with respect to the initial datum in the topology induced by \(\sC\big([0,T];\sL^{1}_{\loc}(\R)\big)\), another important ingredient for the well-posedness of the discontinuous nonlocal conservation law studied later. 

In \cref{sec:nonlocal_well_posedness} we finally study the described class of discontinuous nonlocal conservation laws,  beginning by presenting the assumptions on the data involved. In broad terms, for the initial datum and the discontinuous velocity we assume only \(\sL^{\infty}\) regularity. This is identical to the assumption described for the discontinuous ODE in \cref{sec:discontinuous_ODE}. In \cref{subsec:well_posed_nonlocal} we then study the well-posedness of the discontinuous nonlocal conservation law via formulating a fixed-point problem in the Banach space \(\sL^{\infty}((0,T);\sL^{\infty}(\R))\) and using the method of characteristics. We first establish well-posedness of solutions on small time horizons, followed by stability results for the solution with respect to the discontinuous and continuous part of the velocity. We also establish stability for the initial datum in a weak topology, enabling the approximation of solutions by smooth solutions of the corresponding ``smoothed'' nonlocal conservation law. Under relatively mild additional assumptions on the nonlocal kernel and the Lipschitz-continuous velocity, for nonnegative initial datum we show different versions of maximum principles in \cref{subsec:maximum_principle}. One version states that the \(\sL^{\infty}\) norm of the solution can only decrease over time providing the discontinuity is monotonically decreasing, while another only gives uniform upper bounds on the solutions for a general discontinuity. These results also imply the semi-global well-posedness of the solutions. 

We conclude the contribution in \cref{sec:conclusions} with some open problems.
\subsubsection*{Perspective from (local) conservation laws}
From the perspective of approximating local conservation laws by nonlocal conservation laws \cite{coclite2020general,pflug4,bressan2019traffic,bressan2021entropy,COLOMBO20211653,spinolo},
we consider the nonlocal approximations of the following  discontinuous (local) conservation laws:
\begin{align*}
   \ q_{t}+ \big(v(x)\cdot f(q)\big)_x = 0,
\end{align*}
with \(f\equiv V\cdot  \mathrm{Id}\) for \(V\in \sW^{1,\infty}_{\loc}(\R)\) and \(v \in \sL^\infty(\R;\R_{\geq \uv}),\ \uv\in\R_{>0}\). 
Thus, we are dealing with the nonlocal approximation of a multiplicative discontinuous -- in space -- velocity field. However, we will not be studying this limiting behaviour in this work.

Discontinuous conservation laws have been considered in terms of questions of existence and uniqueness, and the need to prescribe the proper Entropy condition at the discontinuity in order to single out the proper (and potentially physical reasonable) solution among the infinite number of weak solutions. A vast number of papers on these topics have been published. For the sake of brevity, we refer the reader to \cite{garavello2007conservation,adimurthi2005optimal,buerger2009engquist,klausen1999stability,klingenberg1995convex,gimse1993conservation,audusse2005uniqueness,buerger2008conservation,ostrov2002solutions,andreianov2010vanishing,adimurthi2004godunov,karlsen2004relaxation,karlsen2004convergence,towers2000convergence,towers2001difference,adimurthi2011existence,adimurthi2007explicit,mishra2007existence,andreianov2011theory} and note that this list is by no means exhaustive.

\subsubsection*{Simplified results covered by the developed theory}
The results obtained can also be applied to special cases of nonlocal conservation law, i.e. nonlocal dynamics with Lipschitz continuous velocity function (setting \(v\equiv 1\))
\begin{align*}
    q_{t}+ \big(V(\gamma\ast q)\,q\big)_x &=0.
\end{align*}
This case (including source terms on the right hand side) has been intensively studied in \cite{pflug} and, indeed, we recover the same results also obtained in \cref{sec:nonlocal_well_posedness}. Thus, the theory proposed here generalizes the results presented in \cite{pflug}.

Discontinous linear conservation laws represent another specific case. Choosing \(V\equiv 1\) we have
\begin{align*}
    q_{t}+ \big(v(x)q\big)_x &=0
\end{align*}
and enriching this with a Lipschitz-continuous (in space) velocity \(\lambda:\OT\rightarrow\) (this is covered by our later analysis on discontinuous ODEs in \cref{sec:discontinuous_ODE}),
for the Cauchy problem we obtain
\begin{align*}
    q_{t}+ \big(v(x)\lambda(t,x) q\big)_x &=0.
\end{align*}
This is supplemented by an initial condition in \(\sL^{\infty}\) that there is a unique weak solution.
Surprisingly, linear conservation laws with discontinuous velocities have not been considered intensively. We refer the reader to \cite{crippa2011lagrangian}, where the author studies
\[
\rho_{t}+ \Big(f(t,x)\rho\Big)_{x}=0
\]
with 
\begin{itemize}
\item \(f\) continuous and nonnegative,
\item \(f\) of such a form that the solutions to the corresponding ODEs do not blow up in finite time (for instance assuming that \(f\) can grow at most linearly with regard to the spatial variable),
\item  the sets of points where \(f\) is zero are somewhat ``nice'' (see \cite[(A1)-(A3), p. 3138]{crippa2011lagrangian}).
\end{itemize}
 
 However, this setup differs from our considered class of equations as we allow \(\sL^{\infty}\) regularity and have no sign restrictions for the Lipschitz-part.
\cite{besson2007solutions} obtains results for velocities of regularity \(\sL^{\infty}\) with the additional assumption that \(\operatorname{div}(f)\in \sL^{\infty}\). The second assumption is weak for multi-D equations as considered in that publication. However, in the scalar case this assumption boils down to a Lipschitz-continuous velocity field, so the presented result can be seen as a generalization in the \(1\text{D}\) case.

The multi-D case is also considered in \cite{ambrosio2004transport,ambrosio2014continuity} where the velocity field is assumed to be in \(\sBV\) or admits other Sobolev regularity.

For the characteristics, \cite{petrova1999linear} uses Filippov solutions \cite{filippov1988} (see \cite[Eq. 2.5 and Eq. 2.6)]{petrova1999linear})
and considers the transport equation (not the conservation law) with a one-sided Lipschitz-continuous velocity processing unique solutions when assuming a continuous initial datum due to the uniqueness of backward characteristics in the sense of Filippov \cite{filippov1960differential}. Similar results are obtained in \cite{bouchut1998one}.
Finally, \cite{CLOP201945} considers again the multi-D case and states conditions on the vector field for existence and uniqueness of solutions. Thereby, the vector field is assumed to be continuous and for existence and uniqueness a ``weakened'' Lipschitz-condition based on the modulus of continuity is required. Solutions are thought of in the space of signed Borel measures.

\subsection{Basic definitions}\label{subsec:basic_definitions}
In this section, we rigorously state the problems that we will tackle. Starting with the discontinuous IVP, the problem reads as:
\begin{defi}[Discontinuous IVP]\label{defi:disc_ODE}
Let \(T\in\R_{>0}\) and \(\underline v \in\R_{>0}\) be given. For a \textbf{discontinuous} \(v\in \sL^\infty(\R\sk \R_{\geq \underline v })\), a \textbf{smooth} $\lambda \in \sL^{\infty}\big((0,T)\sk \sW^{1,\infty}(\R)\big)$ and $x_0 \in \R$, we consider the following discontinuous IVP
\begin{equation}
\begin{aligned}
    x'(t) &= v(x(t))\lambda(t,x(t)), && t\in[0,T]\\
    x(0) &= x_0.
    \end{aligned}
    \label{eq:discontinuous_ODE}
\end{equation}
Thereby,  $v$ represents the \textbf{discontinuous part of the velocity}, $\lambda$ the \textbf{Lipschitz continuous part of the velocity} and $x_0$ the \textbf{initial value}.
\end{defi}
As outlined above, the existence, uniqueness and regularity of solutions to the discontinuous IVP are strongly related to the existence and uniqueness of solutions to the following nonlocal conservation law:

\begin{defi}[The discontinuous (in space) nonlocal conservation law]\label{defi:disc_conservation_law}
Let \(T\in \R_{>0}\) be given and \(\OT \: (0,T)\times \R\). For \(q:\Omega_{T}\rightarrow\R\), \textbf{initial datum} \(q_{0}\in \sL^{\infty}(\R)\), \textbf{Lipschitz-continuous velocity} \(V\in \sW^{1,\infty}_{\loc}(\R)\), \textbf{discontinuous part of the velocity} \(v\in \sL^\infty(\R\sk \R_{\geq \underline v })\) with $\underline v \in \R_{>0}$ and \textbf{nonlocal weight} \(\gamma\in \sBV(\R\sk \R_{>0})\), we call the following Cauchy problem
\begin{align*}
    q_{t}(t,x)+ \partial_{x}\Big(v(x)V\big(\big(\gamma\ast q(t,\cdot)\big)(x)\big)q(t,x)\Big)&=0 && (t,x)\in\OT\\
    q(0,x)&=q_{0}(x)&& x\in\R
\end{align*}
    a \textit{discontinuous nonlocal conservation law}. 
\end{defi}
The stated results can naturally be extended, as outlined in the following \cref{rem:generalization}.
\begin{rem}[Generalizations -- Extensions]\label{rem:generalization}
For the sake of a type of ``completeness or generality'' of the developed theory, we mention that the results established in this work can be extended to general nonlocal terms and explicitly space- and time-dependent velocity functions, as well as balance laws, i.e.,\ it is also possible to obtain the well-posedness of the more general \textbf{discontinuous nonlocal balance law} in \((t,x)\in\OT\)
\begin{align*}
     q_{t}(t,x)+ \partial_{x}\Big(v(x)\tilde{V}\big(t,x,\cW[q,\tilde{\gamma}](t,x)\big)q(t,x)\Big)&=h\big(t,x,q(t,x),\cW[q,\tilde{\gamma}](t,x)\big)\\
    q(0,x)&=q_{0}(x)\\
    \cW[q,\tilde{\gamma}]&\:\int_{\R}\tilde{\gamma}(t,x,y)q(t,y)\dd x,
\end{align*}
with \(\tilde{V}:[0,T]\times\R^{2}\rightarrow\R\) also Lipschitz in the explicit spatial variable, \(\tilde{\gamma}:[0,T]\times\R^{2}\rightarrow\R_{\geq0}\) Lipschitz in the second component and \(\sTV\) in the third component (compare \cite{pflug,spinola,bayen2020modeling}), and \(h:\OT\times\R^{2}\rightarrow\R\) Lipschitz in the third and fourth component and of corresponding regularity in \((t,x)\). We do not go into details here. For smooth kernels, it is even possible to extend results to measure-valued solutions (for measure-valued initial datum) similarly to \cite{crippa2013existence,gong2021weak}.
\end{rem}
For both problem classes \cref{defi:disc_ODE,defi:disc_conservation_law}, we will present proper definitions of solutions in \cref{defi:weak,defi:weak_solution} and demonstrate the existence and uniqueness in \cref{theo:surrogate,theo:existence_uniqueness_nonlocal_conservation_law}. 
It is worth underlining once more that in particular for the discontinuous nonlocal conservation law \textbf{no} Entropy condition is required to obtain uniqueness of weak solutions. This has already been proven in \cite{pflug} for Lipschitz-continuous velocities.

\section{Existence, uniqueness and stability of the discontinuous IVP}\label{sec:discontinuous_ODE}
In this section, we study the existence, uniqueness and stability (with regard to all input parameters and functions) of the discontinuous ODE introduced in \cref{defi:disc_ODE}. Let us first recall the assumptions on the involved datum in the following \cref{ass:input_datum}.
\begin{ass}[Involved datum]\label{ass:input_datum}
For a \(T\in\R_{>0}\) denoting the considered time horizon, we assume
\begin{description}
    \item[Discontinuous part:] \(v\in \sL^{\infty}\big(\R\sk \R_{\geq \underline v }\big)\) with \(\underline v \in \R_{>0}\),
    \item[Lipschitz-continuous part:] \(\lambda \in \sL^{\infty}\big((0,T)\sk \sW^{1,\infty}(\R)\big).\)
\end{description}
\end{ass}
For the  considered class of discontinuous initial value problems in \cref{defi:disc_ODE}, we must first define what we mean by a solution. This becomes clear when recalling that \(v\circ x\) is not necessarily measurable for \(x\in \sW^{1,\infty}((0,T))\) since \(x\) could be locally constant and, as a \(\sL^\infty\) function, \(v\)does not possess significantly ``good representatives'' with respect to the Lebesgue measure. However, due to the positive lower bound on $v$ and its time-independence, we can divide the strong form of solution by $v$ and by integration obtain the following integral definition of a solution:
\begin{defi}[Solutions for \cref{defi:disc_ODE}]\label{defi:weak}
For \(x_{0}\in\R\) and the data as in \cref{ass:input_datum}, a solution to the \textbf{discontinuous IVP} in \cref{defi:disc_ODE} is defined as a function \(x\in \sC([0,T])\) such that
    \begin{align}
        \int_{x_0}^{x(t)} \tfrac{1}{v(y)}\dd y &= \int_{0}^t \lambda(s,x(s)) \dd s, && \forall t\in [0,T]. \label{eqn:weak}
    \end{align}
    A solution is denoted by \(\cX[v,\lambda](x_0\sk \cdot)\), with \(x_{0}\) indicating the considered \textbf{initial datum} at time \(t=0\), \(v\) the \textbf{discontinuous} part of the velocity and \(\lambda\) the \textbf{Lipschitz-continuous} part.
\end{defi}
\begin{rem}[Reasonability of \cref{defi:weak}]
The definition of solutions in \cref{eqn:weak} is more usable then the ``classical'' Carath\'eodory introduced later. It enables the existence -- and later also stability properties -- to be tackled without prescribing additional regularity assumptions on \(x\) (such as measurability of \(v(x(\cdot))\lambda(\cdot,x(\cdot))\)). Compare in particular with
\cref{defi:cara}.

The introduced notation \(\cX[v,\lambda](x_0\sk \cdot)\) is later justified in \cref{subsec:existence_uniqueness}, where we prove existence and uniqueness of solutions.
\end{rem}

\subsection{Existence/Uniqueness of solutions and their relation to ``classical'' Carath\'eodory and Filippov solutions}\label{subsec:existence_uniqueness}
In the following \cref{theo:surrogate}, we prove the existence and uniqueness of solutions by decomposing the problem into two problems that possess "nicer" properties and can be studied separately:
\begin{theo}[Existence and uniqueness of solution in \cref{defi:weak}]\label{theo:surrogate}
Let \(T\in\R_{>0}\) be given and \cref{ass:input_datum} hold. Then, in the sense of \cref{defi:weak} there exists a unique solution \[\cX[v,\lambda](x_{0};\cdot)\in\sW^{1,\infty}((0,T)).\] 
In addition, defining the following surrogate expression
\begin{align}
     \cZ[v](x_0\sk \ast) &\equivd \int_{x_0}^{\ast}\tfrac{1}{v(s)}\dd s && \text{on } \R,\label{eq:surrogate_system1}
\end{align}
it holds that \(\R\ni x\mapsto \cZ[v](x_{0}\sk x)\) is invertible \(\forall (x_{0},v)\in \R\times\sL^{\infty}\big(\R\sk \R_{\geq \uv}\big)\) and for the inverse we write \(\cZ[v]^{-1}(x_{0}\sk\cdot):\R\rightarrow\R\).

Finally, calling $\cC[\lambda,\cZ[v](x_0\sk \cdot)]$ the solution $c$ of the integral equation
\begin{align}
   c(t) = \int_0^t \lambda\big(s,\cZ[v]^{-1}(x_0\sk c(s))\big) \dd s,   && \forall t \in [0,T],\label{eq:surrogate_system2}
\end{align}
the identity ``$\cX \equiv \cZ^{-1} \circ \cC$'' holds -- in full notation --
\begin{equation}
\cX[v,\lambda](x_0\sk\cdot) \equiv \cZ[v]^{-1}\big(x_0\sk \cC[\lambda,\cZ[v](x_0\sk \ast)](\cdot)\big) \qquad \text{ on } [0,T].\label{eq:identity_surrogate_system}
\end{equation}
\end{theo}
\begin{proof}

\Cref{eq:surrogate_system1} is well defined and by construction $\cZ[v](x_{0}\sk \cdot)\in \sW^{1,\infty}_{\loc}(\R)$ so that
\begin{equation}
\tfrac{1}{\|v\|_{\sL^{\infty}(\R)}}\leq \partial_{x}\cZ[v](x_{0}\sk x)\leq \tfrac{1}{\underline{v}}\qquad \forall x\in\R\label{eq:lower_upper_bound_cZ}
\end{equation}
and thus \(\partial_{3}\cZ[v](x_{0}\sk\cdot)\in \sL^{\infty}(\R)\).
Additionally, as \(x\mapsto \cZ[v](x_{0}\sk x)\) is strictly monotone, the inverse mapping $\cZ[v]^{-1}(x_{0}\sk \cdot)$ is well defined and, thanks to \cref{eq:lower_upper_bound_cZ},
\[
\cZ[v]^{-1}(x_{0}\sk \cdot) \in \sW^{1,\infty}_{\loc}(\R):\quad \partial_{3}\cZ[v]^{-1}(x_{0}\sk \cdot)\in \sL^{\infty}(\R).
\]
Next, considering the definition of \(\cC[\lambda,\cZ[v](x_{0}\sk \cdot)]\) in \cref{eq:surrogate_system2} and the fact that \(\lambda \in \sL^{\infty}\big((0,T)\sk \sW^{1,\infty}(\R)\big)\), the composition \(x\mapsto \lambda\big(s,\cZ[v]^{-1}(x_{0}\sk x)\big)\), is -- thanks to the previous estimates -- globally Lipschitz-continuous for each \((s,x_{0})\in (0,T)\times\R\) and thus, there exists a unique Carath\'eodory solution (see for instance \cite{coddington}) \(\cC[\lambda,\cZ[v](x_{0}\sk \cdot)]\in\sW^{1,\infty}((0,T))\).

We now need to check whether the \(\cX[v,\lambda](x_{0}\sk \cdot)\in \sW^{1,\infty}((0,T))\) as in \cref{eq:identity_surrogate_system} indeed satisfies \cref{defi:weak}.
We have  by the very definition of \(\cX[v,\lambda](x_{0}\sk \cdot)\) \(\forall t\in[0,T]\)
\begin{align*}
    \cC[\lambda,\cZ[v](x_{0}\sk \cdot)](t)=\int_{0}^{t}\lambda\big(s,\cZ[v]^{-1}(x_{0}\sk c(s)) \big)\dd s=\int_{0}^{t}\lambda\big(s,\cX[v,\lambda](x_{0}\sk s)\big)\dd s
\end{align*}
and, as \(\cZ[v](x_{0}\sk \cdot) \circ \cX[v,\lambda](x_{0}\sk\cdot) \equiv \cC[\lambda,\cZ[v](x_{0}\sk \cdot)]\), by assumption
\[
\cZ[v]\big(x_{0}\sk \cX[v,\lambda](x_{0};t)\big)=\int_{x_{0}}^{\cX[v,\lambda](x_{0}\sk t)}\tfrac{1}{v(s)}\dd s,
\]
which is the definition of a solution in \cref{defi:weak}. This demonstrates the existence of solutions.
For the uniqueness, assume that we have two solutions \(\cX,\tilde{\cX}\in \sC([0,T])\) satisfying \cref{defi:weak}. Then, the difference satisfies
\begin{align*}
\int_{\tilde{\cX}(t)}^{\cX(t)}\tfrac{1}{v(z)}\dd z&=\int_{0}^{t}\lambda\big(s,\cX(s)\big)-\lambda\big(s,\tilde{\cX}(s)\big)\dd s\quad \forall t\in[0,T]
\end{align*}
and, as \(\lambda\in \sL^{\infty}\big((0,T)\sk \sW^{1,\infty}(\R)\big)\), we obtain
\begin{align*}
\tfrac{\vert \cX(t) - \tilde{\cX}(t)\vert }{\|v\|_{\sL^{\infty}((0,T))}}
\leq \bigg\vert \int_{\tilde{\cX}(t)}^{\cX(t)}\tfrac{1}{v(z)}\dd z\bigg\vert &=\bigg\vert \int_{0}^{t}\lambda\big(s,\cX(s)\big)-\lambda\big(s,\tilde{\cX}(s)\big)\dd s \bigg\vert  \\
&\leq \|\partial_2 \lambda\|_{\sL^\infty((0,T)\sk \sL^\infty(\R))} \int_0^t \vert \cX(s) - \tilde{\cX}(s)\vert  \dd s.
\end{align*}
Applying Gr\"onwall's inequality \cite[Chapter I, III Gronwall's inequality]{walter} yields
\[
\vert x(s)-\tilde{x}(s)\vert =0,\quad \forall s\in[0,t],
\]
thus the two solutions must be identical. This concludes the proof.
\end{proof}
In the following, we show that the unique solution of the discontinuous IVP in \cref{defi:disc_ODE} in the sense of \cref{defi:weak} is also a ``classical'' Carath\'eodory solution in the following sense:
\begin{defi}[Carath\'eodory solutions for \cref{defi:disc_ODE}]\label{defi:cara}
Let \cref{ass:input_datum} hold. Then, for the initial datum \(x_{0}\in\R\) we call a function \(\cX\in \sC([0,T])\) a \textbf{Carath\'eodory solution} for \cref{defi:disc_ODE} iff $t\mapsto v\big(\cX(t)\big)\lambda\big(t,\cX(t)\big)$ is Lebesque measurable and
\begin{align}
    \cX(t)&=x_{0}+\int_{0}^t v\big(\cX(s)\big)\lambda\big(s,\cX(s)\big)\dd s, \qquad \forall t \in [0,T].\label{eq:cara_ODE}
    \end{align}
\end{defi}
Notice the difference to the usual definition of a Carath\'eodory solution, where the measurability of the integrand is given by construction (either by being continuous or having a right hand side which is strictly bounded away from zero so that solutions are strictly monotone).

\begin{lem}[Equivalence Carath\'eodory solution and solutions as in \cref{defi:weak}]\label{theo:equivalence_cara_weak}
There exists a unique Carath\'eodory solution as in \cref{defi:cara} iff there exists a unique  solution as in \cref{defi:weak}.
\end{lem}

\begin{proof}
We start by showing that solutions in the sense of \cref{defi:weak} are also Carath\'eodory solutions as in \cref{defi:cara}.
Recalling the steps of the proof in \cref{theo:surrogate}, we find that
\begin{align}
        \cZ[v]^{-1}(x_{0}\sk s) &= x_0 + \int_{0}^s v\Big(\cZ[v]^{-1}(x_{0}\sk u)\Big) \dd u\quad \forall s\in\R.\label{eq:cZ_inverse}
\end{align}
This identity is a direct consequence of the following manipulation for \(s\in\R\) (for the sake of briefer notation we write \(\cZ,\cZ^{-1}\) and suppress the dependencies on \(v\) and \(x_{0}\)):
\begin{align}
    \cZ^{-1}(\cZ(s)) &= x_0 + \int_{0}^{\int_{x_0}^{s}\frac{1}{v(x)}\dd x} v\big(\cZ^{-1}(u)\big) \dd u.
    \intertext{Substituting \(\cZ^{-1}(u)=w\) (which is possible according to \cref{eq:lower_upper_bound_cZ})},
    &= x_0 + \int_{x_0}^{s} v(w)\cZ'(w) \dd v  = x_0 + \int_{x_0}^{s} \dd v = s.
\end{align}
As this holds for all \(s\in[0,T]\), we have proven that \(\cZ[v]^{-1}\) indeed satisfies \cref{eq:cZ_inverse}.

To show that \(\cX\), which is constructed via \cref{theo:surrogate}, satisfies \cref{eq:cara_ODE}, we next apply $\cZ^{-1}$ to \cref{eqn:weak} and obtain
\begin{align}
   \cZ^{-1}\bigg(\int_0^t \lambda\big(s,\cX(s)\big) \dd s\bigg) &= x_0 + \int_0^{\int_0^t \lambda\big(s,\cX(s)\big) \dd s} v\big(\cZ^{-1}(u)\big) \dd u\\
   &= x_0 + \int_{\cC(0)}^{\cC(t)} v\big(\cZ^{-1}(u)\big) \dd u.
    \intertext{Substituting $u = \cC(\tau)$ for \(\tau\in\R\) chosen accordingly},
     &= x_0 + \int_0^{t} v\big(\cZ^{-1}(\cC(\tau))\big)\cC'(\tau) \dd \tau \\
     &= x_0  + \int_0^{t} v(\cX(\tau))\lambda(\tau,\cX(\tau)) \dd\tau. 
\end{align}
As the left hand side of \cref{eqn:weak} is given by $\cZ \circ \cX$ in terms of $\cZ$ (as defined in \cref{eq:surrogate_system1}), by applying $\cZ^{-1}$ we obtain
\begin{align}
    \cZ^{-1}\big(\cZ(\cX(t))\big) = \cX(t),\quad  \forall t\in[0,T].
\end{align}
Thus, if $\cX$ is a solution in the sense of \cref{defi:weak}, it is also a Carath\'eodory solution. For showing the equivalence we mention that all of the previous manipulations are equivalent transforms, and we can start with the identity \cref{eq:cara_ODE} and go backwards in the presented proof.
\end{proof}
Finally, we want to close the gap to Filippov solutions \cite{filippov1960differential,filippov1988}. We first define what we mean by Filippov solutions, sticking with \cite[2: Definition of the solution]{filippov1988}:
\begin{defi}[Filippov solution for a differential equation]\label{defi:Filippov}
We call a function \(\cX\in\sW^{1,1}((0,T))\) for \(T\in\R_{>0}\) a \textbf{Filippov solution} of the \textbf{discontinuous initial value problem} in \cref{defi:disc_ODE}
iff
\begin{align*}
\cX(0)&=x_{0}\\
\cX'(t)&\in K[f(t,\cdot)](\cX(t)),\qquad \forall t\in[0,T]\text{ a.e.}
\end{align*}
with \(f\equiv v\cdot \lambda\), \(K\) being defined as
\begin{align}
    K[f(t,\cdot)] \: \bigcap_{\delta\in\R_{>0}}\bigcap_{\substack{N\subset\R:\\ \mu(N)=0}}\operatorname{conv}f\Big(t,B_{\delta}(\cX(t))\setminus N\Big),\label{eq:defi_K}
\end{align}
where \(\mu\) denotes the Lebesgue measure on \(\R\).
\end{defi}
Given this definition, we have the existence of Filippov solutions according to the following
\begin{theo}[Existence of solutions]
Given \cref{ass:input_datum}, there exists a Filippov solution as in \cref{defi:Filippov}.
\end{theo}
\begin{proof}
This is a direct consequence of \cite[Theorem 4, Theorem]{filippov1988}, recalling that \(v\cdot\lambda\) is essentially measurable and essentially bounded and thus satisfies \textbf{Condition B} as well as the boundedness of any solutions on any finite time horizon.
\end{proof}
In the next \cref{lem:relation_Filippov}, we make a connection between the solutions in \cref{defi:weak} and general Filippov solutions.
\begin{lem}[Relation of \cref{defi:weak} to Filippov solutions]\label{lem:relation_Filippov}
Solutions in the sense of \cref{defi:weak} are Filippov solutions as defined in \cref{defi:Filippov}.
\end{lem}

\begin{proof}
According to \cref{theo:surrogate}, the solution \(x \in \sC([0,T])\) to \cref{defi:weak} exists and it is both unique and Lipschitz. Thus, \(x\in \sW^{1,1}((0,T))\). 
As \cref{defi:weak} is invariant with regard to the choice of representative of \(f\) in the Lebesgue measure, we can choose \(\tilde f \) as follows (\(\eps \in \R_{>0}\)):
\begin{align*}
    \tilde f(t,\cdot) \equiv \lim_{\eps \rightarrow 0}\tfrac{1}{2\eps} \int_{\cdot - \eps}^{\cdot + \eps} f(t,y) \dd y
\end{align*}
for \(t\in [0,T]\) a.e..
According to the Lebesgue differentiation theorem \cite[3.21 The Lebesgue Differentiation Theorem]{folland}, we have \(f\equiv \tilde f\) a.e.. As it also holds that
\begin{align}
\bigcap_{\substack{N\subset\R:\\ \mu(N)=0}}\operatorname{conv}\tilde{f}\Big(t,B_{\eps}(x(t))\setminus N\Big) = \bigg[\essinf_{y \in B_\eps(x)} \tilde{f}(t,y) \, , \, \esssup_{y \in B_\eps(x)} \tilde{f}(t,y) \bigg],
\end{align}
we can estimate uniformly in \(\eps\in\R_{>0}\) for \(x\in\R\)
\begin{align}
 \tfrac{1}{2\eps} \int_{x-\eps}^{x+\eps} \tilde{f}(t,y)  \dd y \in \bigg[\essinf_{y \in B_\eps(x)} \tilde{f}(t,y) \, , \, \esssup_{y \in B_\eps(x)} \tilde{f}(t,y) \bigg] 
\end{align}
and obtain 
\begin{align*}
\cX'(t) \in K\big[\tilde{f}(t,\cdot)\big](\cX(t)) \qquad \forall t\in [0,T] \text{ a.e.},
\end{align*}
with \(K\) as in \cref{eq:defi_K}. This is the definition of a Filippov solution, concluding the proof.
\end{proof}

However, uniqueness results for Filippov solutions are only presented for right hand sides of specific structure (\cite[Theorem 10]{filippov1988} (too strong for our setup)) or for autonomous right hand sides \cite{fjordholm2018sharp} where the famous \textbf{Osgood condition} plays a crucial role \cite{osgood1898beweis}. We detail this in the following for our setup, but need to restrict ourselves to the fully autonomous case (as this is where \cite{fjordholm2018sharp} is applicable). We thus assume that the Lipschitz part of the velocity, i.e.\ \( \lambda\), does not explicitly depend on time.
\begin{theo}[Uniqueness of Filippov solutions for scalar autonomous discontinuous ODEs as presented in \cite{fjordholm2018sharp}]\label{theo:uniqueness_Filippov}
Let \cref{ass:input_datum} hold. Moreover, let \(\exists \tilde{\lambda}\in \sW^{1,\infty}(\R):\ \lambda(t,\cdot)\equiv \tilde{\lambda}\) on \(\R\) (i.e., the discontinuous IVP is autonomous) and let
\begin{align*}
    \Big\{ x \in \R \ : \ \tilde \lambda(x) \neq 0 \text{ and } v \text{ discontinuous at } x\Big\}
\end{align*}
have Lebesgue measure zero. Then, the Filippov solution to the discontinuous IVP in \cref{defi:disc_ODE} is unique.
\end{theo}
\begin{proof}
We take advantage of the result in \cite{fjordholm2018sharp} and repeat what is stated there: Recall the definition of \(K\) in \cref{eq:defi_K} and assume that
\begin{itemize}
    \item the set \[\Big\{x\in\R:\ 0\notin K[v\cdot \tilde\lambda](x)  \text{ and } v\cdot \tilde \lambda \text{ discontinuous at } x\Big\}\subset\R\] has \textbf{Lebesgue measure} zero,
    \item for every \(x\in\R\) with \(0 \in K[v\cdot \tilde\lambda](x)\), the function 
    \begin{equation}
    g[v\cdot\tilde \lambda]:\begin{cases}
    \R&  \mapsto\R\\
    z&\mapsto \lim_{\delta\rightarrow 0}\esssup_{y\in B_{\delta}(x)}\big((v\cdot\tilde{\lambda})(y+z)\sgn(z)\big)^{+}
    \end{cases}\label{eq:defi_g}
    \end{equation}
    is an \textbf{Osgood} function, i.e.\ \(g[v\cdot\tilde \lambda]\) is non-negative, Borel measurable, and for a \(\delta\in\R_{>0}\) satisfies
    \begin{equation}
    \int_{-\tilde{\delta}}^{0}\tfrac{1}{g[v\cdot\tilde \lambda](u)}\dd u=\infty=\int_{0}^{\tilde{\delta}}\tfrac{1}{g[v\cdot\tilde \lambda](u)}\dd u \qquad \forall \tilde{\delta}\in(0,\delta].\label{eq:defi_Osgood}
    \end{equation}
    Then there exists a \textbf{unique} Filippov solution to the considered discontinuous IVP.
\end{itemize}
However, both conditions are satisfied as we will detail in the following:
\begin{itemize}
    \item The first point is satisfied by construction as \(f\equiv v\cdot\tilde{\lambda}\), \(\tilde{\lambda}\) is Lipschitz and \(v\geqq\uv\).
    \item For the second point, we first recall that the definition of solution does not vary with respect to the representative (here \(f\)) in the Lebesgue-measure. Instead of \(f\equiv v\cdot\tilde{\lambda}\), we can therefore choose a Borel measurable function \(\hat{v}\in\sL^{\infty}(\R;\R_{\uv})\), with
    \[
    v(x)=\hat{v}(x),\qquad x\in\R \text{ a.e}.
    \]
    (For this, use Lusin's theorem \cite[44. Lusin's Theorem, p.64]{folland} to approximate \(v\) by a continuous function up to a set of arbitrarily small Lebesgue measure).
    Then, the corresponding function \(g[\hat{v}\cdot\tilde{\lambda}]\) defined in \cref{eq:defi_g} is by construction non-negative and Borel measurable. We now prove the so-called Osgood condition in \cref{eq:defi_Osgood}. To accomplish this, we estimate as follows for \(\tilde{\delta}\in(0,\delta)\) and \(z\in\R\):
    \begin{align*}
    \esssup_{y\in B_{\delta}(x)}\big((v\cdot\tilde{\lambda})(y+z)\sgn(z)\big)^{+}&\leq \|\hat{v}\|_{\sL^{\infty}(\R)}\esssup_{y\in B_{\delta}(x)}\bigg\vert\tilde{\lambda}(x)+\int_{x}^{y+z}\tilde{\lambda}'(s)\dd s\bigg\vert.
    \intertext{\(\lambda(x)=0\) as we are in the case \(0\in K[v\cdot\tilde{\lambda}(x)]\) and \(\hat{v}\geqq\uv\)}\\
    &\leq \|\hat{v}\|_{\sL^{\infty}(\R)}\vert z+\delta\vert\|\tilde{\lambda}'\|_{\sL^{\infty}(\R)}.
    \end{align*}
    Letting \(\delta\rightarrow 0\), for \(z\in\R\) we obtain
    \begin{align*}
        g[\hat{v}\cdot\tilde{\lambda}](z)\leq \|\hat{v}\|_{\sL^{\infty}(\R)}\|\tilde{\lambda}'\|_{\sL^{\infty}(\R)}\vert z\vert,
        \end{align*}
from which \cref{eq:defi_Osgood} follows. This concludes the proof.
\end{itemize}
\end{proof}
The previous \cref{theo:uniqueness_Filippov} has made a connection between our discontinuous IVP and Filippov theory, and has established the necessary uniqueness for the fully autonomous IVP. However, it does not directly apply to the general non-autonomous case. More importantly, although continuous dependency of the solution with regard to the input datum might be obtained, we require rather strong stability or continuity results, which can be obtained with our definition of solution \cref{defi:weak} by taking advantage of the surrogate system in \cref{theo:surrogate}. This is detailed in the next section.
\subsection{Stability of the solutions with respect to initial datum and velocities}\label{subsec:stability_initial_datum_velocities}
In this section, we deal with the stability of the discontinuous IVP introduced in \cref{defi:disc_ODE}. To this end, we use the surrogate system introduced in \cref{theo:surrogate} and study its components \cref{eq:surrogate_system1,eq:surrogate_system2} in detail.
\begin{prop}[Auxiliary stability results]\label{theo:existence_uniqueness_stability_surrogate}
Given \cref{ass:input_datum} and \(T\in\R_{>0}\), the surrogate ODEs defined in \cref{eq:surrogate_system1} and \cref{eq:surrogate_system2} are stable with respect to the initial datum, smooth velocity $\lambda$ and discontinuous velocity \(v\), i.e., the following stability results hold: For \(\cZ\)
\begin{equation}
\begin{split}
&\forall (u,x_{0},\tilde{x}_{0},v,\tilde{v})\in\R^{3}\times\sL^{\infty}\big(\R\sk \R_{\geq \uv}\big)^{2}:\\
     &\big\vert \cZ[v_{0}](x_0\sk u)-\cZ[\tilde{v}_{0}](\tilde{x}_{0}\sk u)\big\vert \\
     &\qquad\leq\tfrac{1}{\underline{v}} \vert x_0 - \tilde x_0\vert  + \tfrac{1}{\uv^2}\|v-\tilde v\|_{\sL^1((\min\{x_0,\tilde x_0,u\},\max\{x_0,\tilde x_0,u\}))}
\end{split}
     \label{eq:stability_z}
\end{equation}
and for \(\cC\) and \(\forall t\in[0,T]\)
\begin{equation}
\begin{split}    
&\forall (\lambda,\tilde{\lambda},x_{0},\tilde{x}_{0},v,\tilde{v})\in \sL^{\infty}\big((0,T)\sk \sW^{1,\infty}(\R)\big)^{2}\times\R^{2}\!\!\times\sL^{\infty}\big(\R\sk \R_{\geq \uv}\big)^{2}\\
&\big\vert \cC\big[\lambda,\cZ[v_{0}](x_0\sk \cdot)\big](t) - \cC\big[\tilde \lambda,\cZ[\tilde{v}_{0}](\tilde{x}_{0}\sk \cdot)\big](t)\big\vert \\
&\leq \e^{t\|v\|_{\sL^{\infty}(\R)}\mathcal L_2}\int_{0}^{t}\big\|\lambda(s,\cdot)-\tilde \lambda(s,\cdot)\big\|_{\sL^{\infty}(X(x_0,v,\lambda))}\dd s\\
&\qquad +\e^{t\|v\|_{\sL^{\infty}(\R)}\mathcal L_2}t\|\tilde{v}\|_{\sL^{\infty}(\R)}\mathcal L_2\|\cZ[v](x_0\sk \cdot)-\cZ[\tilde{v}](\tilde{x}_{0}\sk \cdot)\|_{\sL^{\infty}(X(x_0,v,\lambda))}
\end{split}
\label{eq:stability_c}
\end{equation}
   with
   \begin{equation}
   \begin{split}
   \mL&\:\max\big\{\|\lambda\|_{\sL^{\infty}((0,T);\sL^{\infty}(\R))},\|\tilde{\lambda}\|_{\sL^{\infty}((0,T);\sL^{\infty}(\R))}\big\}\\
   \mL_{2}&\:\max\big\{\|\partial_{2}\lambda\|_{\sL^{\infty}((0,T)\sk \sL^{\infty}(\R))},\|\partial_{2}\tilde{\lambda}\|_{\sL^{\infty}((0,T)\sk \sL^{\infty}(\R))}\big\} \\
   X(x_0,v,\lambda) &\: x_{0}+ T\|v\|_{\sL^\infty(\R)}\|\lambda\|_{\sL^{\infty}((0,T)\sk \sL^{\infty}(\R))}\cdot(-1,1) \subset \R.
   \end{split}
   \label{eq:X_x_0}
\end{equation}
\end{prop}
\begin{proof} We start by proving \cref{eq:stability_z}. To achieve this, recall the definition of \(\cZ\) in \cref{eq:surrogate_system1}. From this definition, we can make the following estimate for \(u\in\R\):
\begin{align}
    \big\vert \cZ[v_{0}](x_{0}\sk u)-\cZ[\tilde{v}_{0}](\tilde{x}_{0}\sk u)\big\vert 
    &= \bigg\vert \int_{x_0}^{u}\!\!\!\tfrac{1}{v(s)}  \dd s - \int_{\tilde x_0}^u \tfrac{1}{\tilde v(s)}  \dd s\bigg\vert\notag \\
    &\leq \bigg\vert \int_{x_0}^{\tilde x_0}\!\!\!\tfrac{1}{v(s)}  \dd s\bigg\vert  + \!\int_{\min\{x_0,\tilde x_0,u\}}^{\max\{x_0,\tilde x_0,u\}} \!\Big\vert \tfrac{ v(s)-\tilde v(s)}{\tilde v(s)v(s)}  \Big\vert \dd s\label{eq:123123123123}\\
    &\leq \tfrac{1}{\underline{v}} \vert x_0 - \tilde x_0\vert  + \tfrac{1}{\underline{v}^2}\|v-\tilde v\|_{\sL^1((\min\{x_0,\tilde x_0,u\},\max\{x_0,\tilde x_0,u\}))}.\notag
\end{align}
This proves the first claim. For the second, namely the estimate in \cref{eq:stability_c}, we
first show that
\begin{equation}
 \tfrac{1}{\|v\|_{\sL^{\infty}(\R)}}\vert x-\tilde{x}\vert \leq \vert \cZ[v](x_{0}\sk x)-\cZ[v](x_{0}\sk \tilde{x})\vert \qquad \forall x,\tilde{x}\in\R^{2}.\label{eq:lower_bound_z_Lipschitz}
\end{equation}
Recalling again \cref{eq:surrogate_system1} we end up with
\begin{align*}
\vert \cZ[v](x_{0}\sk x)-\cZ[v](x_{0}\sk \tilde{x})\vert &=\bigg\vert \int_{\tilde{x}}^{x}\tfrac{1}{v(s)}  \dd s\bigg\vert =\int_{\min\{x,\tilde{x}\}}^{\max\{x,\tilde{x}\}}\!\!
\tfrac{1}{v(s)}\dd s\\
&\geq \big(\max\{x,\tilde{x}\}-\min\{x,\tilde{x}\}\big)\tfrac{1}{\|v\|_{\sL^{\infty}(\R)}},
\end{align*}
which is exactly \cref{eq:lower_bound_z_Lipschitz}.

Finally, focusing on \cref{eq:stability_c}, we recall the definition of \(\cC\) in \cref{eq:surrogate_system2}
and thus estimate for \(t\in[0,T]\)
\begin{align}
&\Big\vert \cC\big[\lambda,\cZ[v](x_{0},\cdot)\big](t)-\cC\big[\tilde{\lambda},\cZ[\tilde{v}](\tilde{x}_{0}\sk \cdot)\big](t)\Big\vert \notag\\ 
&=\bigg\vert \int_0^t \lambda\Big(s,\cZ[v]^{-1}\big(x_{0}\sk c\big[\lambda,\cZ[v]\big(x_{0}\sk \cdot)\big](s)\big)\Big)\notag\\
&\qquad\qquad\qquad\qquad-\tilde{\lambda}\Big(s,\cZ[\tilde{v}]^{-1}\big(\tilde{x}_{0}\sk c\big[\tilde{\lambda},\cZ[\tilde{v}](\tilde{x}_{0}\sk \cdot)\big](s)\big)\Big) \dd s\bigg\vert \notag\\
    &\leq  \mL_{2}\int_{0}^{t} \Big\vert \cZ[v]^{-1}\big(x_{0}\sk c\big[\lambda,\cZ[v]\big(x_{0}\sk \cdot)\big](s)\big)-\cZ[\tilde{v}]^{-1}\big(\tilde{x}_{0}\sk c\big[\tilde{\lambda},\cZ[\tilde{v}](\tilde{x}_{0}\sk \cdot)\big](s)\big)\Big\vert \dd s\notag\\
    &\quad + \int_{0}^{t}\|\lambda(s,\cdot)-\tilde{\lambda}(s,\cdot)\|_{\sL^{\infty}(X(x_{0},v,\lambda))}\dd s\notag\\
     &\leq \mL_{2}\int_{0}^{t} \Big\vert \cZ[v]^{-1}\big(x_{0}\sk c\big[\lambda,\cZ[v]\big(x_{0}\sk \cdot)\big](s)\big)\notag\\
     &\qquad\qquad\qquad\qquad-\cZ[\tilde{v}]^{-1}\big(\tilde{x}_{0}\sk c\big[\lambda,\cZ[v](x_{0}\sk \cdot)\big](s)\big)\Big\vert \dd s\label{eq:4211}\\
    &\quad + \mL_{2}\int_{0}^{t} \Big\vert \cZ[\tilde{v}]^{-1}\big(\tilde{x}_{0}\sk c\big[\lambda,\cZ[v]\big(x_{0}\sk \cdot)\big](s)\big)\notag\\
    &\qquad\qquad\qquad\qquad-\cZ[\tilde{v}]^{-1}\big(\tilde{x}_{0}\sk c\big[\tilde{\lambda},\cZ[\tilde{v}](\tilde{x}_{0}\sk \cdot)\big](s)\big)\Big\vert \dd s\label{eq:4212}\\
    &\quad + \int_{0}^{t}\|\lambda(s,\cdot)-\tilde{\lambda}(s,\cdot)\|_{\sL^{\infty}(X(x_{0},v,\lambda))}\dd s. \notag
    \end{align}
    For the term in \cref{eq:4212}, we now apply  \cref{eq:lower_bound_z_Lipschitz}. As \(\|v\|_{\sL^{\infty}(\R)}\) is then an upper bound on the derivative of \(\cZ[v]^{-1}\) (recall the estimate in \cref{eq:lower_upper_bound_cZ}), we obtain
    \begin{align*}
        \eqref{eq:4212} &\leq \mathcal L_2\|\tilde{v}\|_{\sL^{\infty}(\R)}\int_{0}^{t} \big\vert \cC[\lambda,\cZ[v](x_{0}\sk \cdot)](s)-\cC[\tilde{\lambda},\cZ[\tilde{v}(\tilde{x_{0}},\cdot)](s)\big\vert \dd s.
    \end{align*}
    For the term in \cref{eq:4211}, we get the following estimate by substitution:
    \begin{align*}
        \eqref{eq:4211}    &\leq  \mathcal L_2 t \big\|\cZ[v]^{-1}\big(x_{0}\sk \cC\big[\lambda,\cZ[v]\big(x_{0}\sk \cdot)\big]\big)-\cZ[\tilde{v}]^{-1}\big(\tilde{x}_{0}\sk \cC\big[\lambda,\cZ[v](x_{0}\sk \cdot)\big]\big)\big\|_{\sL^{\infty}((0,T))}\\
 &\leq  \mathcal L_2 t \big\|\cZ[v]^{-1}\big(x_{0}\sk \cdot\big)-\cZ[\tilde{v}]^{-1}\big(\tilde{x}_{0}\sk \cdot\big)\big\|_{\sL^{\infty}((- T\|\lambda\|_{\sL^{\infty}((0,T)\sk \sL^{\infty}(\R))},T\|\lambda\|_{\sL^{\infty}((0,T)\sk \sL^{\infty}(\R))}))}\\
    &\leq  \mL_2 t\big\|\ast-\cZ[\tilde{v}]^{-1}\big(\tilde{x}_{0}\sk \cZ[v](x_{0}\sk \ast)\big)\big\|_{\sL^{\infty}(X(x_0,v,\lambda))},\\
    \intertext{with $X(x_0,v,\lambda)$  defined in \cref{eq:X_x_0}}
    &\leq  \mL_2 t\Big\|\cZ[\tilde{v}]^{-1}\big(\tilde{x}_{0}\sk \cZ[\tilde{v}](\tilde{x}_{0}\sk \ast)\big)-\cZ[\tilde{v}]^{-1}\big(\tilde{x}_{0}\sk \cZ[v](x_{0}\sk \ast)\big)\Big\|_{\sL^{\infty}(X(x_0,v,\lambda))}.\\
    \intertext{Applying once more \cref{eq:lower_bound_z_Lipschitz}, as well as \(\|\tilde{v}\|_{\sL^{\infty}(\R)}\), there is then an upper bound on the derivative of \(\cZ[\tilde{v}]^{-1}\) when recalling once more \cref{eq:lower_upper_bound_cZ}}
    &\leq  \mL_2 t\|\tilde{v}\|_{\sL^{\infty}(\R)}\big\|\cZ[\tilde{v}](\tilde{x}_{0}\sk \cdot)-\cZ[v](x_{0}\sk \cdot)\big\|_{\sL^{\infty}(X(x_0,v,\lambda))}.
    \end{align*}
    Thus, all together we obtain 
    \begin{align*}    
    &\Big\vert \cC\big[\lambda,\cZ[v](x_{0},\cdot)\big](t)-\cC\big[\tilde{\lambda},\cZ[\tilde{v}](\tilde{x}_{0}\sk \cdot)\big](t)\Big\vert  \\
    &\quad\leq  \mL_2 t\|\tilde{v}\|_{\sL^{\infty}(\R)}\big\|\cZ[\tilde{v}](\tilde{x}_{0}\sk \cdot)-\cZ[v](x_{0}\sk \cdot)\big\|_{\sL^{\infty}(X(x_0,v,\lambda))}\\
    &\qquad +\mathcal L_2\|\tilde{v}\|_{\sL^{\infty}(\R)}\int_{0}^{t} \big\vert \cC[\lambda,\cZ[v](x_{0}\sk \cdot)](s)-\cC[\tilde{\lambda},\cZ[\tilde{v}(\tilde{x_{0}},\cdot)](s)\big\vert \dd s \\
    &\qquad + \int_{0}^{t}\|\lambda(s,\cdot)-\tilde{\lambda}(s,\cdot)\|_{\sL^{\infty}(X(x_{0},v,\lambda))}\dd s.
\end{align*}
The claimed inequality in \eqref{eq:stability_c} then follows by applying Gr\"onwall's inequality \cite[Chapter I, III Gronwall's inequality]{walter}, concluding the proof.
\end{proof}
Having obtained the previous stability results on the ``surrogate system'', we can apply these results to obtain the stability of solutions to the original discontinuous ODE in \cref{defi:disc_ODE} in the sense of \cref{defi:weak}:
\begin{theo}[Stability of solutions for initial datum and velocities]\label{theo:stability}
For \((x_{0},\tilde{x}_{0})\in\R^{2}\), \((v,\tilde{v})\in\sL^{\infty}\big(\R\sk \R_{\geq \uv}\big)^{2}\) and \((\lambda,\tilde{\lambda})\in \sL^{\infty}\big((0,T)\sk \sW^{1,\infty}(\R)\big)\), the following stability result holds for the corresponding solutions \(\cX\) as in \cref{defi:weak} for all \(t\in[0,T]\):
\begin{align}
    &\big\vert \cX[v,\lambda](x_0\sk t)-\cX[\tilde{v},\tilde{\lambda}](\tilde{x}_{0}\sk t)\big\vert \notag\\
    & \leq\|v\|_{\sL^{\infty}(\R)}\e^{t\|v\|_{\sL^{\infty}(\R)}\mathcal L_2}\int_{0}^{t}\|\lambda(s,\cdot)-\tilde{\lambda}(s,\cdot)\|_{\sL^{\infty}(X(x_{0},v,\lambda))}\dd s \label{eq:stability_x}\\
    &\ +\Big(\|v\|_{\sL^{\infty}(\R)}\e^{t\|v\|_{\sL^{\infty}(\R)}\mathcal L_2}t\mL_{2}+1\Big)\tfrac{\|\tilde{v}\|_{\sL^{\infty}(\R)}}{\uv}\Big(\vert x_{0}-\tilde{x}_{0}\vert +\tfrac{1}{\uv}\|v-\tilde{v}\|_{\sL^{1}(Y(x_0,\tilde x_0,v,\mL))}\Big),\notag
\end{align}
where
\begin{small}
\begin{align}
   \mL_{2}&\:\max\big\{\|\partial_{2}\lambda\|_{\sL^{\infty}((0,T)\sk \sL^{\infty}(\R))},\|\partial_{2}\tilde{\lambda}\|_{\sL^{\infty}((0,T)\sk \sL^{\infty}(\R))}\big\} \\
   \mL&\:\max\big\{\| \lambda\|_{\sL^{\infty}((0,T)\sk \sL^{\infty}(\R))},\|\tilde \lambda\|_{\sL^{\infty}((0,T)\sk \sL^{\infty}(\R))}\big\} \\ 
  Y(x_0,\tilde x_0,v,\mathcal L) &\:  \Big(\min\big\{x_0-T \|v\|_{\sL^{\infty}(\R)} \mathcal L,\tilde x_0\big\},\max\big\{x_0+T \|v\|_{\sL^{\infty}(\R)}\mathcal L,\tilde x_0\big\}\Big) \label{eq:def_Z}\\
  X(x_{0},v,\lambda)&\: x_{0}+ T\|v\|_{\sL^\infty(\R)}\|\lambda\|_{\sL^{\infty}((0,T)\sk \sL^{\infty}(\R))}\cdot(-1,1) \subset \R.\label{defi:X}
\end{align}
\end{small}
\end{theo}
\begin{proof}
This is a direct application of the previous stability results for the surrogate system of ODEs in \cref{theo:existence_uniqueness_stability_surrogate}. We detail the required steps in the following and start by estimating for \(t\in[0,T]\) using \cref{eq:identity_surrogate_system}
\begin{align}
   &\Big\vert \cX[v,\lambda](x_0\sk t)-\cX\big[\tilde{v},\tilde{\lambda}\big](\tilde{x}_{0}\sk t)\Big\vert \label{eq:2810}\\
    &\quad=\Big\vert \cZ[v]^{-1}\big(x_{0}\sk  \cC[\lambda,\cZ[v](x_{0}\sk \cdot)](t)\big)-\cZ[\tilde{v}]^{-1}\big(\tilde{x}_{0},\cC\big[\tilde{\lambda},\cZ[\tilde{v}](\tilde{x}_{0}\sk \cdot)\big](t)\big)\Big\vert \\
    &\quad\leq\Big\vert \cZ[v]^{-1}\big(x_{0}\sk  \cC[\lambda,\cZ[v](x_{0}\sk \cdot)](t)\big)-\cZ[v]^{-1}\big(x_{0}\sk \cC[\tilde \lambda,\cZ[\tilde{v}](\tilde{x}_{0}\sk \cdot)](t)\big)\Big\vert \notag\\
    &\qquad +\Big\vert \cZ[v]^{-1}\big(x_{0}\sk \cC[\tilde{\lambda},\cZ[\tilde{v}](\tilde{x}_{0}\sk \cdot)](t)\big)-\cZ[\tilde{v}]^{-1}\big(\tilde{x}_{0}\sk \cC[\tilde{\lambda},\cZ[\tilde{v}](\tilde{x}_{0}\sk \cdot)](t)\big)\Big\vert, \notag
    \intertext{applying \cref{eq:lower_bound_z_Lipschitz} so that we can estimate the derivative of the inverse of \(\cZ[v](x_{0}\sk \cdot)\) by \(\|v\|_{\sL^{\infty}(\R)}\),}
    &\quad\leq \|v\|_{\sL^{\infty}(\R)}\big\vert \cC[\lambda,\cZ[v](x_{0}\sk \cdot)](t)-\cC[\tilde{\lambda},\cZ[\tilde{v}](\tilde{x}_{0}\sk \cdot)](t)\big\vert \notag\\
    &\qquad +\Big\vert \cZ[v]^{-1}\big(x_{0}\sk \cC[\tilde{\lambda},\cZ[\tilde{v}](\tilde{x}_{0}\sk \cdot)](t)\big)-\cZ[\tilde{v}]^{-1}\big(\tilde{x}_{0}\sk \cC[\tilde{\lambda},\cZ[\tilde{v}](\tilde{x}_{0}\sk \cdot)](t)\big)\Big\vert. \label{eq:3112}
\end{align}
Focusing on the second term, we have with $y\: \cZ[v]^{-1}\big(x_{0}\sk \cC[\tilde{\lambda},\cZ[\tilde{v}](\tilde{x}_{0}\sk \cdot)](t)\big)$ and thus \(\cZ[v](x_{0}\sk y)=\cC[\tilde{\lambda},\cZ[\tilde{v}](\tilde{x}_{0}\sk \cdot)](t)\)
\begin{align}
    \eqref{eq:3112} &=\Big\vert y-\cZ[\tilde{v}]^{-1}\big(\tilde{x}_{0}\sk \cZ[v](x_0\sk y))\big)\Big\vert.     
    \intertext{Applying \(y=\cZ[\tilde{v}]^{-1}(\tilde{x}_{0}\sk \cZ[\tilde{v}](\tilde{x}_{0}\sk y))\)}
    &= \Big\vert \cZ[\tilde{v}]^{-1}(\tilde{x}_{0}\sk \cZ[\tilde{v}](\tilde{x}_{0}\sk y))-\cZ[\tilde{v}]^{-1}\big(\tilde{x}_{0}\sk \cZ[v](x_0\sk y))\big)\Big\vert. 
    \intertext{Again using \cref{eq:lower_bound_z_Lipschitz} to obtain a Lipschitz-estimate from above for the inverse mapping}
    & \leq \|\tilde{v}\|_{\sL^\infty(\R)}\vert \cZ[\tilde{v}](\tilde{x}_{0}\sk y)-\cZ[v](x_0\sk y)\vert . \label{eq:2212}
\end{align}
Next, estimating \(y\) we have by the definitions of \(\cZ[v]\) and \(\cC[\cdot,\ast]\) as in \cref{eq:surrogate_system1,eq:surrogate_system2} (and by \cref{eq:lower_bound_z_Lipschitz} to once more obtain an upper bound on the Lipschitz constant of \(\cZ[v]^{-1}\))
\begin{align*}
\vert y - x_0\vert  &= 
    \big\vert \cZ[v]^{-1}\big(x_{0}\sk \cC[\tilde{\lambda},\cZ[\tilde{v}](\tilde{x}_{0}\sk \cdot)](t)\big) -  \cZ[v]^{-1}\big(x_{0}\sk 0\big)\big\vert  \\
    &\leq\|v\|_{\sL^{\infty}(\R)}\vert \cC[\tilde{\lambda},\cZ[\tilde{v}](\tilde{x}_{0}\sk \cdot)](t)\vert\leq T \|v\|_{\sL^{\infty}(\R)}\|\tilde \lambda\|_{\sL^{\infty}((0,T)\sk \sL^{\infty}(\R)).}
\end{align*}
In the last estimate we have used the identity for \(\cC\) \cref{eq:surrogate_system2}, which implies the upper bounds used.
Altogether, this implies that \(y\in X(x_{0},v,\tilde{\lambda})\) with \(X\) as in \cref{eq:X_x_0}.
Continuing the estimate, we have
\begin{align}
     \eqref{eq:3112} \leq \eqref{eq:2212}&= \|\tilde{v}\|_{\sL^\infty(\R)}\vert \cZ[\tilde{v}](\tilde{x}_{0}\sk y)-\cZ[v](x_0\sk y)\vert  \notag\\
         &\leq \|\tilde{v}\|_{\sL^\infty(\R)}\|\cZ[\tilde{v}](\tilde{x}_{0}\sk \cdot)-\cZ[v](x_0\sk \cdot)\|_{\sL^\infty (X(x_0,v,\tilde \lambda))} \label{eq:12341234}
         \intertext{and applying the stability estimate in \(\cZ\) in \cref{theo:existence_uniqueness_stability_surrogate}}
    &\overset{\eqref{eq:stability_z}}{\leq} \|\tilde{v}\|_{\sL^\infty(\R)} \Big(\tfrac{1}{\underline{v}} \vert x_0 - \tilde x_0\vert  + \tfrac{1}{\uv^2}\|v-\tilde v\|_{\sL^1(Y(x_0,\tilde x_0,v,\mathcal L))} \Big), \notag
\end{align}
with $Y$ as defined in \cref{eq:def_Z}.

Continuing the original estimate, we take advantage of the stability in \(\cC\) in \cref{eq:stability_c} and have
\begin{align*}
   \eqref{eq:2810} &\leq \|v\|_{\sL^{\infty}(\R)}\e^{t\|v\|_{\sL^{\infty}(\R)}\mathcal L_2}\int_{0}^{t}\!\!\big\|\lambda(s,\cdot)-\tilde \lambda(s,\cdot)\big\|_{\sL^{\infty}(X(x_{0},v,\lambda))}\dd s \\
    &\quad +\|v\|_{\sL^{\infty}(\R)}\e^{t\|v\|_{\sL^{\infty}(\R)}\mathcal L_2} t\|\tilde{v}\|_{\sL^{\infty}(\R)}\mathcal L_2\big\|\cZ[v](x_0\sk \cdot)-\cZ[\tilde{v}](\tilde{x}_{0}\sk \cdot)\big\|_{\sL^{\infty}(X(x_0,v,\lambda))}\\
    &\quad + \|\tilde{v}\|_{\sL^\infty(\R)} \Big(\tfrac{1}{\underline{v}} \vert x_0 - \tilde x_0\vert  + \tfrac{1}{\uv^2}\|v-\tilde v\|_{\sL^1(Y(x_0,\tilde x_0,v,\mathcal L))} \Big). 
    \intertext{Using the stability estimate for \(\cZ[\cdot]\) in \cref{eq:stability_z}}
    &\leq \|v\|_{\sL^{\infty}(\R)}\e^{t\|v\|_{\sL^{\infty}(\R)}\mathcal L_2}\bigg(\int_{0}^{t}\big\|\lambda(s,\cdot)-\tilde \lambda(s,\cdot)\big\|_{\sL^{\infty}(X(x_{0},v,\lambda))}\dd s\\
    &\qquad\qquad\qquad\qquad\qquad+ t\|\tilde{v}\|_{\sL^{\infty}(\R)}\mathcal L_2\Big(\tfrac{1}{\underline{v}} \vert x_0 - \tilde x_0\vert  + \tfrac{1}{\uv^2}\|v-\tilde v\|_{\sL^1(Y(x_{0},\tilde{x}_{0},v,\mL)}\Big)\bigg)\\
    &\quad + \|\tilde{v}\|_{\sL^\infty(\R)} \Big(\tfrac{1}{\underline{v}} \vert x_0 - \tilde x_0\vert  + \tfrac{1}{\uv^2}\|v-\tilde v\|_{\sL^1(Y(x_0,\tilde x_0,v,\mathcal L))} \Big) \\
    &= \|v\|_{\sL^{\infty}(\R)}\e^{t\|v\|_{\sL^{\infty}(\R)}\mathcal L_2}\int_{0}^{t}\|\lambda(s,\cdot)-\tilde{\lambda}(s,\cdot)\|_{\sL^{\infty}(X(x_{0},v,\lambda))}\dd s\\
    &\quad +\Big(\|v\|_{\sL^{\infty}(\R)}\e^{t\|v\|_{\sL^{\infty}(\R)}\mathcal L_2}t\mL_{2}+1\Big)\tfrac{\|\tilde{v}\|_{\sL^{\infty}(\R)}}{\uv}\Big(\vert x_{0}-\tilde{x}_{0}\vert +\tfrac{1}{\uv}\|v-\tilde{v}\|_{\sL^{1}(Y(x_{0},\tilde{x}_{0},v,\mL))}\Big).\end{align*}
This is indeed the claimed estimate and thus the proof is concluded.    
\end{proof}
The previous result in \cref{theo:stability} gives Lipschitz-continuity for the solution with regard to the initial datum. 
As the existence of an explicit solution formula for smooth velocities (\(v,\lambda\)) for \(\partial_{3}\cX\) is important for several later results, we detail it in the following:
\begin{rem}[An ``explicit'' formula for {$\partial_{x_{0}}\cX[v,\lambda](x_{0}\sk \cdot)$} if $v$ is smooth]\label{rem:partial_3_X_explicit}
Let \cref{ass:input_datum} hold and also  \(v\in \sC^{1}(\R)\) and \(\lambda\in \sC^{1}([0,T]\times\R)\). Then we have for \((x_{0},t)\in\R\times(0,T)\) a solution formula for the derivative of \(\cX[v,\lambda](x_{0},\cdot)\) with regard to \(x_{0}\in\R\), namely for \(t\in[0,T]\)
\begin{equation}
 \partial_{3}\cX[v,\lambda](x_{0}\sk t)=\tfrac{v(\cX[v,\lambda](x_0\sk t))}{v(x_0)} \e^{\int_{0}^{t}\partial_2 \lambda\big(s,\cX[v,\lambda](x_{0}\sk s)\big) v\big(\cX[v,\lambda](x_{0}\sk s)\big) \dd s} .\label{eq:partial_1_x_explicit_solution_formula_1}
\end{equation}
To show this, we can differentiate through the integral form of the -- now -- continuously differentiable IVP in \cref{defi:disc_ODE}. We thus take the derivative of \(\cX\) with regard to \(x_{0}\) and have -- following the Carath\'eodory solution in \cref{defi:cara} -- for \(t\in[0,T]\)
\begin{align}
    &\partial_{3}\cX[v,\lambda](x_{0}\sk t) \notag\\
    &\quad=1+\int_{0}^{t} \Big(v'(\cX[v,\lambda](x_{0}\sk s))\lambda(s,\cX[v,\lambda](x_{0}\sk s)) \label{eq:partial_1_x_explicit_solution_formula_2}\\
    &\qquad\qquad +v(\cX[v,\lambda](x_{0}\sk s))\partial_{2}\lambda(s,\cX[v,\lambda](x_{0}\sk s))\Big)
     \cdot\partial_{3}\cX[v,\lambda](x_{0}\sk s)\dd s, 
\notag
\end{align}
which we can explicitly solve to obtain
\begin{align*}
    \partial_{3}\cX[v,\lambda](x_{0}\sk t)=\e^{\int_{0}^{t}a(s)+v(\cX[v,\lambda](x_{0}\sk s))\partial_{2}\lambda(s,\cX[v,\lambda](x_{0}\sk s))\dd s}
\end{align*}
with $a(s) \: v'(\cX[v,\lambda](x_{0}\sk s))\lambda(s,\cX[v,\lambda](x_{0}\sk s))$.
Focusing on the first factor we can write
\begin{align*}
    \e^{\int_0^t a(s) \dd s }
    &=\e^{\int_{0}^{t}v'(\cX[v,\lambda](x_{0}\sk s))\tfrac{v(\cX[v,\lambda](x_{0}\sk s))}{v(\cX[v,\lambda](x_{0}\sk s))}\lambda(s,\cX[v,\lambda](x_{0}\sk s))\dd s }
    \intertext{and by \cref{defi:cara}, i.e.\ \(\partial_{4}\cX[v,\lambda](\cdot\sk \ast)\equiv v\big(\cX[v,\lambda](\cdot\sk \ast)\big)\lambda\big(\ast,\cX[v,\lambda](\cdot\sk \ast)\big)\)}
    &=\e^{\int_{0}^{t}\tfrac{v'(\cX[v,\lambda](x_{0}\sk s))}{v(\cX[v,\lambda](x_{0}\sk s))}\partial_{4}\cX[v,\lambda](x_{0}\sk s)\dd s}\\
    &=\e^{\big[\ln\big(v\big(\cX[v,\lambda](x_{0}\sk s)\big)\big)\big]_{s=0}^{s=t} }=\tfrac{v\big(\cX[v,\lambda](x_{0}\sk t)\big)}{v(x_{0})}.
\end{align*}
In the latter manipulation, we also employed the fact that, according to \cref{defi:cara}, \(\cX[v,\lambda](x_{0}\sk 0)=x_{0}\).
Together with \cref{eq:partial_1_x_explicit_solution_formula_2}, the solution formula for \(\partial_{3}\cX\) in \cref{eq:partial_1_x_explicit_solution_formula_1} follows.
Although the solution formula in \cref{eq:partial_1_x_explicit_solution_formula_1} does not require any regularity on the derivative of \(v\), there is a problem in its interpretation if \(\partial_{2}\lambda\) and \(v\) are not continuous in the time integration in the exponent. For instance, assume that the characteristic curve \(\cX[v,\lambda](x_{0};\cdot)\) is constant over a small time horizon. It is not clear how to interpret the integral in \cref{eq:partial_1_x_explicit_solution_formula_1} over this specific time horizon.
\end{rem}
To obtain an improved estimate for \(\partial_{3}\cX\), we take advantage of \cref{rem:partial_3_X_explicit} for smooth datum \((v,\lambda)\) and use an approximation argument. This is performed in the following \cref{cor:upper_lower_bounds_partial_3_X}:
\begin{cor}[Improved bounds on the Lipschitz constant of $\cX$]\label{cor:upper_lower_bounds_partial_3_X}
Let velocities \(v,\lambda\) satisfy \cref{ass:input_datum}. Then, the solution of the discontinuous IVP as in \cref{defi:weak} satisfies  \(\forall (t,x_{0},\tilde{x}_{0})\in[0,T]\times\R^{2}\)
\begin{align}
\tfrac{\uv}{\|v\|_{\sL^{\infty}(\R)}}\e^{-t \mathcal{L}_{2}\|v\|_{\sL^{\infty}(\R)}} \leq \tfrac{\vert \cX[v,\lambda](x_{0}\sk t)-\cX[v,\lambda](\tilde{x}_{0}\sk t)\vert }{\vert x_{0}-\tilde{x}_{0}\vert }\leq \tfrac{\|v\|_{\sL^{\infty}(\R)}}{\uv}\e^{t \mathcal{L}_{2}\|v\|_{\sL^{\infty}(\R)}}\label{eq:improved_bound_partial_3_X}
\end{align}
with  \(\mathcal{L}_{2}\:\|\partial_{2}\lambda\|_{\sL^{\infty}((0,T)\sk \sL^{\infty}(\R))}.\)
\end{cor}
\begin{proof}
We recall the surrogate system in \cref{eq:surrogate_system1,eq:surrogate_system2}, use the approximation result in \cref{cor:ODE_approximation_smooth} with \(v_{\eps},\lambda_{\eps}\) smoothed, and assume w.l.o.g. \(x_{0}>\tilde{x}_{0}\). Then, we can estimate for \(t\in[0,T]\)
\begin{align*}
\cX[v,\lambda](x_0\sk t)-\cX\big[v,\lambda\big](\tilde{x}_{0}\sk t)&\geq \cX[v_{\eps},\lambda_{\eps}](x_0\sk t)-\cX\big[v_{\eps},\lambda_{\eps}\big](\tilde{x}_{0}\sk t)\\
&\qquad -\Big\vert \cX[v,\lambda](x_0\sk t)-\cX\big[v_{\eps},\lambda_{\eps}\big](x_{0}\sk t)\Big\vert \\
&\qquad 
-\Big\vert \cX[v_{\eps},\lambda_{\eps}](\tilde{x}_0\sk t)-\cX\big[v,\lambda\big](\tilde{x}_{0}\sk t)\Big\vert.
\end{align*}
According to \cref{cor:ODE_approximation_smooth}, the last two terms in the latter estimate vanish if \(\eps\rightarrow 0\). So we only focus on the first term and continue -- for given \(\eps\in\R_{>0}\) -- the estimate as follows
\begin{align}
  \cX[v_{\eps},\lambda_{\eps}](x_0\sk t)-\cX\big[v_{\eps},\lambda_{\eps}\big](\tilde{x}_{0}\sk t) &\geq \inf_{x\in\R} \partial_{3}\cX[v_{\eps},\lambda_{\eps}](x\sk t)(x_{0}-\tilde{x}_{0}).\label{eq:4713}
\end{align}
As \(\cX[v_{\eps},\lambda_{\eps}]\) is the classical solution of an ODE with a smooth right hand side in space, the ``explicit'' solution formula for \(\partial_{3}\cX\) in \cref{rem:partial_3_X_explicit}, namely \cref{eq:partial_1_x_explicit_solution_formula_1}, applies and we have
\begin{align*}
    \eqref{eq:4713}&\geq (x_{0}-\tilde{x}_{0})\inf_{y\in\R} \tfrac{v_{\eps}(\cX[v_{\eps},\lambda_{\eps}](y\sk t))}{v_{\eps}(y)} \e^{\int_{0}^{t}\partial_2 \lambda_{\eps}\big(u,\cX[v_{\eps},\lambda_{\eps}](y\sk u)\big) v_{\eps}\big(\cX[v_{\eps},\lambda_{\eps}](y\sk u)\big) \dd u}\\
    &\geq (x_{0}-\tilde{x}_{0})\tfrac{\uv}{\|v\|_{\sL^{\infty}(\R)}}\exp\Big(-t \|\partial_{2}\lambda\|_{\sL^{\infty}((0,T)\sk \sL^{\infty}(\R))}\|v\|_{\sL^{\infty}(\R)}\Big).
\end{align*}
Here we have also used the fact that it holds by construction
\[
\|\partial_{2}\lambda_{\eps}\|_{\sL^{\infty}((0,T)\sk \sL^{\infty}(\R))}\leq \|\partial_{2}\lambda\|_{\sL^{\infty}((0,T)\sk \sL^{\infty}(\R))},\qquad \uv\leq v_{\eps}(x)\leq \|v\|_{\sL^{\infty}(\R)}\ \forall x\in\R\text{ a.e.}.
\]
Letting \(\eps\rightarrow 0\) we obtain the lower bound when also taking into account the previous arguments on the approximation. For the upper bound, almost the same arguments can be made.
\end{proof}
In the following example, we illustrate the obtained result by means of numerics.
\begin{example}[Numerical illustration of the results for $\cX$ and $\partial_{3}\cX$]
We consider the following data
\[
v \equiv \sgn(\sin(\pi \ast^{-1}))+2,\quad \lambda(\cdot,\ast) \in\{{\color{blue}1},\ {\color{red!80!black}1-\ast},\  {\color{yellow!60!red!70!black}\cos(2\pi \cdot)}\},\quad x_0 \in \{-1,-0.5,0\}
\]
and solve the discontinuous initial value problem in \cref{defi:disc_ODE} by an explicit Runge-Kutta scheme \cite{runge1895,kutta1901beitrag} with adaptive time stepping. As can be seen, the numerical approximations are highly accurate as the blue circles (which represent the exact value at the considered time) match the blue line.
In the different cases, the impact of the discontinuities that accumulate at \(x=0\) can be observed. The pictures on the right illustrate the finite difference approximation of \(\partial_{3}\cX\). The solid lines represent the derived bounds, and the dashed lines are the named numerical approximations. Clearly, these bounds are satisfied. It is worth mentioning that the green lines represent the bounds for the blue \textbf{and} yellow cases and that these estimates are somewhat sharp.
\begin{figure}[ht!]
    \centering
    {\includegraphics[scale=0.63,clip,trim=10 0 20 10]{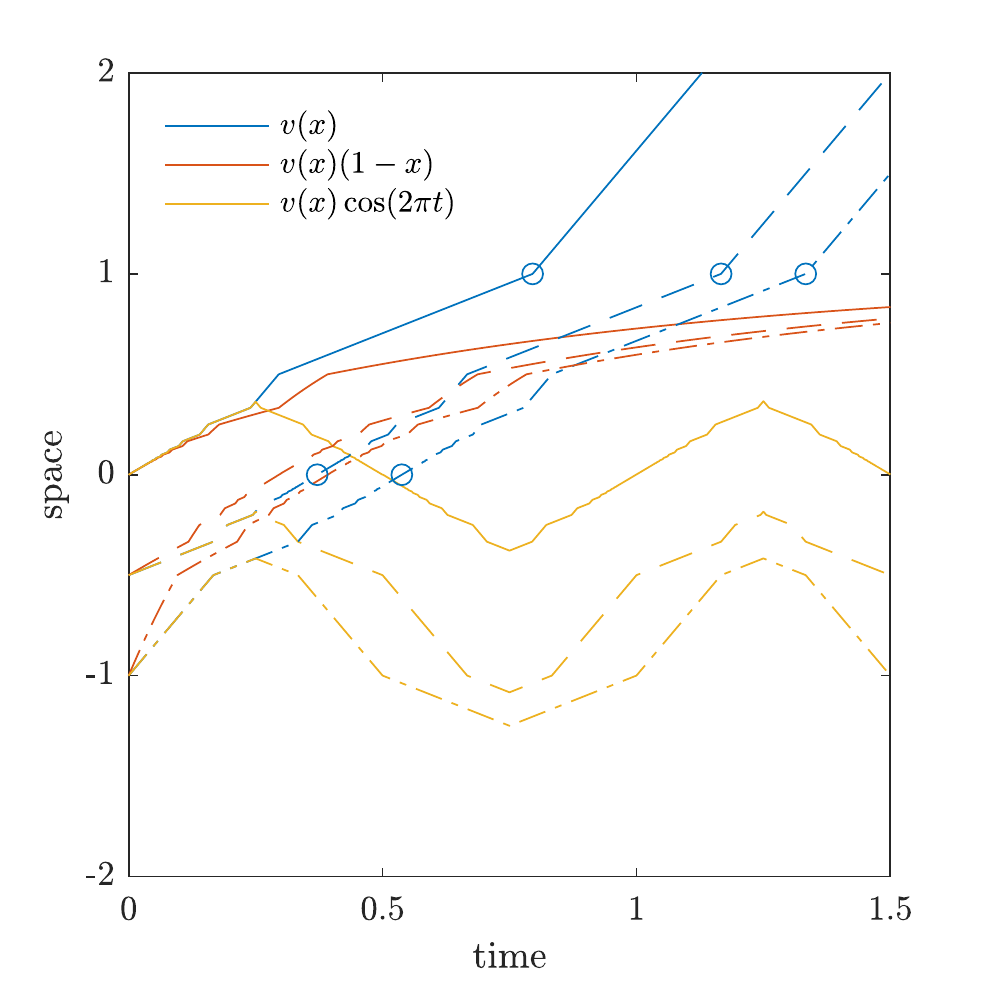}}
    {\includegraphics[scale=0.63,clip, trim=20 0 20 10]{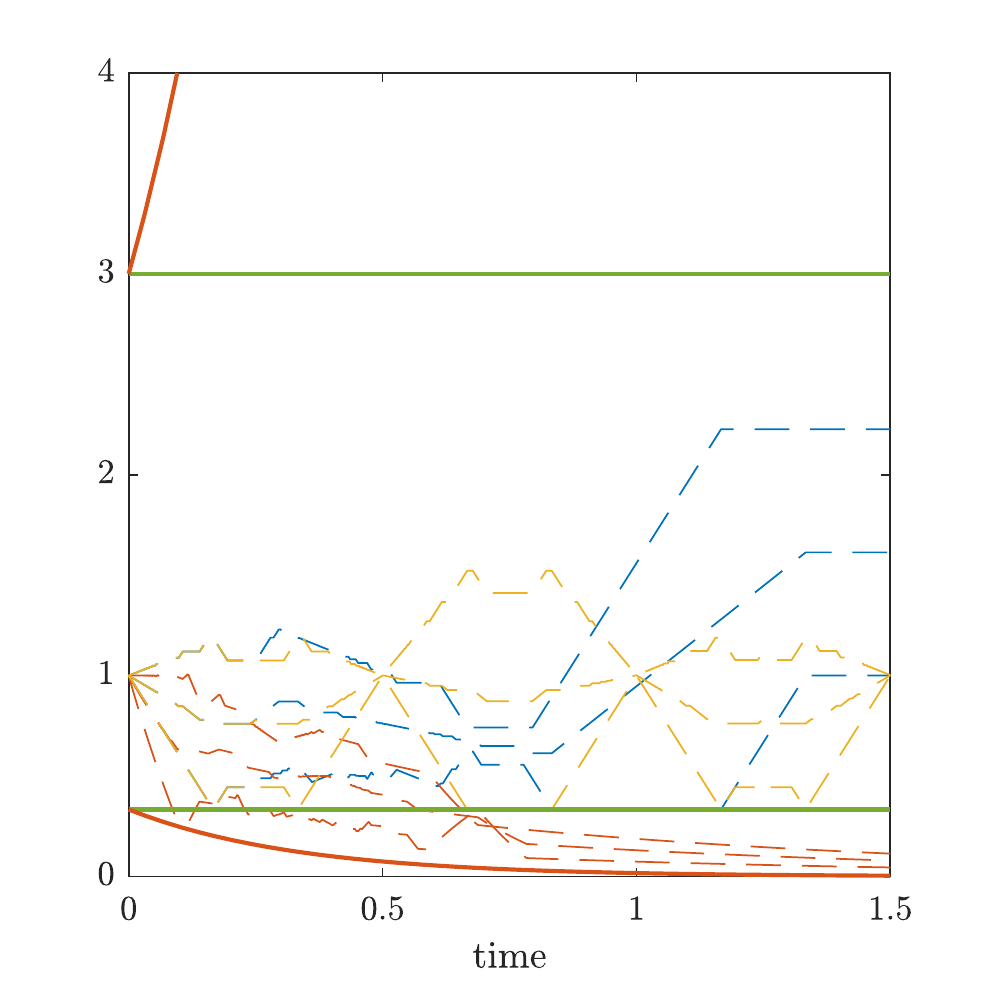}}
    \caption{\textbf{Left:} Numerical approximation of the solutions of the discontinuous IVP for $v \equiv \sgn(\sin(\pi \ast^{-1}))+2$ and $\lambda \equiv 1$ (blue) $\lambda(\cdot,\ast) \equiv 1-\ast$ (red) and $\lambda(\cdot,\ast) \equiv \cos(2\pi \cdot)$ (yellow), as well as initial datum $x_0 \in \{-1,-0.5,0\}$ (dash-dotted, dashed and solid respectively), over time. The blue circles denote analytical solutions for the autonomous and piecewise constant case, i.e. $\lambda \equiv 1,$ showing the high accuracy of the numerical approximation. \textbf{Right:} Finite difference analog to \cref{cor:upper_lower_bounds_partial_3_X} and the bounds derived therein (solid green lines for the blue and yellow case).}
    \label{fig:discontinuous_IVP}
\end{figure}
\end{example}
Though not a classical stability result, the result presented in the following \cref{rem:stability_weak} will enable us  later in \cref{sec:discontinuous_ODE} to obtain a more general approximation result when measuring the differences in the discontinuous velocities in a weak topology.
\begin{rem}[Stability of the solution for $v$ close in a weak topology]\label{rem:stability_weak}
Given the assumptions of \cref{theo:stability}, the difference in the characteristics can actually be estimated for \(t\in[0,T]\) as
\begin{equation}
\begin{aligned}
    &\big\vert \cX[v,\lambda](x_0\sk t)-\cX[\tilde{v},\tilde{\lambda}](\tilde{x}_{0}\sk t)\big\vert \\
    & \leq\|v\|_{\sL^{\infty}(\R)}\e^{t\|v\|_{\sL^{\infty}(\R)}\mathcal L_2}\int_{0}^{t}\|\lambda(s,\cdot)-\tilde{\lambda}(s,\cdot)\|_{\sL^{\infty}(X(x_{0},v,\lambda))}\dd s\\
    &\ +\Big(\|v\|_{\sL^{\infty}(\R)}\e^{t\|v\|_{\sL^{\infty}(\R)}\mathcal L_2}t\mL_{2}+1\Big)\tfrac{\|\tilde{v}\|_{\sL^{\infty}(\R)}}{\uv}\vert x_{0}-\tilde{x}_{0}\vert \\
        &\ +\Big(\|v\|_{\sL^{\infty}(\R)}\e^{t\|v\|_{\sL^{\infty}(\R)}\mathcal L_2}t\mL_{2}+1\Big)\tfrac{\|\tilde{v}\|_{\sL^{\infty}(\R)}}{\uv^2}\bigg\vert\int_{Y(x_{0},\tilde{x}_{0},v,\mL)}  v(s)-\tilde v(s)  \dd s\bigg\vert,
    \end{aligned}
    \label{eq:stability_weak}
\end{equation}
with \(Y(x_{0},\tilde{x}_{0},v,\mL),X(x_{0},v,\lambda)\) as in \cref{eq:def_Z}.
The proof consists of improving the estimate \cref{eq:123123123123}
and using this estimate in \cref{theo:stability}, in particular after \cref{eq:12341234}.

This result is significantly stronger than the result in \cref{theo:stability} as the term \(v-\tilde{v}\) goes into the estimate integrated and not in ``norm difference''. The estimate is not as canonical as a classical \(\sL^{1}\)-norm estimate, but is required -- particularly in \cref{sec:nonlocal_well_posedness} in \cref{theo:weak_stability}, a stability result with regard to the input datum.
\end{rem}

The previous stability result in \cref{theo:stability} and \cref{rem:stability_weak} enables us to approximate the discontinuous IVP by a sequence of continuous IVPs:

\begin{cor}[Approximating the discontinuous IVP by a smooth IVP]\label{cor:ODE_approximation_smooth}
Let \(v\in \sL^{\infty}\big(\R\sk \R_{\geq\uv}\big)\) and \(\lambda\in \sL^{\infty}\big((0,T)\sk \sW^{1,\infty}(\R)\big)\) be given. Take 
\begin{gather*}
\{v_{\eps}\equivd\phi_{\eps}\ast v\}_{\eps\in\R_{>0}}\!\!\subset \sC^{\infty}(\R)\cap \sL^{\infty}\big(\R\sk \R_{\geq\uv}\big),\\ \{\lambda_{\eps}\equivd\phi_{\eps}\ast \lambda(t,\cdot)\}_{\eps\in\R_{>0}}\!\!\subset \sC^{\infty}(\R)\cap \sW^{1,\infty}(\R) \ \forall t\in[0,T] \text{ a.e.}
\end{gather*}
with \(\{\phi_{\eps}\}_{\eps\in\R_{>0}}\subset C^{\infty}(\R)\) the standard mollifier as in \cite[Remark C.18, ii]{leoni}.
Then, for the solutions \(\cX\) to the corresponding discontinuous IVP as in \cref{defi:weak}, it holds
\[
\lim_{\eps \rightarrow 0}\big\|\cX[v,\lambda]-\cX[v_{\eps},\lambda_{\eps}]\big\|_{\sL^{\infty}((0,T)\sk\sL^{\infty}(\R))}=0.
\]
\end{cor}
\begin{proof}
Thanks to the used mollifier we have that 
\begin{equation}
\lim_{\eps\rightarrow 0}\sup_{y\in\R}\Big\vert\int_{0}^{y}v_{\eps}(x)-v(x)\dd x\Big\vert=0\text{ and }  \lim_{\eps\rightarrow 0}\|\lambda_{\eps}-\lambda\|_{\sL^{\infty}((0,T)\sk\sL^{\infty}(\R))}\label{eq:a_nice_one}
\end{equation} (for the precise result see \cite[Theorem 13.9 \& Remark 13.11]{leoni}). 
However, the first approximation result needs to be detailed. Let \(y\in\R\) be given. We can estimate for \(\eps\in\R_{>0}\)
\begin{align*}
    &\Big\vert\int_{0}^{y}\big(v\ast \phi_{\eps}\big)(x)-v(x)\dd x\Big\vert
    =\Big\vert\int_{-\eps}^{\eps}\phi_{\eps}(z)\Big(\int_{z}^{z+y}\!\!v(x)\dd x -\int_{0}^{y}\!\!v(x)\dd x\Big)\dd z\Big\vert\\
    &\ =\Big\vert\int_{-\eps}^{\eps}\!\!\!\!\phi_{\eps}(z)\Big(\!-\!\!\int_{0}^{z}\!\!\!\!\!\!v(x)\dd x +\int_{y}^{z+y}\!\!\!\!\!\!\!\!v(x)\dd x\Big)\dd z\Big\vert
    \leq 2\!\!\int_{-\eps}^{\eps} \!\!\phi_{\eps}(z)\vert z\vert\|v\|_{\sL^{\infty}(\R)}\dd z\leq 2\eps\|v\|_{\sL^{\infty}(\R)}.
\end{align*}
As this is uniform in \(y\) we indeed obtain the claim in \cref{eq:a_nice_one}.
Thus, we can recall \cref{rem:stability_weak} and obtain for given \(x_{0}\in\R\) and \(\tilde{x}_{0}=x_{0}\) and \(t\in[0,T]\)
\begin{equation}
\begin{aligned}
    &\big\vert \cX[v,\lambda](x_0\sk t)-\cX[v_{\eps},\lambda_{\eps}](x_{0}\sk t)\big\vert \\
    & \leq\|v\|_{\sL^{\infty}(\R)}\e^{t\|v\|_{\sL^{\infty}(\R)}\mathcal L_2}\int_{0}^{t}\|\lambda(s,\cdot)-\lambda_{\eps}(s,\cdot)\|_{\sL^{\infty}(X(x_{0},v,\lambda))}\dd s\\
    &\ +\Big(\|v\|_{\sL^{\infty}(\R)}\e^{t\|v\|_{\sL^{\infty}(\R)}\mathcal L_2}t\mL_{2}+1\Big)\tfrac{\|v_{\eps}\|_{\sL^{\infty}(\R)}}{\uv^2}\bigg\vert\int_{Y(x_{0},x_{0},v,\mL)}  v(s)-\tilde v_{\eps}(s)  \dd s\bigg\vert,
    \end{aligned}
\end{equation}
with
\begin{align*}
Y(x_{0},x_{0},v,\mL)&=x_0 + T\|v\|_{\sL^{\infty}(\R)}\mL(-1,1)\\
X(x_{0},v,\lambda)&= x_{0}+ T\|v\|_{\sL^\infty(\R)}\|\lambda\|_{\sL^{\infty}((0,T)\sk \sL^{\infty}(\R))}\cdot(-1,1) \subset \R
\end{align*}
as in \cref{defi:X}.
Making this uniform in \(x_{0}\) and \(t\), we indeed obtain
\begin{align*}
    &\big\|\cX[v,\lambda]-\cX[v_{\eps},\lambda_{\eps}]\big\|_{\sL^{\infty}((0,T);\sL^{\infty}(\R))}\\
    &\leq \|v\|_{\sL^{\infty}(\R)}\e^{T\|v\|_{\sL^{\infty}(\R)}\mathcal L_2}T\|\lambda-\lambda_{\eps}\|_{\sL^{\infty}((0,T);\sL^{\infty}(X(x_{0},v,\lambda)))} \\
                &\quad+ \Big(\|v\|_{\sL^{\infty}(\R)}\e^{T\|v\|_{\sL^{\infty}(\R)}\mathcal L_2}T\mL_{2}+1\Big)\tfrac{\|v\|_{\sL^{\infty}(\R)}}{\uv^2}\sup_{y\in\R}\bigg\vert\int_{y-T\|v\|_{\sL^{\infty}(\R)}\mL}^{y+T\|v\|_{\sL^{\infty}(\R)}\mL}\hspace{-5pt} v(s)-\tilde v_{\eps}(s)  \dd s\bigg\vert.
\end{align*}
For \(\eps\rightarrow 0\), the last two terms go to zero thanks to \cref{eq:a_nice_equation}. Thus, we have shown the claimed approximation result.
\end{proof}
In the following remark we explain why the result in \cref{theo:stability} needs to be strengthened by a regularity result connecting the dependency of \(\cX\) with the initial datum and the considered time.
\begin{rem}[{$\partial_{3}\cX[v,\lambda](\cdot,t)$} as a function]\label{rem:partial_3_X}
Thanks to \cref{theo:stability}, the mapping \(y\mapsto \partial_{y}\cX[v,\lambda](y,t)\) is well-defined for each \(t\in[0,T]\) in the sense that we have by \cref{theo:stability} that \(\cX[v,\lambda](y,t)\) is Lipschitz with regard to \(y\). Thus \(\forall t\in[0,T]\) it holds that
\[
\cX[v,\lambda](\cdot,t)\in \sW^{1,\infty}_{\loc}(\R): \partial_{3}\cX[v,\lambda](\cdot,t)\in \sL^{\infty}(\R)
\]
(also compare \cref{rem:partial_3_X_explicit}).
\end{rem}
However, the regularity discussed in \cref{rem:partial_3_X} does not tell us anything about how the solution changes over time. As we later require a continuity in time, the most natural choice in space for this continuity to hold is \(\sL^{1}_{\loc}\). Then, the continuity can be proven under no additional assumptions on the discontinuous velocity. This is detailed and shown in the \cref{subsec:time_continuity_partial_2}, and in particular in \cref{lem:stability_C_L_1}.
\subsection{Time-continuity/properties of the derivative of the solutions with respect to the initial datum in \texorpdfstring{\(\sL^{1}_{\loc}(\R)\)}{L1}}\label{subsec:time_continuity_partial_2}
As previously mentioned, in this section we study the regularity of \(\partial_{3}\cX\) in time when measuring space in \(\sL^{1}_{\loc}(\R)\). This is particularly important for our later analysis of the discontinuous nonlocal conservation law in \cref{defi:disc_conservation_law} (see \cref{sec:nonlocal_well_posedness}).
Before detailing the claimed continuity, we require a convergence result for a composition of functions in \(\sL^{1}_{\loc}(\R)\) with locally Lipschitz-continuous functions.
\begin{lem}[Convergence of a composition of sequences of functions]\label{lem:convergence_composition_L_1}
Let \(\big\{f_{\eps}\big\}_{\eps\in\R_{>0}}\subset\sL^{1}_{\loc}(\R)\ni f\) be given, as well as \(\{g_{\eps}\}_{\eps\in\R_{>0}}\subset W^{1,\infty}_{\loc}(\R)\ni g\) so that \(\exists C\in\R_{>0}\ \forall \eps\in\R_{>0}:\ C\leq g'_{\eps}\). Then, it holds that
\begin{equation}
\lim_{\eps\rightarrow 0}\|f_{\eps}-f\|_{\sL^{1}_{\loc}(\R)}=0=\lim_{\eps\rightarrow 0}\|g_{\eps}-g\|_{\sL^{\infty}(\R)} 
 \ \implies \
    \lim_{\eps\rightarrow 0}\|f_{\eps}\circ g_{\eps}-f\circ g\|_{\sL^{1}_{\loc}(\R)}=0. \label{eq:convergence_composition_assumption}
\end{equation}
\end{lem}
\begin{proof}
Let \(X\subset \R\) be open and bounded. With \(\big\{\psi_{\delta}\big\}_{\delta\in\R_{>0}}\subset\sC^{\infty}(\R)\) we estimate the standard mollifier as in \cite[Remark C.18, ii]{leoni} as follows:
\begin{align*}
    &\|f_{\eps}\circ g_{\eps}-f\circ g\|_{\sL^{1}_{\loc}(X)}\\
    &\leq \|f_{\eps}\circ g_{\eps}-f\circ g_{\eps}\|_{\sL^{1}_{\loc}(X)}+\|f\circ g_{\eps}-(f\ast \psi_\delta)\circ g_{\eps}\|_{\sL^{1}_{\loc}(X)}\\
    &\qquad + \|(f\ast \psi_\delta)\circ g_{\eps}-(f\ast \psi_\delta)\circ g\|_{\sL^{1}_{\loc}(X)} + \|(f\ast \psi_\delta)\circ g-f\circ g\|_{\sL^{1}_{\loc}(X)}.
\end{align*}
Letting \(\eps\rightarrow 0\), the first term (due to the uniform Lipschitz bound of \(g_{\eps}\) from below) and the third term by the dominated convergence \cite[2.24 The Dominated Convergence Theorem]{folland} converge to zero. The second and fourth term vanish for \(\delta\rightarrow 0\) as \(\psi_{\delta}\) was a standard mollifier.
\end{proof}

In the following \cref{lem:stability_C_L_1}, we present the main result of this section, the continuity of \(\partial_{3}\cX\) in time when measuring space in \(\sL^{1}_{\loc}\). To this end, we take advantage of the derived solutions formula in \cref{rem:partial_3_X_explicit} for smooth velocities and the stability of solutions \(\cX\) with regard to different velocities (see \cref{theo:stability}).
\begin{prop}[Stability spatial derivative of $\cX$ in time assuming $\sTV_{\loc}$ regularity in $v$]\label{lem:stability_C_L_1}
Given \(x_{0}\in\R\) and \((v,\lambda)\in\sL^{\infty}(\R\sk \R_{\geq \uv})\times\sL^{\infty}\big((0,T)\sk \sW^{1,\infty}(\R)\big)\) as in \cref{ass:input_datum},

then, for \(\cX[v,\lambda](x_0\sk \cdot)\) as in \cref{defi:weak}, it holds that
\[
[0,T]\ni t \mapsto \Big( \R\ni y \mapsto \partial_{y}\cX[v,\lambda](y,t) \Big)\in \sC\big([0,T]\sk \sL^{1}_{\loc}(\R)\big),
\]
i.e.\ 
\[
\lim_{t\rightarrow\tilde{t}}\big\|\partial_{3}\cX[v,\lambda](\cdot;t)-\partial_{3}\cX[v,\lambda](\cdot;\tilde{t})\big\|_{\sL^{1}_{\loc}(\R)}=0.
\]
In addition, if we have
\[
v\in \sTV_{\loc}(\R) :\Longleftrightarrow \vert v\vert _{\sTV(\Omega)}<\infty\qquad \forall \Omega\subset\R \text{ compact},
\]
we obtain Lipschitz-continuity in time, i.e., the following estimate holds \(\forall (t,\tt)\in[0,T]^{2}\) and \(\forall X\subset\R\) open and bounded:
\begin{align*}
    &\|\partial_{3}\cX[v,\lambda](\cdot\sk t)-\partial_{3}\cX[v,\lambda](\cdot\sk \tt)\|_{\sL^{1}(X)}\\
    &\leq \tfrac{\vert t-\tilde{t}\vert\|v\|_{\sL^{\infty}(\R)}^{2}\e^{T\mathcal{L}_{2}\|v\|_{\sL^{\infty}(\R)}}}{\uv}\bigg(\mathcal{L}_{2}\vert X\vert+\tfrac{\|v\|_{\sL^{\infty}(\R)}^{2}\|\mathcal{L}\|}{\uv}\e^{T\mathcal{L}_{2}\|v\|_{\sL^{\infty}(\R)}}\vert v\vert_{\sTV\left( X+\|v\|_{\sL^{\infty}(\R)}T\mL(-1,1)\right)}\bigg),
    \end{align*}
    with
    \begin{align}
    \mL&\: \|\lambda\|_{\sL^{\infty}((0,T)\sk \sL^{\infty}(\R))},\qquad \mL_{2}\: \|\partial_{2}\lambda\|_{\sL^{\infty}((0,T)\sk \sL^{\infty}(\R))}.
 \end{align}  
\end{prop}
\begin{proof}
We show the claim by a finite difference approximation in the initial value and approximate the discontinuous velocity \(v\) by a smoothed velocity, as well as the Lipschitz-continuous velocity \(\lambda\) by a smoothed velocity. Now, let \((\eps,h)\in\R_{>0}^{2}\) be given, 
in the notation suppress the dependency of \(x\) with regard to the velocities, i.e., for now write (recall \cref{defi:weak})
\[
x \equivd \cX[v,\lambda]\quad \text{ and }\quad \xe \equivd \cX[\ve,\lambda_{\eps}] \quad \text{ on } \R\times(0,T).
\] Consider for \(X\subset\R\) compact and  \((t,\tilde{t})\in[0,T]^2\):
\begin{align}
\|\partial_{1}x(\cdot\sk\tilde{t})-\partial_{1}x(\cdot\sk t)\|_{\sL^{1}(X)}&=\int_{X}\lim_{h\rightarrow 0} \bigg\vert  \tfrac{x(z+h\sk \tt)-x(z\sk \tt)}{h} - \tfrac{x(z+h\sk t)-x(z\sk t)}{h}\bigg\vert \dd z \label{eq:partial_2_xi_t_tilde_t}\\
&=\lim_{h\rightarrow 0}\int_{X} \bigg\vert  \tfrac{x(z+h\sk \tt)-x(z\sk \tt)}{h} - \tfrac{x(z+h\sk t)-x(z\sk t)}{h}\bigg\vert \dd z,\label{eq:123321123321}
\end{align}
where changing the order of the limit with the integration follows by the dominated convergence theorem \cite[2.24 The Dominated Convergence Theorem]{folland}. Adding several zeros -- the smoothed version of the previous terms -- we estimate
\begin{align}
\eqref{eq:123321123321}& \leq \lim_{h\rightarrow 0}\tfrac{1}{h}\Big(\|x(\cdot+h\sk \tt)-x_{\eps}(\cdot+h\sk \tt)\|_{\sL^1(X)} +\|x(\cdot\sk \tt)-x_{\eps}(\cdot\sk \tt)\|_{\sL^1(X)}\\
   &\quad +\|x(\cdot+h\sk t)-x_{\eps}(\cdot+h\sk t)\|_{\sL^1(X)}+\|x(\cdot\sk t)-x_{\eps}(\cdot\sk t)\|_{\sL^1(X)}\Big) \\ 
    &\quad + \lim_{h\rightarrow 0}\int_{X} \bigg\vert  \tfrac{x_{\eps}(z+h\sk \tt)-x_{\eps}(z\sk \tt)}{h} - \tfrac{x_{\eps}(z+h\sk t)-x_{\eps}(z\sk t)}{h}\bigg\vert \dd z.
\end{align}
Concentrating on the last term, we have for \(h\in\R_{>0},\ h\leq 1\) and by defining  $g_\eps(\cdot,\ast)\: \partial_2 \lambda_\eps(\cdot,\ast)v_\eps(\ast)$ fixed:
\begin{align}
    &\int_{X} \bigg\vert  \tfrac{x_{\eps}(z+h\sk \tt)-x_{\eps}(z\sk \tt)}{h} - \tfrac{x_{\eps}(z+h\sk t)-x_{\eps}(z\sk t)}{h}\bigg\vert \dd z \label{eq:a_nice_equation}\\
    &\overset{\eqref{eq:partial_1_x_explicit_solution_formula_1}}{=}\tfrac{1}{h}\int_{X}\bigg\vert \int_0^h \tfrac{v_{\eps}(x_{\eps}(z+s\sk t))}{v_{\eps}(z+s)} \e^{\int_{0}^{t}g_{\eps}(u,x_{\eps}(z+s\sk u))  \dd u}\notag\\
    &\qquad\qquad\qquad\qquad -  \tfrac{v_{\eps}(x_{\eps}(z+s\sk \tilde{t}))}{v_{\eps}(z+s)} \e^{\int_{0}^{\tilde{t}}g_{\eps}(u,x_{\eps}(z+s\sk u))  \dd u}\dd s \bigg\vert \dd z\notag \\
    &\ \leq\tfrac{1}{h\uv}\int_{X}\int_0^h \Big\vert  v_{\eps}(x_{\eps}(z+s\sk t)) \e^{\int_{0}^{t}g_{\eps}(u,x_{\eps}(z+s\sk u))  \dd u} \notag\\
    &\qquad\qquad\qquad\qquad-  v_{\eps}(x_{\eps}(z+s\sk \tilde{t}))\e^{\int_{0}^{\tilde{t}}g_{\eps}(u,x_{\eps}(z+s\sk u))  \dd u}\Big\vert \dd s \dd z,\label{eq:12331231231231233}
    \end{align}
where the last two estimates used the identity in \cref{rem:partial_3_X_explicit} for smooth data and the fact that \(\uv\leq v_{\eps}(x)\ \forall \eps\in\R_{>0}, \forall x\in\R\) a.e.. Again recalling that it holds 
\[\|v_{\eps}\|_{\sL^{\infty}(\R)}\leq \|v\|_{\sL^{\infty}(\R)}\ \text{as well as}\ \partial_{2}\lambda_{\eps}\|_{\sL^{\infty}((0,T)\sk \sL^{\infty}(\R))}\leq\partial_{2}\lambda\|_{\sL^{\infty}((0,T)\sk \sL^{\infty}(\R))}\]
 and once more adding zeros yields
\begin{align*}
    \eqref{eq:12331231231231233}&\leq\tfrac{\|v\|_{\sL^{\infty}(\R)}}{h\uv}\e^{T\|\partial_2 \lambda\|_{\sL^\infty}\|v\|_{\sL^{\infty}(\R)}}\int_{X}\int_0^h \Big\vert  \int_{\tilde{t}}^{t}g_{\eps}(u,x_{\eps}(z+s\sk u))  \dd u\Big\vert \dd s \dd z\\
    &\qquad +  \tfrac{\e^{T\|\partial_2 \lambda\|_{\sL^\infty}\|v\|_{\sL^{\infty}(\R)}}}{h\uv}\int_{X}\int_0^h \Big\vert  v_{\eps}(x_{\eps}(z+s\sk t))  -  v_{\eps}(x_{\eps}(z+s\sk \tilde{t}))\Big\vert \dd s \dd z\\
    &\leq\tfrac{\|v\|_{\sL^{\infty}(\R)}}{\uv}\e^{T\|\partial_2 \lambda\|_{\sL^\infty}\|v\|_{\sL^{\infty}(\R)}}\|\partial_2 \lambda\|_{\sL^\infty}\|v\|_{\sL^{\infty}(\R)} \vert X\vert \vert t-\tilde t\vert \\
    &\qquad +  \tfrac{\e^{T\|\partial_2 \lambda\|_{\sL^\infty}\|v\|_{\sL^{\infty}(\R)}}}{h\uv}\int_{X}\int_0^h \Big\vert  v_{\eps}(x_{\eps}(z+s\sk t))  -  v_{\eps}(x_{\eps}(z+s\sk \tilde{t}))\Big\vert \dd s \dd z\\
    &\leq\tfrac{\|v\|_{\sL^{\infty}(\R)}}{\uv}\e^{T\|\partial_2 \lambda\|_{\sL^\infty}\|v\|_{\sL^{\infty}(\R)}}\|\partial_2 \lambda\|_{\sL^\infty}\|v\|_{\sL^{\infty}(\R)} \vert X\vert \vert t-\tilde t\vert \\
    &\qquad +  \tfrac{\e^{T\|\partial_2 \lambda\|_{\sL^\infty}\|v\|_{\sL^{\infty}(\R)}}}{\uv}\sup_{s\in(0,h)}\int_{X}\Big\vert  v_{\eps}(x_{\eps}(z+s\sk t))  -  v_{\eps}(x_{\eps}(z+s\sk \tilde{t}))\Big\vert \dd z\\
    &\leq \tfrac{\|v\|_{\sL^{\infty}(\R)}}{\uv}\e^{T\|\partial_2 \lambda\|_{\sL^\infty}\|v\|_{\sL^{\infty}(\R)}}\|\partial_2 \lambda\|_{\sL^\infty}\|v\|_{\sL^{\infty}(\R)} \vert X\vert \vert t-\tilde t\vert \\
    &\qquad +  \tfrac{\e^{T\|\partial_2 \lambda\|_{\sL^\infty}\|v\|_{\sL^{\infty}(\R)}}}{\uv}\int_{\tilde{X}(h)}\Big\vert  v_{\eps}(x_{\eps}(z\sk t))  -  v_{\eps}(x_{\eps}(z\sk \tilde{t}))\Big\vert \dd z,
\end{align*}
with \[\tilde{X}(h)\: (0,h)+X \subset \R.\] Letting \(\eps\rightarrow 0\) the first term remains the same, while by \cref{lem:convergence_composition_L_1} the second yields
\begin{align*}
    \lim_{\eps\rightarrow 0}\int_{\tilde{X}(h)}\Big\vert  v_{\eps}(x_{\eps}(z\sk t))  -  v_{\eps}(x_{\eps}(z\sk \tilde{t}))\Big\vert \dd z&=\int_{\tilde{X}(h)}\Big\vert  v(x(z\sk t))  -  v(x(z\sk \tilde{t}))\Big\vert \dd z.
\end{align*}
It is worth noting that this last term is well-defined as \(v(x(z\sk t))\) is integrated over \(z\), for which the ODE solution \(x\:\cX[v,\lambda]\) is strictly monotone according to \cref{cor:upper_lower_bounds_partial_3_X}.

Recalling all the previous estimates starting from \cref{eq:partial_2_xi_t_tilde_t}, for \(h\rightarrow 0\) we have:

\begin{equation}
\begin{split}
   \|\partial_{1}x(\cdot\sk\tilde{t})-\partial_{1}x(\cdot\sk t)\|_{\sL^{1}(X)} &\leq \tfrac{\e^{T\|\partial_2 \lambda\|_{\sL^\infty}\|v\|_{\sL^{\infty}(\R)}}}{\uv}\|v\|_{\sL^{\infty}(\R)}^{2}\|\partial_2 \lambda\|_{\sL^\infty}\vert X\vert \vert t-\tilde t\vert \\
   &\quad + \tfrac{\e^{T\|\partial_2 \lambda\|_{\sL^\infty}\|v\|_{\sL^{\infty}(\R)}}}{\uv}\int_{X}\Big\vert  v(x(z\sk t))  -  v(x(z\sk \tilde{t}))\Big\vert \dd z.
\end{split}
   \label{eq:partial_2_xi_continuity}
\end{equation}

Now, distinguish the two considered cases:
\begin{description}
    \item[$v\not\in \sTV_{\loc}(\R)$:] Then, when letting \(t\rightarrow \tilde{t}\), the first term in \cref{eq:partial_2_xi_continuity} converges to zero. This is also the case for the second term by \cref{lem:convergence_composition_L_1} as we have \(\|x(\cdot\sk t)-x(\cdot\sk \tilde{t})\|_{\sL^{\infty}(\R)}\rightarrow 0\) for \(t\rightarrow \tilde{t}\) and \(\partial_{1}x\) bounded away from zero by \cref{cor:upper_lower_bounds_partial_3_X}. 
    \item[$v\in \sTV_{\loc}(\R)$:] Since we then want to obtain the Lipschitz-continuity in time, we need to study the second term in \cref{eq:partial_2_xi_continuity} in more detail. To this end, we reformulate as follows and use the standard mollifier for \(v_{\eps}\) to have
    \begin{align*}
        &\int_{X}\Big\vert v(x(z\sk t))-v(x(z\sk \tilde{t}))\Big\vert \dd z\\
        &\leq 
        \int_{X}\Big\vert v(x(z\sk t))-v_{\eps}(x(z\sk t))\Big\vert \dd z + \int_{X}\Big\vert v_{\eps}(x(z\sk t))-v_{\eps}(x(z\sk \tilde{t}))\Big\vert \dd z\\
        &\qquad +\int_{X}\!\Big\vert v_{\eps}(x(z\sk \tilde{t}))-v(x(z\sk \tilde{t}))\Big\vert \dd z.
    \end{align*}
    Focusing only on the second term
    \begin{align*}
        &\int_{X}\!\Big\vert v_{\eps}(x(z\sk t))-v_{\eps}(x(z\sk \tilde{t}))\Big\vert \dd z\\
        &\leq\int_{X}\Big\vert \int_{x(z\sk t)}^{x(z\sk\tilde{t})}\vert v_{\eps}'(y)\vert \dd y\Big\vert \dd z\\
          &\leq \big\|x^{-1}(\cdot,t)-x^{-1}(\cdot,\tilde t)\big\|_{\sL^\infty(x(X;t)\cup x(X;\tilde t))}\int_{x(X;t)\cup x(X;\tilde t)}\vert v_{\eps}'(y)\vert\dd y,\\
          \intertext{with $x^{-1}$ the inverse of $x$ w.r.t. its first argument. This argument exists and is unique according to \cref{cor:upper_lower_bounds_partial_3_X}.}
        &\overset{\eqref{eq:improved_bound_partial_3_X}}{\leq} \tfrac{\|v\|_{\sL^{\infty}(\R)}}{\uv}\e^{T\mathcal{L}_{2}\|v\|_{\sL^{\infty}(\R)}}\|x(\cdot;\tilde{t})-x(\cdot;t)\|_{\sL^{\infty}(\R)}\int_{x(X;t)\cup x(X;\tilde{t})}\vert v'_{\eps}(y)\vert \dd y\\
        &\ \leq \tfrac{\|v\|_{\sL^{\infty}(\R)}^{2}\|\lambda\|_{\sL^{\infty}((0,T);\sL^{\infty}(\R))}}{\uv}\e^{T\mathcal{L}_{2}\|v\|_{\sL^{\infty}(\R)}}\vert t-\tilde{t}\vert\int_{\cup_{s\in[0,T]}x(X;s)}\vert v_{\eps}'(y)\vert\dd y\\
        &\overset{\eps\rightarrow 0}{\leq}\tfrac{\|v\|_{\sL^{\infty}(\R)}^{2}\|\lambda\|_{\sL^{\infty}((0,T);\sL^{\infty}(\R))}}{\uv}\e^{T\mathcal{L}_{2}\|v\|_{\sL^{\infty}(\R)}}\vert t-\tilde{t}\vert \vert v\vert_{\sTV(\cup_{s\in[0,T]}x(X,s))}\\
        &\ \leq \tfrac{\|v\|_{\sL^{\infty}(\R)}^{2}\|\lambda\|_{\sL^{\infty}((0,T);\sL^{\infty}(\R))}}{\uv}\e^{T\mathcal{L}_{2}\|v\|_{\sL^{\infty}(\R)}}\vert t-\tilde{t}\vert \vert v\vert_{\sTV\big(X+T\|v\|_{\sL^{\infty}(\R)}\mathcal{L}\cdot(-1,1)\big)}.
    \end{align*}
    Using this in the estimate in \cref{eq:partial_2_xi_continuity} yields the claim.
\end{description}
\end{proof}

The previous statement used an approximation result to derive the required regularity of \(\partial_{3}\cX\). However, at least for \(v\in\sTV(\R)\) it is possible to obtain this directly with the surrogate system introduced in \cref{theo:surrogate}.
\begin{cor}[An alternative proof of \cref{lem:stability_C_L_1} using \cref{theo:surrogate}]\label{cor:continuity_alternative}
 Given \(x_{0}\in\R\) and \((v,\lambda)\in\sL^{\infty}(\R\sk \R_{\geq \uv})\times\sL^{\infty}\big((0,T)\sk \sW^{1,\infty}(\R)\big)\) as in \cref{ass:input_datum},
assume in addition
\[
v\in \sTV_{\loc}(\R).
\]
Then, we have for \(\cX[v,\lambda](x_0\sk \cdot)\) as in \cref{defi:weak} that \(\forall (t,\tt)\in[0,T]^{2}\)
\begin{align*}
    &\|\partial_{3}\cX[v,\lambda](\cdot\sk t)-\partial_{3}\cX[v,\lambda](\cdot\sk \tt)\|_{\sL^{1}(X)}\\
    &\leq \tfrac{\|v\|_{\sL^{\infty}(\R)}^{2}}{\uv}\mL_2\vert X\vert \vert t-\tilde{t}\vert \e^{T \mathcal{L}_{2}\|v\|_{\sL^{\infty}(\R)}}\\
    &\quad +\Big(1+t\|v\|_{\sL^{\infty}(\R)} \mL_{2}\exp\Big(t\|v\|_{\sL^{\infty}(\R)}\mL_{2}\Big)\Big)\tfrac{\|v\|_{\sL^{\infty}(\R)}^2}{\uv^2}\mL\vert \tt-t\vert \, \vert v\vert _{\sTV\left( X+\|v\|_{\sL^{\infty}(\R)}\mL(-1,1)\right),}
\end{align*}
    with
    \begin{align}
    \mL&\: \|\lambda\|_{\sL^{\infty}((0,T)\sk \sL^{\infty}(\R))},\qquad \mL_{2}\: \|\partial_{2}\lambda\|_{\sL^{\infty}((0,T)\sk \sL^{\infty}(\R))}.
\end{align}  
In particular, it holds that
\[
[0,T]\ni t \mapsto \Big( \R\ni y \mapsto \partial_{y}\cX[v,\lambda](y,t) \Big)\in \sC\big([0,T]\sk \sL^{1}_{\loc}(\R)\big).
\]
\end{cor}
\begin{proof}
This time, we prove the claim (with slightly different bounds) for discontinuous \(v\in \sTV_{\loc}(\R)\) by taking advantage of the surrogate system in \cref{theo:surrogate}.
Following the first steps in the previous proof of \cref{lem:stability_C_L_1}, we now only smooth the discontinuous velocity and use the notation
\[
x \equivd \cX[v,\lambda]\quad \text{ and }\quad \xe \equivd \cX[\ve,\lambda_{\eps}] \quad \text{ on } \R\times(0,T).
\]
Then, it holds that
\begin{align}
\eqref{eq:a_nice_equation} &= \int_{X} \bigg\vert  \tfrac{\xe(z+h\sk \tt)-\xe(z+h\sk t)}{h} - \tfrac{\xe(z\sk \tt)-\xe(z\sk t)}{h}\bigg\vert \dd z.
\intertext{Using that the solution is a Carath\'eodory solution as in \cref{theo:equivalence_cara_weak}}
&= \tfrac{1}{h}\int_{X} \bigg\vert  \int_t^{\tt}v_{\eps}(\xe(z+h\sk s))\lambda(s,\xe(z+h\sk s)) \notag\\
&\qquad\qquad\qquad\qquad - v_{\eps}(\xe(z\sk s))\lambda(s,\xe(z\sk s)) \dd s\bigg\vert \dd z\notag\\
\intertext{and applying the Fundamental theorem of integration}
& = \tfrac{1}{h}\int_{X} \bigg\vert  \int_{t}^{\tt} \int_{\xe(z\sk s)}^{\xe(z+h\sk s)} v'_{\eps}(y)\lambda(s,y) + v_{\eps}(y) \partial_2 \lambda(s,y) \dd y \dd s \bigg\vert \dd z\notag \\
& \leq \tfrac{1}{h}\int_{X} \bigg\vert  \int_{t}^{\tt} \int_{\xe(z\sk s)}^{\xe(z+h\sk s)} \big\vert v'_{\eps}(y)\lambda(s,y)\big\vert \dd s \bigg\vert \dd z \notag
\\
&\qquad + \tfrac{1}{h}\int_{X} \bigg\vert  \int_{t}^{\tt} \int_{\xe(z\sk s)}^{\xe(z+h\sk s)} \big\vert v_{\eps}(y) \partial_2 \lambda(s,y) \dd y \dd s\big\vert  \bigg\vert \dd z.\label{eq:all_or_nothing_revisited}
\end{align}
The second term in the previous estimate,  \cref{eq:all_or_nothing_revisited}, is estimated as follows 
\begin{align*}
    &\tfrac{1}{h}\int_{X} \bigg\vert  \int_{t}^{\tt} \int_{\xe(z\sk s)}^{\xe(z+h\sk s)} \big\vert v_{\eps}(y) \partial_2 \lambda(s,y) \dd y \dd s\big\vert  \bigg\vert \dd z\\
    &\leq \tfrac{\|v_{\eps}\|_{\sL^{\infty}(\R)}\|\partial_{2}\lambda\|_{\sL^{\infty}((0,T)\sk \sL^{\infty}(\R))}\vert X\vert \vert t-\tilde{t}\vert}{h}  \esssup\nolimits_{(z,s)\in \R\times (0,T)}\vert \xe(z+h\sk s)-\xe(z\sk s)\vert. 
    \intertext{By applying the stability estimate in \cref{cor:upper_lower_bounds_partial_3_X} and recalling that  \(\|v_{\eps}\|_{\sL^{\infty}(\R)}\leq \|v\|_{\sL^{\infty}(\R)}\), we obtain}
    &\leq \tfrac{\|v\|_{\sL^{\infty}(\R)}^{2}}{\uv}\mL_2\vert X\vert \vert t-\tilde{t}\vert \e^{T \mathcal{L}_{2}\|v\|_{\sL^{\infty}(\R)}}.
\end{align*}
Clearly, for \(t\rightarrow \tilde{t}\) the previous term converges to zero uniformly in \(h\) and \(\eps\).

The first term in \cref{eq:all_or_nothing_revisited} is more involved. However, thanks to the smoothing of \(v\) by \(v_{\eps}\), for now we can take advantage of this higher regularity and have
\begin{align}
    & \tfrac{1}{h}\int_{X} \bigg\vert  \int_{t}^{\tt}
    \int_{\xe(x_0\sk s)}^{\xe(x_0+h\sk s)} v'_{\eps}(y)\lambda(s,y) \dd y \dd s \bigg\vert \dd x_0. \\
\intertext{Using \cref{eq:surrogate_system1,eq:surrogate_system2} and \cref{eq:identity_surrogate_system}: \(\xe(x_{0}\sk \cdot)\equiv \cZ[v_{\eps}]^{-1}\big(x_0\sk \cC[\lambda,\cZ[v_{\eps}](x_0\sk \ast)](\cdot)\big)\) }
    &= \tfrac{1}{h}\int_{X} \bigg\vert  \int_{t}^{\tt} \int_{\cZ[v_{\eps}]^{-1}\big(x_0\sk \cC[\lambda,\cZ[v_{\eps}](x_0\sk \ast)](s)\big)}^{\cZ[v_{\eps}]^{-1}\big(x_0+h\sk \cC[\lambda,\cZ[v_{\eps}](x_0+h\sk \ast)](s)\big)} v'_{\eps}(y)\lambda(s,y) \dd y \dd s \bigg\vert \dd x_0. 
    \intertext{Substituting \(y=\cZ[v_{\eps}]^{-1}(x_{0}\sk u)\Rightarrow u=\cZ[v_{\eps}](x_{0}\sk y)\), we have for the derivative \(\tfrac{\dd}{\dd u}\cZ[v_{\eps}]^{-1}(x_{0}\sk u)=v_{\eps}\big(\cZ[v_{\eps}]^{-1}(x_{0}\sk u)\big)\). Thus}
    &= \tfrac{1}{h}\!\!\int_{X} \!\bigg\vert\!\!  \int_{t}^{\tt}\!\!\! \int_{\cC[\lambda,\cZ[v_{\eps}](x_0\sk \cdot)](s)}^{\cZ[v_{\eps}]\big(x_{0}\sk \cZ[v_{\eps}]^{-1}\big(x_0+h\sk \cC[\lambda,\cZ[v_{\eps}](x_0+h\sk \ast)](s)\big)\big)}\!\!\!\!\!\! a(u,s,x_{0}) \dd u \dd s \bigg\vert \dd x_0, \label{eq:2021}
\end{align}
with
\begin{equation}
    a(u,s,x_{0})\:v_{\eps}'\big(\cZ[v_{\eps}]^{-1}(x_{0}\sk u)\big)\lambda\big(s,\cZ[v_{\eps}]^{-1}(x_{0}\sk u)\big) v_{\eps}\big(\cZ[v_{\eps}]^{-1}(x_{0}\sk u)\big).
\end{equation}
Recalling that for \(x\in\R\) we can estimate
\begin{align}
\cZ[v_{\eps}]\big(x_{0}\sk \cZ[v_{\eps}]^{-1}(x_{0}+h\sk x)\big)&=
\cZ[v_{\eps}]\big(x_{0}\sk \cZ[v_{\eps}]^{-1}(x_{0}+h\sk x)\big)\\
&\quad -\cZ[v_{\eps}]\big(x_{0}+h\sk \cZ[v_{\eps}]^{-1}(x_{0}+h\sk x)\big)\label{eq:cZ_estimate_1}\\
&\quad +\cZ[v_{\eps}]\big(x_{0}+h\sk \cZ[v_{\eps}]^{-1}(x_{0}+h\sk x)\big)\notag\\
&=\cZ[v_{\eps}]\big(x_{0}\sk \cZ[v_{\eps}]^{-1}(x_{0}+h\sk x)\big)\\
&\quad -\cZ[v_{\eps}]\big(x_{0}+h\sk \cZ[v_{\eps}]^{-1}(x_{0}+h\sk x)\big)+x\notag
\intertext{and by \cref{eq:stability_z} in \cref{theo:existence_uniqueness_stability_surrogate}}
&\leq x+\tfrac{h}{\underline{v}}.\label{eq:cZ_estimate_2}
\end{align}
In addition, by \cref{eq:stability_c} together with \cref{eq:stability_z} in \cref{theo:existence_uniqueness_stability_surrogate}, we have for \(s\in[0,T]\)
\begin{align}
    &\cC[\lambda,\cZ[v_{\eps}](x_0+h\sk \cdot)](t) - \cC[\lambda,\cZ[v_{\eps}](x_0\sk \cdot)](t)\notag\\
    &\leq t\tfrac{\|v_{\eps}\|_{\sL^{\infty}(\R)}}{\uv} \mL_{2}\exp\Big(t\|v_{\eps}\|_{\sL^{\infty}(\R)}\mL_{2}\Big)h\leq \underbrace{t\tfrac{\|v\|_{\sL^{\infty}(\R)}}{\uv} \mL_{2}\exp\Big(t\|v\|_{\sL^{\infty}(\R)}\mL_{2}\Big)}_{\eqqcolon C(v,\mL_2,t)}h\\
    &= C(v,\mL_2,t)h.\label{eq:Lipschitz_cC}
\end{align}
Continuing the previous estimate, assuming without loss of regularity that \(\tt\geq t\) yields
\begin{align}
    \eqref{eq:2021}&\overset{\eqref{eq:cZ_estimate_1},\eqref{eq:cZ_estimate_2}}{\leq} \tfrac{1}{h}\int_{X}  \int_{t}^{\tt} \int_{\cC[\lambda,\cZ[v_{\eps}](x_0\sk \cdot)](s)}^{{\cC[\lambda,\cZ[v_{\eps}](x_0+h\sk \cdot)](s)}+\frac{h}{\underline{v}}}\big\vert a(u,s,x_{0})\big\vert \dd u \dd s \dd x_0\\
    &\overset{\eqref{eq:Lipschitz_cC}}{\leq} \tfrac{1}{h}\int_{X}  \int_{t}^{\tt} \int_{\cC[\lambda,\cZ[v_{\eps}](x_0\sk \cdot)](s)}^{{\cC[\lambda,\cZ[v_{\eps}](x_0\sk \cdot)](s)}+\tfrac{h}{\uv}\left(1+C(v,\mL_2,s)\right)}\big\vert  a(u,s,x_{0})\big\vert \dd u \dd s \dd x_0.\\
    \intertext{By \cref{eq:surrogate_system2} \( \big\vert \cC[\lambda,\cZ[v_{\eps}](x_0\sk \cdot)](s)\big\vert \leq T\|\lambda\|_{\sL^{\infty}((0,T)\sk \sL^{\infty}(\R))}\ \forall (x_{0},s)\in\R\times(0,T)\) }
    &\leq \tfrac{1}{h}\int_{X}  \int_{t}^{\tt}\!\! \sup_{y\in(-\mL T,\mL T)}\int_{y}^{y+\tfrac{h}{\uv}\left(1+C(v,\mL_2,s)\right)}\big\vert  a(u,s,x_{0})\big\vert \dd u \dd s \dd x_0\\
    &\leq \tfrac{\vert \tt-t\vert \left(1+C(v,\mL_2,T)\right)\|v_{\eps}\|_{\sL^{\infty}(\R)}\mL}{\uv}\hspace{-10pt}\sup_{u\in Z(v,\mL,\mL_{2},T,h)}\!\int_{X}\!\big\vert v_{\eps}'\big(\cZ[v_{\eps}]^{-1}(x_{0}\sk u)\big)\big\vert \dd x_{0},
\intertext{with \(Z(v,\mL,\mL_2,T,h) \: h\left(\frac{1}{\underline{v}}+C(v,\mL_2,T)\right)+\mL T(-1,1)\subset\R\)}
&\leq \vert \tt-t\vert\tfrac{\left(1+C(v,\mL_2,T)\right)\|v_{\eps}\|_{\sL^{\infty}(\R)}^{2}\mL}{\uv^{2}} \sup_{u\in Z(v,\mL,\mL_2,T,h)}\int_{\cZ[v_{\eps}]^{-1}(X\sk u)}\big\vert v_{\eps}'(y)\big\vert \dd y. \label{eq:the_final_problem}
\end{align}
In the previous estimate we have used the substitution
\begin{align*}
y&=\cZ[v_{\eps}]^{-1}(x_{0}\sk u)\Rightarrow
\cZ[v_{\eps}](x_{0}\sk y)=u\\
\Rightarrow & \  \tfrac{\dd}{\dd y}\cZ[v_{\eps}](x_{0}\sk y)=\tfrac{1}{v_{\eps}(y)}-\tfrac{\dd}{\dd y}x_{0}(y)\tfrac{1}{v_{\eps}(x_{0}(y))}=0 \ \Rightarrow \ \big\vert \tfrac{\dd}{\dd y}x_{0}(y)\big\vert \leq \tfrac{\|v_{\eps}\|_{\sL^{\infty}(\R)}}{\uv}.
\end{align*}
Thus, for $x\in X$ and $u \in Z(v,\mL,\mL_2,T,h)$ we can estimate
\begin{align*}
    \big\vert \cZ[v]^{-1}(x\sk u)-x\big\vert &= \big\vert \cZ[v]^{-1}(x\sk u)-\cZ[v]^{-1}(x\sk 0)\big\vert  \leq \|v\|_{\sL^{\infty}(\R)}\vert u\vert\\
    &\leq  \|v\|_{\sL^{\infty}(\R)}\Big(\mL T+\tfrac{h}{\uv}\big(1+C(v,\mL_2,T)\big)\Big).
\end{align*}
Consequently for \(h\rightarrow 0\)
\begin{align*}
  \cZ[v_{\eps}]^{-1}(X\sk u) \subset X +\|v\|_{\sL^{\infty}(\R)}\mL T(-1,1).
\end{align*}
Using this to further estimate in \cref{eq:the_final_problem}, we have
\begin{align*}
    \eqref{eq:the_final_problem}  &\leq \Big(1+C(v,\mL_2,T)\Big)\tfrac{\|v_{\eps}\|_{\sL^{\infty}(\R)}^{2}}{\uv^2}\mL\vert \tt-t\vert \, \vert v_{\eps}\vert _{\sTV\left( X+\|v\|_{\sL^{\infty}(\R)}\mL(-1,1)\right)}\\
    &\leq \Big(1+C(v,\mL_2,T)\Big)\tfrac{\|v\|_{\sL^{\infty}(\R)}^2}{\uv^2}\mL\vert \tt-t\vert \, \vert v\vert _{\sTV\left( X+\|v\|_{\sL^{\infty}(\R)}\mL(-1,1)\right)}.
\end{align*}
As the terms are all bounded, we can let \(\tt\rightarrow t\) and obtain -- together with the previous estimates -- the claimed continuity in time.
\end{proof}

\section{Analysis of discontinuous nonlocal conservation laws} \label{sec:nonlocal_well_posedness}
In this section, we leverage the theory established in \cref{sec:discontinuous_ODE} to obtain existence and uniqueness of weak solutions for the following class of nonlocal conservation laws with discontinuous (in space) velocity (as stated in \cref{defi:disc_conservation_law}). 
First, we state the assumptions on the involved datum:
\begin{ass}[Input datum -- discontinuous nonlocal conservation law]\label{ass:input_datum_conservation_laws}
For \(T\in\R_{>0}\) it holds that
\begin{itemize}
    \item \(q_{0}\in \sL^{\infty}(\R)\)
    \item \(\gamma\in \sBV(\R\sk \R_{\geq0})\) with \(\|\cdot\|_{\sBV(\R)}\: \|\cdot\|_{\sL^{1}(\R)}+\vert\cdot\vert_{\sTV(\R)}\).
    \item \(V\in \sW^{1,\infty}_{\loc}(\R)\)
    \item \(v\in \sL^{\infty}(\R\sk \R_{\geq \uv})\) for a \(\uv\in\R_{>0}\).
\end{itemize}
\end{ass}
As can be seen, the assumptions on \(V\) and \(v\) are identical to those for the discontinuous IVP in \cref{defi:disc_ODE} (compare with \cref{ass:input_datum}) and are not restrictive. The assumptions on the initial datum \(q_{0}\in\sL^{\infty}(\R)\) are relatively standard in the theory of conservation laws and the assumptions on the nonlocal kernel \(\gamma\) are also minimal (compare in particular with \cite{coclite2021existence}).

We use the classical definition of weak solutions as follows
\begin{defi}[Weak solution]\label{defi:weak_solution}
For the initial datum \(q_{0}\in\sL^{\infty}(\R)\), \(q\in \sC\big([0,T]\sk \sL^{1}_{\loc}(\R)\big)\cap\sL^{\infty}\big((0,T)\sk \sL^{\infty}(\R)\big)\) with datum as in \cref{ass:input_datum_conservation_laws} is called a \textbf{weak solution} of \cref{defi:disc_conservation_law} iff \(\forall \phi\in \sC^{1}_{\text{c}}((-42,T)\times\R)\) and it holds that
\begin{gather*}
\iint_{\OT}q(t,x)\Big(\phi_{t}(t,x)+v(x)V\big(\big(\gamma\ast q(t,\cdot)\big)(x)\big)\phi_{x}(t,x)\Big)\dd x\dd t\\
+\int_{\R}\phi(0,x)q_{0}(x)\dd x=0.
\end{gather*}
\end{defi}
\subsection{Well-posedness of solutions}\label{subsec:well_posed_nonlocal}
In this section, we will establish the existence and uniqueness of the weak solution to discontinuous nonlocal conservation laws. We start with existence for sufficiently small time horizons and use a reformulation in terms of a fixed-point problem. Such a reformulation has been used in various contributions dealing with nonlocal conservation laws, including \cite{coron,wang,sarkar,keimer2,wangshang,pflug,pflug2,spinola,pflug3,pflug4,crippa2013existence}.
\begin{theo}[Existence/uniqueness, weak solutions, small time horizon]\label{theo:existence_uniqueness_nonlocal_conservation_law}
There exists a time horizon \(T\in\R_{>0}\) such that the nonlocal, discontinuous conservation law in \cref{defi:disc_conservation_law} admits a unique weak solution (as in \cref{defi:weak_solution}) \[q\in \sC\big([0,T]\sk \sL^{1}_{\;\loc}(\R)\big)\cap \sL^{\infty}\big((0,T)\sk \sL^{\infty}(\R)\big).\]
The solution can be stated as
\begin{align}
q(t,x)=q_{0}(\xi_{w}(t,x\sk 0))\partial_{2}\xi_{w}(t,x\sk 0),&& (t,x)\in\OT\label{eq:solution_formula}
\end{align}
where \(\xi_{w}:\OT\times [0,T] \rightarrow \R\) is the unique solution of the IVP as in \cref{defi:weak} for \((t,x)\in\OT\)
\begin{equation}
\begin{aligned}
\partial_3 \xi_w(t,x\sk \tau) &= v\big(\xi_w(t,x\sk \tau)\big)V\big(w(\tau,\xi(t,x\sk \tau))\big),\qquad \tau\in(0,T)\\
\xi_{w}(t,x\sk t)&=x
\end{aligned}
\label{eq:char}
\end{equation}
and \(w\) is the solution of the fixed-point equation in \(\sL^{\infty}\big((0,T)\sk \sL^{\infty}(\R)\big)\)
\begin{align*}
w(t,x)&=\int_{\R} \gamma\big(x-\xi_{w}(0,y\sk t)\big)q_{0}(y)\dd y, && (t,x)\in \OT.
\end{align*}
\end{theo}
\begin{proof}
Define the fixed-point mapping
\begin{equation}
    F:\begin{cases}
    \sL^{\infty}\big((0,T)\sk \sW^{1,\infty}(\R)\big)&\rightarrow \sL^{\infty}\big((0,T)\sk \sW^{1,\infty}(\R)\big)\\
    w&\mapsto \Big((t,x)\mapsto \int_{\R}\gamma\big(x-\xi_{w}(0,y\sk t)\big)q_{0}(y)\dd y\Big)
    \end{cases}
    \label{defi:fixed_point}
\end{equation}
with \(\xi_{w}\) the characteristics as defined in \cref{eq:char}.
Let us first look into the well-posedness of these characteristics. Given that \(w\in \sL^{\infty}\big((0,T)\sk \sW^{1,\infty}(\R)\big)\) and recalling \cref{ass:input_datum_conservation_laws}, we can invoke \cref{theo:surrogate} to demonstrate that \(\xi_{w}\) is uniquely determined by \(w\).
Next, we show that \(F\) is a fixed-point mapping on the proper subset of \(\sL^{\infty}\big((0,T)\sk \sW^{1,\infty}(\R)\big)\).
To this end, define
\begin{align}
    M&\:42\|\gamma\|_{\sL^{1}(\R)}\|q_{0}\|_{\sL^{\infty}(\R)}\tfrac{\|v\|_{\sL^{\infty}(\R)}}{\uv}, \label{eq:defi_M}\\
    M'&\:42\vert\gamma\vert_{\sTV(\R)}\|q_{0}\|_{\sL^{\infty}(\R)}\tfrac{\|v\|_{\sL^{\infty}(\R)}}{\uv},\notag\\
    \Omega_{M}^{M'}(T)&\: \Big\{w\in \sL^{\infty}\big((0,T)\sk \sW^{1,\infty}(\R)\big): \|w\|_{\sL^{\infty}((0,T)\sk \sL^{\infty}(\R))}\leq M \notag \\
    &\qquad\qquad\qquad\qquad\qquad\qquad\qquad
    \wedge \|\partial_{2}w\|_{\sL^{\infty}((0,T)\sk \sL^{\infty}(\R))}\leq M'\Big\}.\label{eq:Omega}
\end{align}
\begin{description}
\item[Self-mapping:] 
Taking \(w\in\Omega_{M}^{M'}(T_{1})\) for a \(T_{1}\in(0,T]\), we estimate for \(t\in(0,T_{1}]\) 
\begin{align*}
    &\|F[w]\|_{\sL^{\infty}((0,t)\sk \sL^{\infty}(\R))}\leq \|q_{0}\|_{\sL^{\infty}(\R)}\|\gamma\|_{\sL^{1}(\R)}\|\partial_{2}\xi_{w}(t,\cdot\sk 0)\|_{\sL^{\infty}(\R)}\\
    &\overset{\cref{eq:improved_bound_partial_3_X}}{\leq}\|q_{0}\|_{\sL^{\infty}(\R)}\|\gamma\|_{\sL^{1}(\R)}\tfrac{\|v\|_{\sL^{\infty}(\R)}}{\uv}\e^{t\|v\|_{\sL^{\infty}(\R)}\|V'\|_{\sL^{\infty}((-M,M))}M'.}
\end{align*}
Here we have used the stability estimate of the IVP in \cref{cor:upper_lower_bounds_partial_3_X} to uniformly estimate the spatial derivative of the characteristics.
Thus, \(\|F[w]\|_{\sL^{\infty}((0,T_{1})\sk \sL^{\infty}(\R))}\leq M\) holds if
\begin{equation}
\e^{T_{1}\|v\|_{\sL^{\infty}(\R)}\|V'\|_{\sL^{\infty}((-M,M))}M'}\leq 42.\label{eq:only_one_time_required}
\end{equation}
We then pick the maximal \(T_{1}\in (0,T]\) satisfying this inequality.

Next, consider the spatial derivative of \(F\) on the time interval \(T_{2}\in(0,T_{1}]\) and this time choose \(w\in \Omega_{M}^{M'}(T_{2})\). Analogous to the previous estimate, we estimate for \(t\in(0,T_{2}]\) 
\begin{align*}
    &\|\partial_{2}F[w]\|_{\sL^{\infty}((0,t)\sk \sL^{\infty}(\R))}\leq \|q_{0}\|_{\sL^{\infty}(\R)}\vert\gamma\vert_{\sTV(\R)} \|\partial_{2}\xi_{w}(t,\cdot\sk 0)\|_{\sL^{\infty}(\R)}\\
    &\overset{\cref{eq:improved_bound_partial_3_X}}{\leq}\|q_{0}\|_{\sL^{\infty}(\R)}\vert\gamma\vert_{\sTV(\R)} \tfrac{\|v\|_{\sL^{\infty}(\R)}}{\uv}\e^{t\|v\|_{\sL^{\infty}(\R)}\|V'\|_{\sL^{\infty}((-M,M))}M'.}
\end{align*}
Here we have used the stability estimate of the IVP in \cref{cor:upper_lower_bounds_partial_3_X} to uniformly estimate the spatial derivative of the characteristics.
Thus, \(\|\partial_{2}F[w]\|_{\sL^{\infty}((0,T_{2})\sk \sL^{\infty}(\R))}\leq M'\) holds if 
\begin{equation}
\e^{T_{2}\|v\|_{\sL^{\infty}(\R)}\|V'\|_{\sL^{\infty}((-M,M))}M'}\leq 42.\label{eq:one_time_use_only} \end{equation}
As this is identical to the condition in \cref{eq:only_one_time_required}, we can indeed pick \(T_{2}=T_{1}\) as our considered time horizon. Based on the previous estimates, we thus have a self-mapping on the considered time horizon, i.e.\ 
\[
F\Big(\Omega_{M}^{M'}\big(T_{1}\big)\Big)\subseteq\Omega_{M}^{M'}\big(T_{1}\big).
\]
\item[Contraction:] Next, we show that the mapping \(F\) is a contraction for a yet to be determined \(T_{3}\in(0,T_{1}]\) in \(\sL^{\infty}((0,T_{3})\sk \sL^{\infty}(\R))\). To this end, take \(w,\tilde{w}\in \Omega_{M}^{M'}(T_{3})\) and estimate for \((t,x)\in [0,T_{1}]\times\R\)
\begin{align}
    &\vert F[w](t,x)-F[\tilde{w}](t,x)\vert \notag\\
    &\leq \|q_{0}\|_{\sL^{\infty}(\R)}\int_{\R} \vert \gamma(x-\xi_{w}(0,y\sk t))-\gamma(x-\xi_{\tilde{w}}(0,y\sk t))\vert \dd y\label{eq:4242}\\
    &\overset{\eqref{eq:improved_bound_partial_3_X}}{\leq} \tfrac{\|q_{0}\|_{\sL^{\infty}(\R)}\vert\gamma\vert_{\sTV(\R)}\|v\|_{\sL^{\infty}(\R)}}{\uv}\|\xi_{w}-\xi_{\tilde{w}}\|_{\sL^{\infty}((0,t)\times\R\times(0,t))}\e^{t\|v\|_{\sL^{\infty}(\R)}\|V'\|_{\sL^{\infty}((-M,M))}M'}\notag
\end{align}
In the last estimate we have used the following: 
\begin{enumerate}
    \item For \(f\in \sTV(\R)\cap\sL^{\infty}(\R)\) and diffeomorphisms \(g,h\) it holds (for the proof see for instance \cite[Lemma 2.4]{coron_pflug}) \label{item:1}
\[
\|f\circ g-f\circ h\|_{\sL^{1}(\R)}\leq \vert f\vert_{\sTV(\R)}\big\|g^{-1}-h^{-1}\big\|_{\sL^{\infty}(\R)}.
\]
It also holds for the characteristics that
\begin{align}
    \xi(t,\xi(\tau,x\sk t)\sk \tau)=x\quad \forall (t,x,\tau)\in\OT\times(0,T),\label{eq:characteristics_inverse_characteristics}
\end{align}
meaning that the inverse of the mapping \(x\mapsto \xi(t,x;\tau)\) is the mapping \(x\mapsto \xi(\tau,x;t)\).
This can be shown by approximating \(\xi\) by \(\xi_{\eps}\) with smooth \(v_{\eps},\lambda_{\eps}\) and the claim that for \(\eps\in\R_{>0}\) it holds that 
\[
  \xi_{\eps}(t,\xi_{\eps}(\tau,x\sk t)\sk \tau)=x\quad \forall x\in\R,\ (t,\tau)\in[0,T].
\]
However, this result was carried out in \cite[Lemma 2.6 Item 1]{pflug}. As we have the strong convergence of \(\xi_{\eps}\) to \(\xi\) by \cref{theo:stability}, this carries over to \cref{eq:characteristics_inverse_characteristics}.
\item We have identified in \cref{theo:stability} that
\begin{align*}
\lambda&\equiv V(w),\ \tilde{\lambda}\equiv V(\tilde{w}) \text{ and }\\ \mL_{2}&\:\max\Big\{\|\partial_{2}\lambda\|_{\sL^{\infty}((0,T_{3})\sk \sL^{\infty}(\R))},\|\partial_{2}\tilde{\lambda}\|_{\sL^{\infty}((0,T_{3})\sk \sL^{\infty}(\R))}\Big\}\\
&\leq\|V'\|_{\sL^{\infty}((-M,M))}\cdot\max\Big\{\|\partial_{2}w\|_{\sL^{\infty}((0,T_{3})\sk\sL^{\infty}(\R))},\|\partial_{2}\tilde{w}\|_{\sL^{\infty}((0,T_{3})\sk\sL^{\infty}(\R))}\Big\}\\
&\leq \|V'\|_{\sL^{\infty}((-M,M))}M'
\end{align*} 
where we have used the fact that \(w,\tilde{w}\in \Omega^{M'}_{M}(T_{3})\) and in particular \cref{eq:Omega}. 
\end{enumerate}
To obtain an estimate of \(\|\xi_{w}-\xi_{\tilde{w}}\|_{\sL^{\infty}((0,t)\times\R\times(0,t))}\) in terms of \(\|w-\tilde{w}\|_{\sL^{\infty}((0,T_{3})\sk \sL^{\infty}(\R))}\), we can again take advantage of \cref{theo:stability}, which yields
\begin{align*}
    \eqref{eq:4242}&\leq M'\|\xi_{w}-\xi_{\tilde{w}}\|_{\sL^{\infty}((0,t)\times\R\times(0,t))}\\
    &\overset{\eqref{eq:stability_x}}{\leq} M'\|v\|_{\sL^{\infty}(\R)}\e^{t\|v\|_{\sL^{\infty}(\R)}\|V'\|_{\sL^{\infty}((-M,M))}M'}\int_{0}^{t}\|V(w(s,\cdot))-V(\tilde{w}(s,\cdot))\|_{\sL^{\infty}(\R)}\dd s\\
    &\leq M'\|v\|_{\sL^{\infty}(\R)}t\|V'\|_{\sL^{\infty}((-M,M))}\e^{t\|v\|_{\sL^{\infty}(\R)}\|V'\|_{\sL^{\infty}((-M,M))}M'}\|w-\tilde{w}\|_{\sL^{\infty}((0,t)\sk \sL^{\infty}(\R))}.
\end{align*}
Reconnecting to \cref{eq:4242}, we thus have for small enough time \(T_{3}\in (0,T_{1}]\) (recall that in the previous estimate the right hand side consists of constants except for the term \( \|w-\tilde{w}\|_{\sL^{\infty}((0,t)\sk \sL^{\infty}(\R))}\))
\[
\|F[w]-F[\tilde{w}]\|_{\sL^{\infty}((0,T_{3})\sk \sL^{\infty}(\R))}\leq \tfrac{1}{2}\|w-\tilde{w}\|_{\sL^{\infty}((0,T_{3})\sk \sL^{\infty}(\R))},
\]
i.e., \(F\) is a contraction in \(\sL^{\infty}((0,T_{3})\sk \sL^{\infty}(\R))\).
\item[Concluding the fixed-point argument:] As \(M,M'\in\R_{>0}\) are fixed, we have proven \(F\) to be a self-mapping on \(\Omega_{M}^{M'}\big(T_{3}\big)\), and \(\Omega_{M}^{M'}\big(T_{3}\big)\)  -- thanks to the uniform bound \(M\) on the functions and \(M'\) on their spatial derivatives -- is closed in the topology induced by \(\sL^{\infty}((0,T_{3})\sk \sL^{\infty}(\R))\), we can apply Banach's fixed-point theorem \cite[Theorem 1.a]{zeidler} and obtain
\begin{equation}
\exists!\ w^{*}\in \Omega_{M}^{M'}\big(T_{3}\big):\ F[w^{*}]\equiv w^{*} \text{ on } (0,T_{3})\times\R.
\label{eq:fixed_point_w_star}
\end{equation}
\item[Constructing a solution to the conservation law:] Having obtained the existence and uniqueness of \(w^{*}\) as a fixed-point on a small time horizon, we use the method of characteristics as carried out in \cite[Theorem 2.20]{pflug} to state the solution as
\begin{equation}
    q(t,x)=q_{0}(\xi_{w^{*}}(t,x\sk 0))\partial_{2}\xi_{w^{*}}(t,x\sk 0), \qquad  (t,x)\in (0,T_{3})\times\R.\label{eq:solution_nonlocal}
\end{equation}
Note that due to \cref{theo:stability}, \(x\mapsto \xi_{w^{*}}(t,x\sk 0)\) is Lipschitz-continuous and strictly monotone increasing by \cref{cor:upper_lower_bounds_partial_3_X}. By \cref{lem:stability_C_L_1} \(\partial_{2}\xi_{w^{*}}(\ast,\cdot\sk 0)\in \sC\big([0,T_{3}];\sL^{1}_{\loc}(\R)\big)\) so that \(q\in \sC\big([0,T_{3}];\sL^{1}_{\loc}(\R)\big)\). The fact that \(q\in \sL^{\infty}\big((0,T_{3});\sL^{\infty}(\R)\big)\) is a direct consequence of \cref{eq:solution_formula}. It can easily be checked that \(q\) as in \cref{eq:solution_nonlocal} is a solution by plugging it into the definition of weak solutions in \cref{defi:weak_solution} and applying the substitution rule.

The uniqueness of solutions is more involved, but ultimately only an adaption of the proof in \cite{wang} and adjusted in \cite[Theorem 3.2]{pflug}.
Therefore, we only sketch the idea: In a first step we show by a proper choice of test functions in \cref{defi:weak_solution} that each weak solution can be stated in the form of \cref{eq:solution_nonlocal} with a proper nonlocal term \(w\). The next step is then to show that for the thus constructed solution, the nonlocal term satisfies the same fixed-point mapping as introduced in \cref{defi:fixed_point}. However, as this mapping has a unique fixed-point as we have proven previously, we have shown the uniqueness and are done.
\end{description}
\end{proof}
The previous result only demonstrates the existence of solutions on a small time horizon. Given the later \cref{ass:input_datum_additional}, we can show that even for general discontinuities in \(v\), i.e.,\ \(v\in \sL^{\infty}(\R\sk \R_{\geq \uv})\), the solution remains bounded on every finite time horizon. This is a key ingredient for extending the solution from small time in \cref{theo:existence_uniqueness_nonlocal_conservation_law} to arbitrary times.

To prove a weakened form of a maximum principle, we require the solutions to be smooth. Consequently, we first introduce the following (weak) stability result:
\begin{theo}[Weak stability of $q$ w.r.t.\ discontinuous velocity $v$, velocity $V$ and initial datum $q_0$]\label{theo:weak_stability}
Let \cref{ass:input_datum_conservation_laws} hold. Denote by 
\[
\{v_{\eps}\}_{\eps\in\R_{>0}}\subset \sC^{\infty}(\R),\ \big\{V_{\eps}\big\}_{\eps\in\R_{>0}}\subset\sC^{\infty}(\R) \text{ and } \{q_{0,\eps}\}_{\eps\in\R_{>0}}
\] the mollified versions of \(v,V,q_{0}\) convoluted with the standard mollifier outlined as in \cref{cor:ODE_approximation_smooth}.
Denote by \(T^{*}\in(0,T]\) the minimal time horizon of existence for the solution \(q,\{q_{\eps}\}_{\eps\in\R_{>0}}\) (where \(q\) is the solution to initial datum \(q_{0}\), discontinuous velocity \(v\) and Lipschitz velocity \(V\) and \(q_{\eps}\) the solution to initial datum \(q_{0,\eps}\), discontinuous velocity \(v_{\eps}\) and Lipschitz velocity \(V_{\eps}\)) as guaranteed in \cref{theo:existence_uniqueness_nonlocal_conservation_law}. Then, it holds that
  \begin{equation}
     \forall g\in \sC_{\text{c}}(\R) : \lim_{\eps\rightarrow 0}\max_{t\in[0,T^{*}]}\bigg\vert \int_{\R}\big( q(t,x)-q_{\eps}(t,x)\big) g(x)\dd x\bigg\vert=0.
     \label{eq:weak_continuity_solution}
  \end{equation}
 \end{theo}
\begin{proof}

We start by showing that such a time horizon \(T^{*}\) exists uniformly in \(\eps\). Recalling the proof of \cref{theo:existence_uniqueness_nonlocal_conservation_law}, the properties of the standard mollifier for each \(\eps\in\R_{>0}\) enable us to define the upper bounds on the nonlocal term as in \cref{eq:defi_M}
\begin{equation}
\begin{aligned}
    M_{\eps}&\:42\|\gamma\|_{\sL^{1}(\R)}\|q_{0,\eps}\|_{\sL^{\infty}(\R)}\tfrac{\|v_{\eps}\|_{\sL^{\infty}(\R)}}{\uv}&&\leq 42\|\gamma\|_{\sL^{1}(\R)}\|q_{0}\|_{\sL^{\infty}(\R)}\tfrac{\|v\|_{\sL^{\infty}(\R)}}{\uv}\eqqcolon M\\
    M_{\eps}'&\:42\vert\gamma\vert_{\sTV(\R)}\|q_{0,\eps}\|_{\sL^{\infty}(\R)}\tfrac{\|v_{\eps}\|_{\sL^{\infty}(\R)}}{\uv}&&\leq 42\vert\gamma\vert_{\sTV(\R)}\|q_{0}\|_{\sL^{\infty}(\R)}\tfrac{\|v\|_{\sL^{\infty}(\R)}}{\uv}\eqqcolon M'.
\end{aligned}
\label{eq:M_revisited}
\end{equation}
However, this means that we can take as upper bounds uniformly \(M,M'\). Looking into the self-mapping condition in \cref{eq:only_one_time_required}, it then reads in our case
\[
\exp\big(T_{1}\|v_{\eps}\|_{\sL^{\infty}(\R)}\|V_{\eps}'\|_{\sL^{\infty}((-M,M))}M'\big)\leq  42,
\]
which can also be replaced by the stronger form
\[
\exp\big(T_{1}\|v\|_{\sL^{\infty}(\R)}\|V'\|_{\sL^{\infty}((-M,M))}M'\big)\leq  42.
\]
Now choosing \(T_{1}\) to satisfy the previous inequality this is by construction \(\eps\) invariant. The identical argument can be made for the estimate in \cref{eq:only_one_time_required}, so that for the chosen \(T_{1}\) the mapping \(F\) in \cref{defi:fixed_point} is a self-mapping on \(\Omega_{M}^{M'}(T_{1})\), as in \cref{eq:Omega}. So the only point that remains is to check whether the fixed-point mapping is also a contraction uniformly in \(\eps\in\R_{>0}\) for a small time horizon. Recollecting the contraction estimate starting in \cref{eq:4242}, we have for \(T_{2}\in (0,T_{1}]\)
\begin{align*}
    &\|F[w]-F[\tilde{w}]\|_{\sL^{\infty}((0,T_{2});\sL^{\infty}(\R))}\\
    &\leq \|v_{\eps}\|_{\sL^{\infty}(\R)}T_{2}\|V_{\eps}'\|_{\sL^{\infty}((-M,M))}M'\!\e^{T_{2}\|v_{\eps}\|_{\sL^{\infty}(\R)}\|V_{\eps}'\|_{\sL^{\infty}((-M,M))}M'}\!\!\|w-\tilde{w}\|_{\sL^{\infty}((0,T_{2})\sk \sL^{\infty}(\R))}\\
    &\leq \|v\|_{\sL^{\infty}(\R)}T_{2}\|V'\|_{\sL^{\infty}((-M,M))}M'\!\e^{T_{2}\|v\|_{\sL^{\infty}(\R)}\|V'\|_{\sL^{\infty}((-M,M))}M'}\!\!\|w-\tilde{w}\|_{\sL^{\infty}((0,T_{2})\sk \sL^{\infty}(\R))}.
\end{align*}
Again choosing \(T_{2}\in (0,T_{1}]\) so that 
\[
\|v\|_{\sL^{\infty}(\R)}T_{2}\|V'\|_{\sL^{\infty}((-M,M))}M'\!\e^{T_{2}\|v\|_{\sL^{\infty}(\R)}\|V'\|_{\sL^{\infty}((-M,M))}M'}\leq \tfrac{1}{2}
\]
is invariant on \(\eps\in\R_{>0}\), we can set \(T^{*}\:T_{2}\in\R_{>0}\) and have found the time horizon on which the existence of solutions is guaranteed for all \(\eps\in\R_{>0}\) simultaneously.

Next, we prove the claimed continuity as stated in \cref{eq:weak_continuity_solution}.
We recall the solution formula in \cref{eq:solution_formula} and obtain for a \(g\in \sC_{\loc}(\R)\) and \(t\in[0,T^{*}]\) and for \(\eps\in\R_{>0}\)
\begin{align}
&\bigg\vert \int_{\R}\big( q(t,x)-q_{\eps}(t,x)\big) g(x)\dd x\bigg\vert \notag\\
&=\bigg\vert \int_{\R}\Big( q_{0}(\xi_{w}(t,x\sk 0))\partial_{2}\xi_{w}(t,x\sk 0)-q_{0,\eps}(\xi_{\eps,w_{\eps}}(t,x\sk 0))\partial_{2}\xi_{\eps,w_{\eps}}(t,x\sk 0)\Big) g(x)\dd x\bigg\vert \notag\\
&=\bigg\vert \int_{\R}q_{0}(y)g\big(\xi_{w}(0,y\sk t)\big)-q_{0,\eps}(y)g\big(\xi_{\eps,w_{\eps}}(0,y\sk t)\big)\dd y\bigg\vert \notag\\
&\leq \bigg\vert \int_{\R}\big(q_{0}(y)-q_{0,\eps}(y)\big)g\big(\xi_{w}(0,y\sk t)\big)\dd y\bigg\vert \notag\\
&\quad + \bigg\vert \int_{\R}q_{0,\eps}(y)\Big(g\big(\xi_{w}(0,y\sk t)\big)-g\big(\xi_{\eps,w_{\eps}}(0,y\sk t)\big)\Big)\dd y\bigg\vert \notag\\
&\leq \|g\|_{\sL^{\infty}(\R)}\|q_{0}-q_{0,\eps}\|_{\sL^{1}(\supp(g)+\|v\|_{\sL^{\infty}(\R)}\mL T(-1,1))}\notag\\
&\quad +\|q_{0}\|_{\sL^{\infty}(\R)}\big\|g\circ \xi_{w}(0,\cdot\sk t)-g\circ \xi_{\eps,w_{\eps}}(0,\cdot\sk t)\big\|_{\sL^{1}(\R)}.\label{eq:estimate_weak_star_convergence_2}
\end{align}
Recalling the bounds on the nonlocal term in \cref{eq:M_revisited}
as detailed in \cref{eq:Omega} for \(s\in[0,T^{*}]\), it holds by \cref{rem:stability_weak} and in particular \cref{eq:stability_weak} together with \cref{theo:stability} that
\begin{align}
&\|\xi_{\eps,w_{\eps}}(s,\cdot\sk  \ast)-\xi_{w}(s,\cdot\sk \ast)\|_{\sL^{\infty}((0,T^{*})\sk \sL^{\infty}(\R))}\notag\\
&\leq \|v\|_{\sL^{\infty}(\R)}\e^{T^{*}\|v\|_{\sL^{\infty}(\R)}\mathcal L_2}\int_{0}^{T^{*}}\|V(w(s,\cdot))-V_{\eps}(w_{\eps}(s,\cdot))\|_{\sL^{\infty}(\R)}\dd s\notag\\
    &\quad +\Big(\|v\|_{\sL^{\infty}(\R)}\e^{T^{*}\|v\|_{\sL^{\infty}(\R)}\mathcal L_2}T^{*}\mL_{2}+1\Big)\tfrac{2\|v\|_{\sL^{\infty}(\R)}}{\uv^{2}}\sup_{y\in\R}\Big\vert\int_{0}^{y}v(s)-v_{\eps}(s)\dd s\Big\vert\notag\\
    &\leq \|v\|_{\sL^{\infty}(\R)}\e^{T^{*}\|v\|_{\sL^{\infty}(\R)}\mathcal L_2}\|V'\|_{\sL^{\infty}((-M,M))}\int_{0}^{T^{*}}\|w(s,\cdot)-w_{\eps}(s,\cdot)\|_{\sL^{\infty}(\R)}\dd s\label{eq:xi_w_xi_tilde_w}\\
    &\quad + \|v\|_{\sL^{\infty}(\R)}\e^{T^{*}\|v\|_{\sL^{\infty}(\R)}\mathcal L_2}T^{*}\|V-V_{\eps}\|_{\sL^{\infty}((-M,M))}\notag\\
    &\quad +\Big(\|v\|_{\sL^{\infty}(\R)}\e^{T^{*}\|v\|_{\sL^{\infty}(\R)}\mathcal L_2}T^{*}\mL_{2}+1\Big)\tfrac{2\|v\|_{\sL^{\infty}(\R)}}{\uv^{2}}\sup_{y\in\R}\Big\vert\int_{0}^{y}v(s)-v_{\eps}(s)\dd s\Big\vert.\notag
\end{align}
Applying the fixed-point identity \cref{defi:fixed_point} and \cref{eq:fixed_point_w_star}, we end up with
    \begin{align*}
   &\int_{0}^{T^{*}}\|w(s,\cdot)-w_{\eps}(s,\cdot)\|_{\sL^{\infty}(\R)}\dd s\\
   &= \int_{0}^{T^{*}}\Big\| \int_{\R}\!\gamma\big(\cdot-\xi_{w}(0,y\sk s)\big)q_{0}(y) -\gamma\big(\cdot-\xi_{\eps,w_{\eps}}(0,y\sk s)\big)q_{0,\eps}(y)\dd y\Big\|_{\sL^{\infty}(\R)}\dd s\\
   &\leq \int_{0}^{T^{*}}\Big\| \int_{\R}\!\gamma\big(\cdot-\xi_{w}(0,y\sk s)\big)(q_{0}(y)-q_{0,\eps}(y))\dd y\Big\|_{\sL^{\infty}(\R)}\\
   &\quad + \Big\|\int_{\R}\big(\gamma\big(\cdot-\xi_{w}(0,y\sk s)\big)-\gamma\big(\cdot-\xi_{\eps,w_{\eps}}(0,y\sk s)\big)\big)q_{0,\eps}(y)\dd y\Big\|_{\sL^{\infty}(\R)}\dd s\\
   &\leq \int_{0}^{T^{*}}\Big\| \int_{\R}\!\gamma\big(\cdot-\xi_{w}(0,y\sk s)\big)(q_{0}(y)-q_{0,\eps}(y))\dd y\Big\|_{\sL^{\infty}(\R)}\\
   &\quad + \Big\|\int_{\R}\big(\gamma\big(\cdot-\xi_{w}(0,y\sk s)\big)-\gamma\big(\cdot-\xi_{\eps,w_{\eps}}(0,y\sk s)\big)\big)q_{0,\eps}(y)\dd y\Big\|_{\sL^{\infty}(\R)}\dd s\\
   &\leq \int_{0}^{T^{*}}\Big\| \int_{\R}\!\gamma'\big(\cdot-\xi_{w}(0,y\sk s)\big)\partial_{2}\xi_{w}(0,y\sk s)\int_{0}^{y}(q_{0}(z)-q_{0,\eps}(z))\dd z\dd y\Big\|_{\sL^{\infty}(\R)}\\
 &\quad +\|q_{0}\|_{\sL^{\infty}(\R)}\vert\gamma\vert_{\sTV(\R)}\int_{0}^{T^{*}}\!\!\!\!\big\|\xi_{w}(s,\cdot\sk \ast)- \xi_{\eps,w_\eps}(s,\cdot\sk \ast)\big\|_{\sL^{\infty}((0,T)\sk \sL^{\infty}(\R))}\dd s\\
   &\leq \vert\gamma\vert_{\sTV(\R)}T^{*}\sup_{y\in\R}\Big\vert\int_{0}^{y}q_{0}(y)-q_{0,\eps}(y)\dd y\Big\vert\\
    &\quad +\|q_{0}\|_{\sL^{\infty}(\R)}\vert\gamma\vert_{\sTV(\R)}\int_{0}^{T^{*}}\!\!\!\!\big\|\xi_{w}(s,\cdot\sk \ast)- \xi_{\eps,w_\eps}(s,\cdot\sk \ast)\big\|_{\sL^{\infty}((0,T)\sk \sL^{\infty}(\R))}\dd s.
\end{align*}
In the last estimate we have again used what was described in \cref{item:1} in the proof of \cref{theo:existence_uniqueness_nonlocal_conservation_law} and the assumptions on the involved datum \cref{ass:input_datum_conservation_laws}, particularly \(\gamma\in \sBV(\R)\).

As \(T^{*}\) was arbitrary (but small enough so that solutions still exist), we can apply Gr\"onwall's inequality \cite[Chapter I, III Gronwall's inequality]{walter} and, recollecting all previous terms, obtain
\begin{align*}
&\|\xi_{\eps,w_{\eps}}-\xi_{w}\|_{\sL^{\infty}((0,T^{*})^{2},\sL^{\infty}(\R))}\\
&\leq \bigg(\Big(\|v\|_{\sL^{\infty}(\R)}\e^{T^{*}\|v\|_{\sL^{\infty}(\R)}\mathcal L_2}T^{*}\mL_{2}+1\Big)\tfrac{2\|v\|_{\sL^{\infty}(\R)}}{\uv^{2}}\sup_{y\in\R}\Big\vert\int_{0}^{y}v(s)-v_{\eps}(s)\dd s\Big\vert\\
&\quad +  \|v\|_{\sL^{\infty}(\R)}\e^{T^{*}\|v\|_{\sL^{\infty}(\R)}\mathcal L_2}\|V'\|_{\sL^{\infty}((-M,M))}\|\gamma\|_{\sL^{\infty}(\R)}T^{*}\sup_{y\in\R}\Big\vert\int_{0}^{y}q_{0}(z)-q_{0,\eps}(z)\dd z\Big\vert\\
&\quad + \|v\|_{\sL^{\infty}(\R)}\e^{T^{*}\|v\|_{\sL^{\infty}(\R)}\mathcal L_2}T^{*}\|V-V_{\eps}\|_{\sL^{\infty}((-M,M))}
\bigg)\\
&\ \cdot \|v\|_{\sL^{\infty}(\R)}\e^{T^{*}\|v\|_{\sL^{\infty}(\R)}\mathcal L_2}\|V'\|_{\sL^{\infty}((-M,M))}\|q_{0}\|_{\sL^{\infty}(\R)}\vert\gamma\vert_{\sTV(\R)}.
\end{align*}
However, this means that \(\xi_{w}-\xi_{\eps,w_{\eps}}\) is small in the uniform topology for \(\eps\in\R_{>0}\) small.

Thanks to the lower bounds on the spatial derivatives on \(\xi_{w},\xi_{\eps,w_{\eps}}\), i.e., thanks to the fact that they are diffeomorphisms in space with a Lipschitz constant from below which is greater than zero (see \cref{cor:upper_lower_bounds_partial_3_X}), we can apply
\cref{lem:convergence_composition_L_1} on \cref{eq:estimate_weak_star_convergence_2} and obtain the claimed continuity in \cref{eq:weak_continuity_solution}.
\end{proof}
The previous theory enables to have smooth solutions when assuming smooth initial datum and smooth velocities. Even more, we can later use the previous approximation to derive bounds on the smoothed solution (as considered in the following \cref{lem:classical_solutions}). These bounds carry over to the weak solutions. Let us also state that this smoothness of solutions is in line with the regularity results in \cite{pflug}.
\begin{lem}[Smooth solutions for smooth datum]\label{lem:classical_solutions}
Let \cref{ass:input_datum_conservation_laws} hold. In addition,
\[
q_{0}\in \sC^{\infty}(\R),\ v\in \sC^{\infty}(\R),\ V\in \sC^{\infty}(\R).
\]
Then, there exists \(T^{*}\in\R_{>0}\) so that the weak solution \[q\in\sC\big([0,T^{*}];\sL^{1}_{\;\loc}(\R)\big)\cap \sL^{\infty}\big((0,T^{*});\sL^{\infty}(\R)\big)\] in \cref{defi:weak_solution} of the (now continuous) nonlocal conservation law in \cref{defi:disc_conservation_law} is a classical solution and
\[
q\in \sC^{\infty}(\Omega_{T^{*}}).
\]
\end{lem}
\begin{proof}
From \cref{theo:existence_uniqueness_nonlocal_conservation_law} we know that there exists a solution on an assured time horizon \([0,T^{*}]\) with a sufficiently small \(T^{*}\in\R_{>}\). Due to the regularity of the involved functions, we can take advantage of the fixed-point equation in \cref{defi:fixed_point} as follows. As \(q_{0}\) is smooth, the convolution means that the solution of the fixed-point problem is smooth provided the characteristics \(\xi_{w}\) do not destroy regularity. However, \(\xi_{w}\) as in \cref{eq:char} is -- for given \(w\) a smooth solution to the fixed-point problem -- the solution of an IVP with a smooth right hand side and is thus smooth. This explains why the nonlocal term \(w\) is smooth, the characteristics are smooth in each component and finally, looking at the solution formula in \cref{eq:solution_nonlocal}, the solution \(q\) is also smooth. This solution therefore satisfies the PDE point-wise and is a classical solution.
\end{proof}
\begin{rem}[Regularity of solutions]
It is possible -- similar to the results in \cite[Section 5]{pflug} -- to obtain  regularity results in \(\sW^{k,p}\) for properly chosen initial datum and velocities and \((k,p)\in \N_{\geq0}\times\big(\R_{\geq 1}\cup\{\infty\}\big)\) instead of \(\sC^{\infty}\) solutions as in \cref{lem:classical_solutions}. However, we do not go into details as we only require smooth solutions in the following analysis.
\end{rem}

\subsection{Maximum principles}\label{subsec:maximum_principle}
First we will list some assumptions that are particularly interesting for traffic flow modelling. They are inspired by classical maximum principles as laid out in \cite{scialanga,pflug}:
\begin{ass}\label{ass:input_datum_additional}
In addition to \cref{ass:input_datum_conservation_laws}, we assume
\begin{multicols}{2}
\begin{itemize}
    \item \(V'\leqq 0\)
    \item \(\supp(\gamma)\subset\R_{\geq0}\) 
    \item \(\gamma\) monotonically decreasing on \(\R_{>0}\)
    \item \(q_{0}\in\sL^{\infty}(\R;\R_{\geq0})\), i.e.,\ nonnegative.
\end{itemize}
\end{multicols}
\end{ass}
The assumption that \(V\) is monotonically decreasing is very common in traffic flow (compare with the classical LWR model in traffic \cite{lwr_1,lwr_2,greenshields}) as it states that the velocity must decrease with higher density. The assumption that \(q_{0}\) is nonnegative and essentially bounded is inspired by interpreting solutions as traffic densities on roads that have limited capacity.

Finally, the assumptions on the kernel \(\gamma\) ensure that density further ahead does not impact the nonlocal term as much as density immediately ahead. Traffic density behind generally does not matter. However, new models are emerging that incorporate nudging (looking behind) (see for example \cite{karafyllis2020analysis}). General maximum principles cannot be expected for looking behind nonlocal terms.

\begin{theo}[A maximum principle/uniform bounds]\label{theo:maximum_principle}
Let \cref{ass:input_datum_additional} hold, and consider the following two cases for the weak solution (in the sense of \cref{defi:weak_solution}) of the discontinuous nonlocal conservation law in \cref{defi:disc_conservation_law}. 
\begin{description}
\item[\textbf{Monotonically increasing $\boldsymbol v$}:]
The weak solution exists for each \(T\in\R_{>0}\) and satisfies the classical maximum principle
\begin{equation}
    0\leq q(t,x)\leq \|q_{0}\|_{\sL^{\infty}(\R)}\quad \forall (t,x)\in (0,T)\times\R \text{ a.e.}
\end{equation}
\item[\textbf{Initial datum $\sL^{1}$ integrable and $\gamma$ more regular}:] In detail, assuming
\[q_{0}\in \sL^{1}(\R;\R_{\geq0})\cap \sL^{\infty}(\R;\R_{\geq 0})\ \wedge\ \gamma \in \sW^{1,\infty}(\R_{> 0};\R_{\geq 0}) \cap \sL^1(\R_{>0} ; \R_{\geq 0})\ \wedge\  \gamma' \leqq 0,
\] the weak solution exists for every \(T\in\R_{>0}\) with the following bounds for any $\delta \in\R_{>0}$:
\begin{itemize}
    \item if \(\esssup_{s\in X(q_{0},\gamma)} V'(s)<0\) it holds \(\forall (t,x)\in \OT \text{ a.e.}\)
    \[
    0\leq q(t,x)\leq \max\Big\{\tfrac{\|v\cdot q_{0}\|_{\sL^{\infty}(\R)}}{\uv}\sk  \tfrac{\|v\|_{\sL^{\infty}(\R)}}{\uv}\tfrac{\|V'\|_{\sL^{\infty}(X(q_{0},\gamma)+(-\delta,\delta))}}{-\esssup\limits_{s\in X(q_{0},\gamma)+(-\delta,\delta)}V'(s)}\|q_{0}\|_{\sL^{1}(\R)}\Big\}
    \]
    \item if \(\esssup_{s\in X(q_{0},\gamma)+(-\delta,\delta)} V'(s)=0\) it holds \(\forall (t,x)\in\OT\text{ a.e.}\)
    \[
    0\leq q(t,x)\leq \tfrac{\| v q_0\|_{\sL^{\infty}(\R)}}{\uv}\exp\Big({t\|v\|_{\sL^{\infty}(\R)}\|V'\|_{\sL^{\infty}(X(q_{0},\gamma)+(-\delta,\delta))}\gamma(0)\|q_{0}\|_{\sL^{1}(\R)}}\Big),
    \]
\end{itemize}
with 
\begin{equation}
X(q_{0},\gamma)\: \big(0,\|q_{0}\|_{\sL^{1}(\R)}\|\gamma\|_{\sL^{1}(\R)}\big)\subset\R.
\end{equation}
\end{description}
\end{theo}
\begin{proof}
The nonnegativity of the solution immediately follows from the representation of the solution in \cref{theo:existence_uniqueness_nonlocal_conservation_law}, specifically in \cref{eq:solution_nonlocal}. Thus we only need to focus on the upper bounds in all the presented cases. To this end, approximate \(q_0,v,V\) by a smooth \(q_{0,\eps},v_{\eps},V_{\eps}\) according to \cref{theo:weak_stability}.
As the solutions for \(\eps\in\R_{>0}\) are smooth (and thus, classical solutions), we can work on the classical form and have for \((t,x)\in\OT\)
\begin{align*}
&\partial_t q_{\eps}(t,x)\\
&=-\partial_{x}\Big(v_{\eps}(x)V_{\eps}\big(\cW[q_{\eps}](t,x)\big)q_{\eps}(t,x)\Big)\\
&=-v_{\eps}'(x)V_{\eps}\big(\cW[q_{\eps}](t,x)\big)q_{\eps}(t,x)-v_{\eps}(x)V_{\eps}\big(\cW[q_{\eps}](t,x)\big)\partial_{x}q_{\eps}(t,x)\\ &\quad -v_{\eps}(x)V_{\eps}'\big(\cW[q_{\eps}](t,x)\big)\partial_{x}\cW[q_{\eps}](t,x) q_{\eps}(t,x).
\intertext{For $x \in \R$ s.t \(q_{\eps}(t,x)\) this is maximal (and thus \(\partial_{x}q_{\eps}(t,x)=0\))}
&= \Big(-v_{\eps}'(x)V_{\eps}\big(\cW[q_{\eps}](t,x)\big)-v_{\eps}(x)V_{\eps}'\big(\cW[q_{\eps}](t,x)\big)\partial_{x}\cW[q_{\eps}](t,x)\Big) q_{\eps}(t,x).
\end{align*}
Consider now the \textbf{first case}, i.e., assume $v'\geqq 0$ and by construction $v'_\epsilon \geqq 0$, as well as \(V_{\eps}\geqq 0\). Then, we obtain with the previous computation at the \(x\in\R\) where \(q_{\eps}(t,x)\) is maximal
\begin{align*}
\partial_t q_{\eps}(t,x)\leq -v_{\eps}(x)V'\big(\cW[q_{\eps}](t,x)\big)\partial_{x}\cW[q_{\eps}](t,x).
\end{align*}
However, thanks to \(V'_{\eps}\leqq 0\) and \(\partial_{x}\cW[q_{\eps}](t,x)\leqq 0\) for the \((t,x)\in\OT\) where \(q_{\eps}(t,x)\) is maximal, the last term is nonpositive. This implies that the maxima can only decrease, i.e.,
\[
q_{\eps}(t,x)\leq \|q_{0,\eps}\|_{\sL^{\infty}(\R)}\leq \|q_{0}\|_{\sL^{\infty}(\R)}\quad \forall (t,x)\in\OT.
\]
According to \cref{theo:weak_stability}, we have \[
\forall g\in \sC_{\text{c}}(\R) : \lim_{\eps\rightarrow 0}\max_{t\in[0,T]}\Big\vert \int_{\R}\big( q(t,x)-q_{\eps}(t,x)\big) g(x)\dd x\Big\vert=0
\] 
so that this upper bound carries over from \(q_{\eps}\) to \(q\), completing the first case. 
For the \textbf{second case} by some ``scaling with \(v_\eps\)'' and defining the new variable \(\rho_\eps \equivd v_\eps q_\eps\) we obtain 
\begin{align*}
    \partial_t \rho_\eps(t,x) &=  -v_\eps(x)\partial_{x}\Big(V_\eps\big(\cW[q_\eps](t,x)\big)\rho_\eps(t,x)\Big) \\
    &= -v_\eps(x)V_\eps'\big(\cW[q_\eps](t,x)\big)\partial_{x}\cW[q_\eps](t,x)\rho_\eps(t,x)\\
    &\quad -v_\eps(x)V_\eps\big(\cW[q_\eps](t,x)\big)\partial_{x}\rho_{\eps}(t,x).
    \end{align*}
Taking any $x\in \R$ s.t.\ $\rho_\eps(t,x) = \|\rho_\eps(t,\cdot)\|_{\sL^\infty(\R)}$, we thus have $\partial_x \rho_\eps(t,x) = 0$ and by the previous computations
 \begin{align*}
 \partial_t \rho_\eps(t,x)= -v_\eps(x)V'_\eps\big(\cW[q_\eps](t,x)\big)\partial_{x}\cW[q_\eps](t,x)\rho_\eps(t,x).
    \end{align*}
As we have by assumption $\gamma \in \sW^{1,\infty}(\R_{\geq 0}\sk \R_{\geq 0})$ with  $\gamma' \leqq 0$, this yields for \(x\in\R\) where \(\rho_{\eps}(t,x)=\|\rho_{\eps}(t,\cdot)\|_{\sL^{\infty}(\R)}\),
\begin{align}
    \partial_t \rho_\eps(t,x)&= v_\eps(x)V_\eps'\big(\cW[q_\eps](t,x)\big)\Big(\gamma(0)q_\eps(t,x) + \!\!\int_{\R_{> 0}} \!\!\!\!\gamma'(y-x) q_\eps(t,y) \dd y\Big) \rho_\eps(t,x) \notag \\
    &\leq \essinf_{s\in \big(-\eps,\|q_0\|_{\sL^1(\R)}\|\gamma\|_{\sL^1(\R)}+\eps\big) } V'_\eps\big(s\big)\gamma(0)  \rho_\eps^2(t,x) \label{eq:YAEWWNC}\\
    &\quad + \|v\|_{\sL^\infty(\R)}\|V'\|_{\sL^\infty\left(\left(-\eps,\|q_0\|_{\sL^1(\R)}\|\gamma\|_{\sL^1(\R)}+\eps\right)\right)}
\gamma(0) \|q_0\|_{\sL^1(\R)} \rho_\eps(t,x). \notag
\end{align}
The latter estimate holds due to \(\vert \gamma \vert_{\sTV(\R_{>0})} = \gamma(0)\), as \(\gamma'\leqq 0\) and \(\gamma \in \sBV(\R_{>0})\). We have also used Young's convolution inequality \cite[Theorem 4.15 (Young)]{brezis} several times to estimate $\mathcal W$ and $\partial_x \mathcal W$. In detail, it holds that
\begin{align*}
    0\leq\int_{\R_{>0}} \!\!\!\!\gamma(x-y) q_\eps(t,y) \dd y &\leq \|\gamma\|_{\sL^1(\R)}\|q_\eps\|_{\sL^1(\R)} = \|\gamma\|_{\sL^1(\R)}\|q_{0,\eps}\|_{\sL^1(\R)} \leq \|\gamma\|_{\sL^1(\R)}\|q_{0}\|_{\sL^1(\R)}\\
    \bigg\vert\!\int_{\R_{>0}} \!\!\!\!\gamma'(x-y) q_\eps(t,y) \dd y\bigg\vert &\leq \vert\gamma\vert_{\sTV(\R_{>0})}\|q_\eps\|_{\sL^1(\R)} \leq \gamma(0)\|q_{0}\|_{\sL^1(\R)},
\end{align*}
by construction of $q_{0,\eps}$, the conservation of mass of \(q\), and the fact that $\gamma$ is monotonically decreasing on $\R_{>0}$.
Altogether, we obtain for \((t,x)\in\OT\) a.e.\
\begin{align*}
 \rho_\eps(t,x) \leq \max\bigg\{ \|v\cdot q_0\|_{\sL^\infty(\R)} \ , \   \tfrac{\|v\|_{\sL^\infty(\R)}\|V'\|_{\sL^\infty\left(X(q_0,\gamma)+(-\eps,\eps)\right)}
 \|q_0\|_{\sL^1(\R)}}{-\essinf_{s\in \left(X(q_0,\gamma)+(-\eps,\eps)\right) } V_{\eps}'(s)  }\bigg\}.
\end{align*}
With the identical argument as before, once more using \cref{theo:weak_stability}, we obtain the stated bounds for \(\eps\rightarrow 0\) when later recalling that \(\rho\equiv v\cdot q\) and \(v\geqq \uv\).

The \textbf{third case} follows immediately when reconnecting to \cref{eq:YAEWWNC} and noticing that

\begin{align*}
    \partial_t \rho_\eps(t,x) &\leq \|v\|_{\sL^\infty(\R)}\|V'\|_{\sL^\infty\left(\left(-\eps,\|q_0\|_{\sL^1(\R)}\|\gamma\|_{\sL^1(\R)}+\eps\right)\right)}
\gamma(0) \|q_0\|_{\sL^1(\R)} \rho_\eps(t,x), 
\end{align*}
which leads to at most exponential growth of \(\|\rho_\eps\|_{\sL^\infty((\R))}\). 
\end{proof}
\begin{cor}[Compatible initial datum] In contradiction to the concluding remarks in \cite{chiarello2021existence}, we can state that for any given discontinuity \(v\), 
\(V'\in\sL^{\infty}(\R)\) with \(\esssup_{x\in\R}V'(x)<0\) and desired upper bound \(C\in \R_{>0}\), there exists \(q_{0} \not\equiv 0\) such that the corresponding solution $q$ to the discontinuous nonlocal conservation law as defined in \cref{defi:disc_conservation_law} satisfies \[q \leqq C \ \text{ on } \OT\] for each \(T\in\R_{>0}\).
This is still possible for the case \(\esssup_{x\in\R}V'(x)=0\). However, the time horizon must be fixed. 
\end{cor}
We illustrate the discontinuous nonlocal conservation law by means of the following
\begin{example}[Some numerical illustrations]
In order to visualize the effect of a discontinuity in space (demonstrated via \(v\)), we consider the following modelling archetypes:
\[
q_{0}\:\tfrac{1}{2}\chi_{[-0.5,-0.1]}(x), \ V(\cdot)\equiv 1-\cdot,\ \gamma\equiv 10\chi_{[0,0.1]},\ v\in\big\{1,\ 1+\chi_{\R_{>0}},\ 1-\tfrac{1}{2}\chi_{\R_{>0}}\big\},
\]
which are illustrated in \cref{fig:nonlocal_discontinuous_3d}. As can be seen, for \(v\equiv 1\) and \(v\equiv 1+\chi_{\R_{>0}}\), i.e.\  the monotonically increasing cases, the first statement of \cref{fig:nonlocal_discontinuous_3d} applies, and indeed the proposed maximum principle holds. Moreover, a jump downwards is evident at position \(x=0\) in density \(q\), when the velocity jumps from \(1\) to \(2\) in the second case. This is in line with intuition that an increased speed reduces the density accordingly, in equations roughly
\begin{align*}
\forall t\in[0,T]:\ \lim_{x\nearrow 0} q(t,x)v(x)V(\cW[q](t,x))&= \lim_{x\searrow 0} q(t,x)v(x)V(\cW[q](t,x))\\
\overset{\text{assume }V\neq 0}{\Longleftrightarrow} \ \qquad 
\lim_{x\nearrow 0} q(t,x)
&= 2\lim_{x\searrow 0} q(t,x).
\end{align*}
In the third case, the velocity is halved at \(x=0\) and the density is doubled. This is -- for specific times \(t\in\{0,0.5,1\}\) -- also illustrated in the bottom row of \cref{fig:nonlocal_discontinuous_3d}.
\begin{figure}
    \centering
    {\includegraphics[scale=0.75,clip,trim=2 0 12 0]{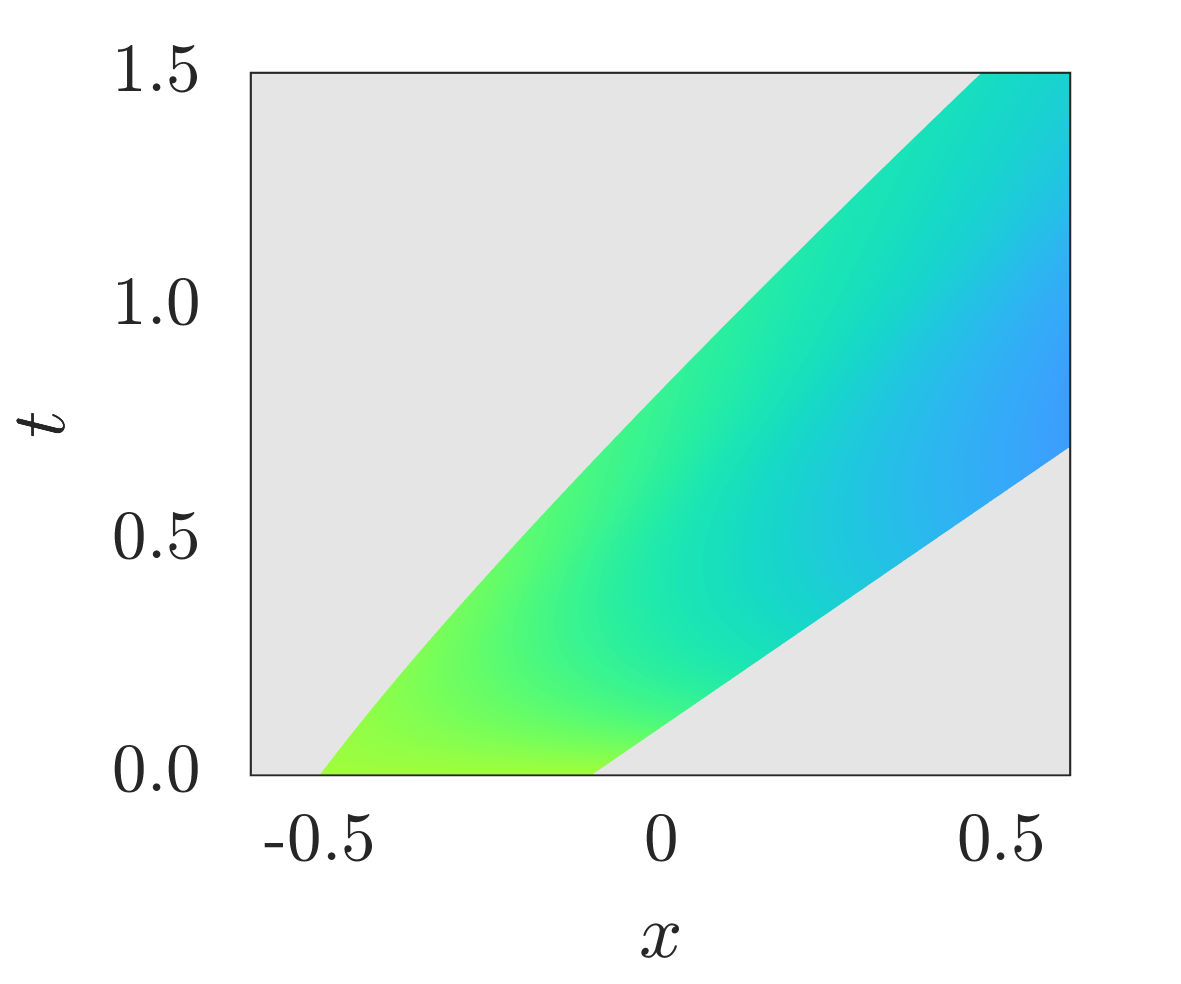}}
     {\includegraphics[scale=0.75,clip,trim=35 0 12 0]{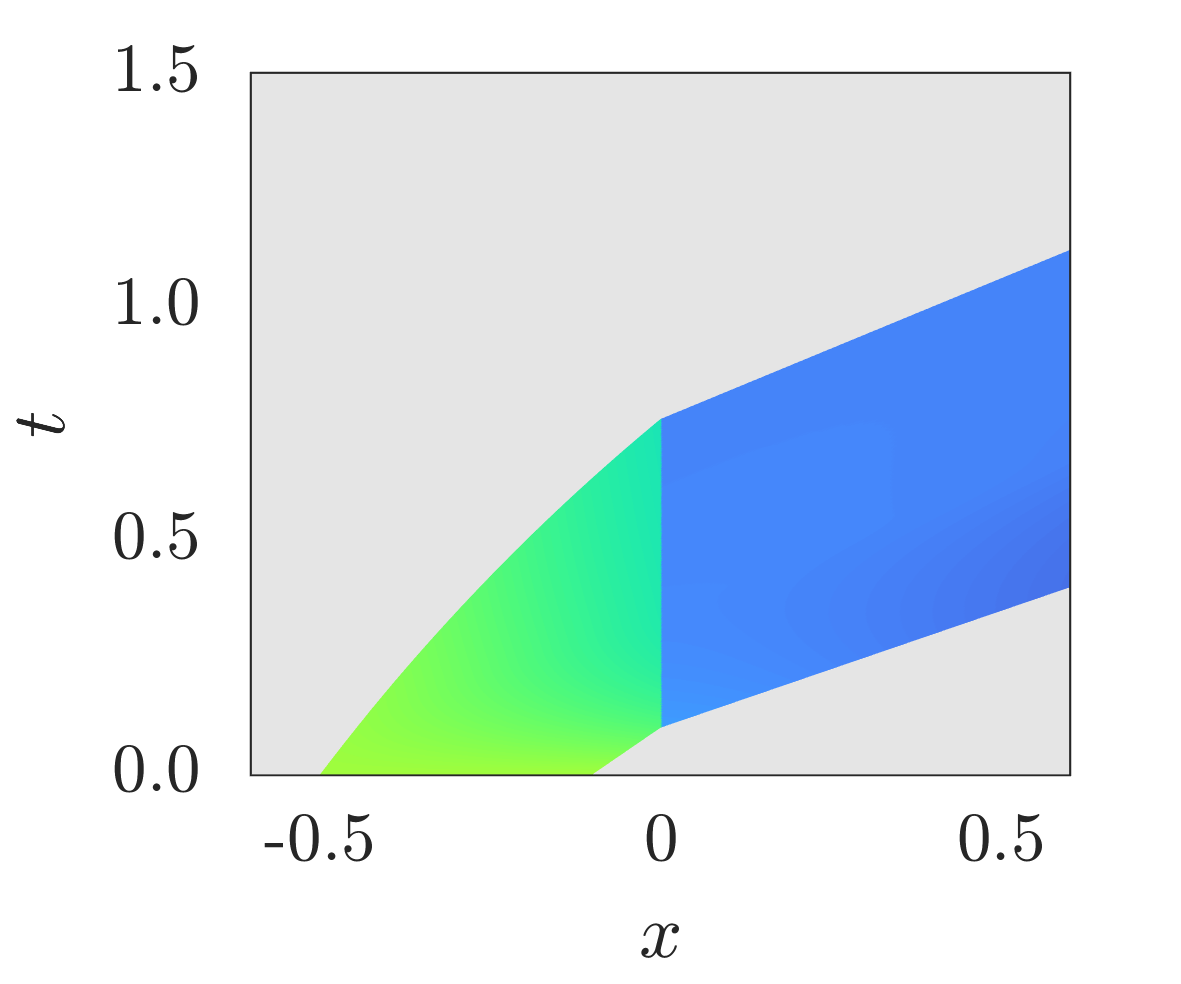}}
     {\includegraphics[scale=0.75,clip,trim=35 0 12 0]{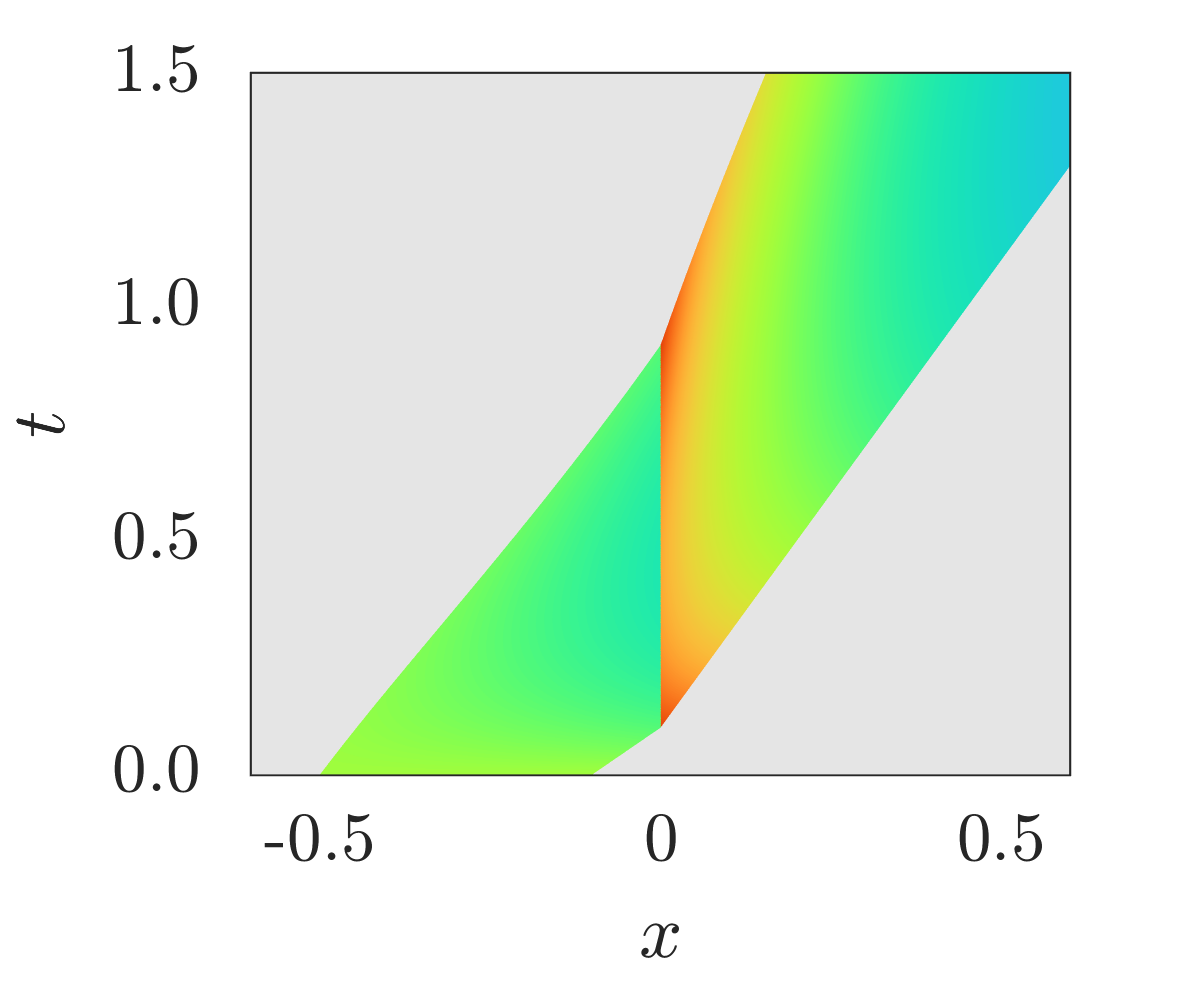}}
     
     \includegraphics[scale=0.75,clip,trim=2 0 12 0]{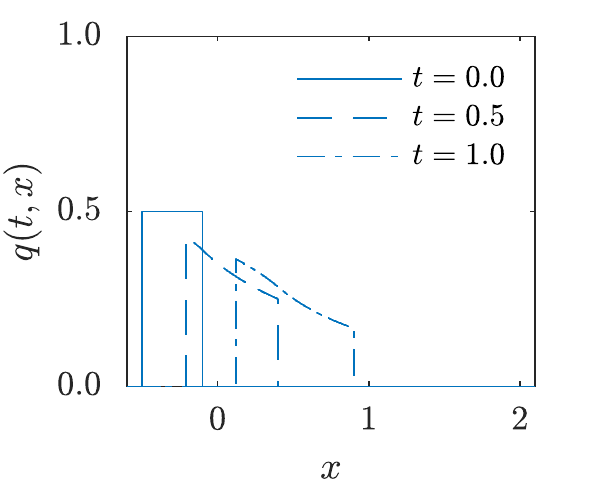}
     \includegraphics[scale=0.75,clip,trim=35 0 12 0]{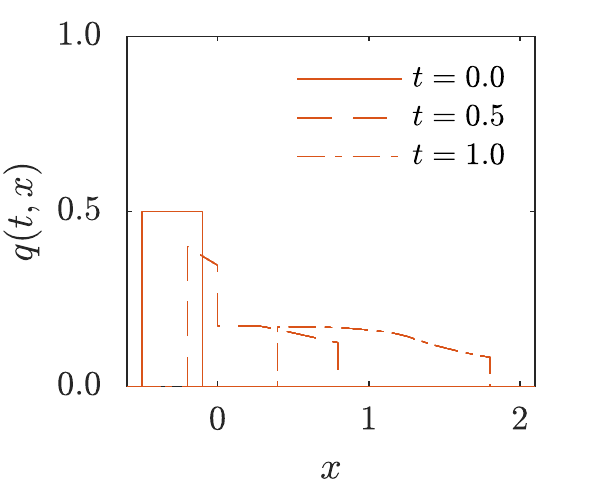}
     \includegraphics[scale=0.75,clip,trim=35 0 12 0]{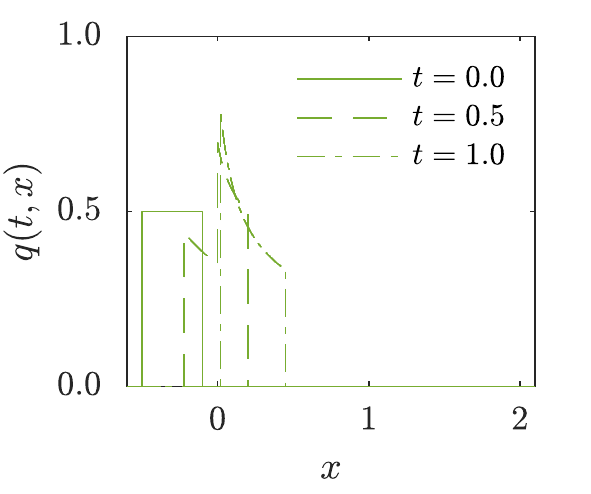}
     
    \caption{Evolution of the solution $q$ in space-time with discontinuities $v \equiv 1$ (\textbf{top left}), $v \equiv 1 + \chi_{\R_{>0}}$ (\textbf{top middle}) and $v \equiv 1 - \tfrac{1}{2} \chi_{\R_{>0}}$ (\textbf{top right}). The solutions at time $t=0$ (\textbf{solid}), $t = 0.5$ (\textbf{dashed}) and $t=1$ (\textbf{dash-dotted}) are shown in the \textbf{lower row}. }
    \label{fig:nonlocal_discontinuous_3d}
\end{figure}

\begin{figure}
    \centering
    \includegraphics[scale=0.75]{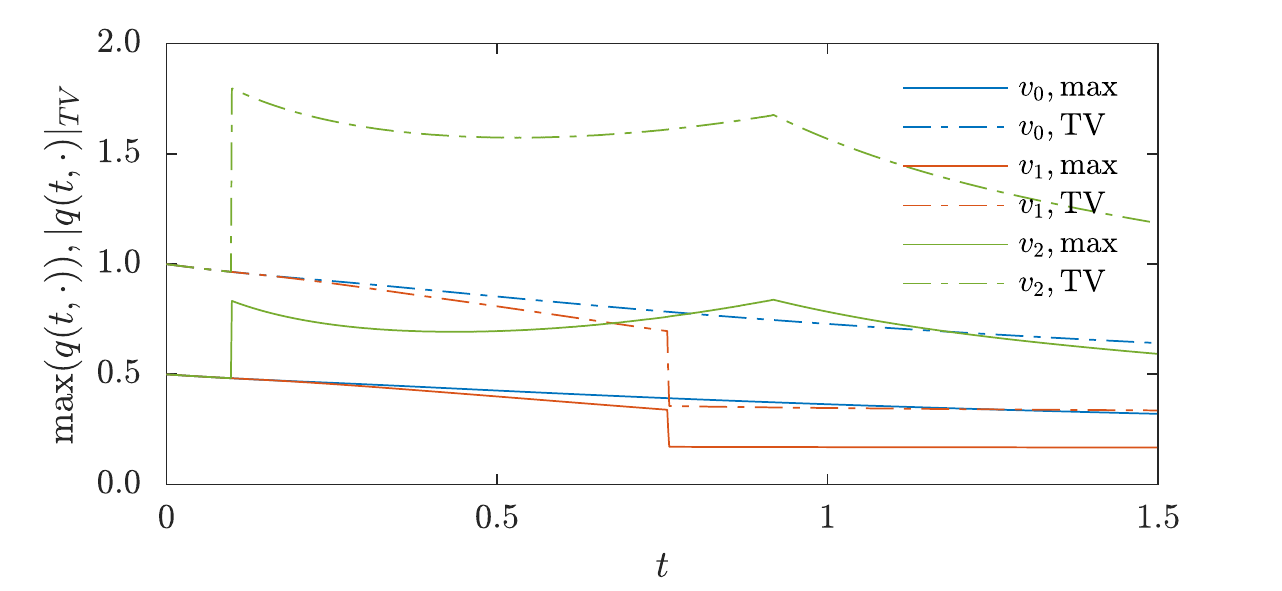}
    \caption{The evolution of the maximum (\textbf{solid}) of the solution as well as its total variation (\textbf{dash-dotted}) are visualized. \textcolor{blue}{Blue} represents the case with \(v\equiv v_{0}\equiv 1\), \textcolor{red!50!black!75!yellow}{orange} the case with \(v\equiv v_{1}\equiv 1+\chi_{\R_{>0}}\) and, finally, \textcolor{green!50!black}{green} the case with \(v\equiv v_{2}\equiv 1-\tfrac{1}{2}\chi_{\R_{>0}}\).}
    \label{fig:nonlocal_discontinuous_2d}
\end{figure}
In \cref{fig:nonlocal_discontinuous_2d}, the evolution of the maximum of the solution, in equations \(\|q(t,\cdot)\|_{\sL^{\infty}(\R)}\), is illustrated. This reflects our previous remarks.
The total variation for the different cases is also shown and as can be seen, it changes significantly when the discontinuity comes into play. Clearly, an upper bound will depend on the total variation of \(q_{0}\) as well as \(v\).
\end{example}
\section{Conclusions and open problems}\label{sec:conclusions}
In this contribution, we have studied nonlocal conservation laws in \(\sC\big([0,T];\sL^{1}_{\loc}(\R)\big)\cap\sL^{\infty}\big((0,T);\sL^{\infty}(\R)\big)\) with general multiplicative discontinuities (\(\sL^{\infty}\)-type) in space. By employing the method of characteristics and a reformulation as a fixed-point problem, we could instead consider specific discontinuous ODEs, which we studied for existence, uniqueness and stability. The results obtained were then applied to the discontinuous nonlocal conservation law to prove existence and uniqueness of weak solutions on a small time horizon. These results were supplemented by several ``maximum principles'' guaranteeing the semi-global existence of solutions. We have thus generalized the existing theory on (purely) nonlocal conservation laws to include discontinuities in space, and have proven that Entropy conditions are -- once more -- obsolete (compare with \cite{pflug}) although still used in literature \cite{chiarello2021existence,chiarello2021nonlocal}. 
The established theory sets the stage for several future directions:
\textbf{1)} similar to \cite{coclite2020general}, consideration of the convergence to the local discontinuous conservation law when we let the convolution kernel in the nonlocal part of the velocity converge to a Dirac distribution,
\textbf{2)} the bounded domain case similar to \cite{keimer1},  
\textbf{3)} measure-valued solutions similar to \cite{crippa2013existence}, assuming that the kernel is in \(\sW^{1,\infty}(\R)\),
\textbf{4)} discontinuous (in space) multi-dimensional nonlocal conservation laws.

\section*{Acknowledgement}
Funded by the Deutsche Forschungsgemeinschaft (DFG, German Research Foundation) – Project-ID 416229255 – SFB 1411.

We would like to thank Deborah Bennett for proofreading this manuscript carefully.
\bibstyle{sn-mathphys}
\bibliography{biblio.bib}

\end{document}